\documentclass[11pt]{amsart}
\usepackage{amsmath, enumerate}
\usepackage{amsfonts, amssymb, amsthm}
\usepackage{mathtools}
\usepackage{geometry}
\usepackage{a4wide}
\usepackage{listings}
\usepackage{xcolor}
\usepackage{bbm}
\usepackage{hyperref}

\definecolor{codegreen}{rgb}{0,0.6,0}
\definecolor{codegray}{rgb}{0.5,0.5,0.5}
\definecolor{codepurple}{rgb}{0.58,0,0.82}
\definecolor{backcolour}{rgb}{0.95,0.95,0.92}

\lstdefinestyle{mystyle}{
    backgroundcolor=\color{backcolour},   
    commentstyle=\color{codegreen},
    keywordstyle=\color{magenta},
    numberstyle=\tiny\color{codegray},
    stringstyle=\color{codepurple},
    basicstyle=\ttfamily\footnotesize,
    breakatwhitespace=false,         
    breaklines=true,                 
    captionpos=b,                    
    keepspaces=true,                 
    numbers=left,                    
    numbersep=5pt, 
    frame = single,                
    showspaces=false,                
    showstringspaces=false,
    showtabs=false,                  
    tabsize=1
}

\let\originalleft\left
\let\originalright\right
\renewcommand{\left}{\mathopen{}\mathclose\bgroup\originalleft}
\renewcommand{\right}{\aftergroup\egroup\originalright}

\numberwithin{equation}{section}
\theoremstyle{plain}

\newtheorem{theorem}{Theorem}[section]
\newtheorem{lemma}[theorem]{Lemma}
\newtheorem{proposition}[theorem]{Proposition}
\newtheorem{corollary}[theorem]{Corollary}

\theoremstyle{definition}
\newtheorem{definition}[theorem]{Definition}

\newtheorem{example}[theorem]{Example}

\theoremstyle{remark}
\newtheorem{remark}[theorem]{Remark}

\newcommand{\vol}{\mathrm{vol}}

\newcommand{\res}{\mathrm{res}}

\newcommand{\Gr}{\mathrm{Gr}}
\newcommand{\R}{\mathbb{R}}
\newcommand{\N}{\mathbb{N}}
\newcommand{\Z}{\mathbb{Z}}
\newcommand{\Q}{\mathbb{Q}}
\newcommand{\C}{\mathbb{C}}
\newcommand{\A}{\mathbb{A}}
\newcommand{\e}{\mathfrak{e}}
\newcommand{\reg}{\mathrm{reg}}
\newcommand{\hol}{\mathrm{hol}}
\newcommand{\cusp}{\mathrm{cusp}}
\renewcommand{\H}{\mathbb{H}}
\newcommand{\Mp}{{\mathrm{Mp}}}
\newcommand{\SL}{{\mathrm{SL}}}
\newcommand{\PSL}{{\mathrm{PSL}}}
\newcommand{\GL}{{\mathrm{GL}}}
\newcommand{\SO}{{\mathrm{SO}}}

\newcommand{\tr}{\mathrm{tr}}
\newcommand{\sgn}{\mathrm{sgn}}

\newcommand{\CT}{\mathrm{CT}}

\newcommand{\gen}{\mathrm{gen}}
\renewcommand{\O}{\mathrm{O}}
\renewcommand{\cosh}{\mathrm{cosh}}

\renewcommand{\det}{\mathrm{det}}
\renewcommand{\Re}{\mathrm{Re}}
\renewcommand{\Im}{\mathrm{Im}}

\newcommand{\1}{\mathbbm{1}}

\begin{document}
\title{Special values of Green's functions on hyperbolic $3$-space}

\author{S.~Herrero}
\address{Universidad de Santiago de Chile, Dept.~de Matem\'atica y Ciencia de la Computaci\'on, Av.~Libertador Bernardo O'Higgins 3363, Santiago, Chile, and ETH, Mathematics Dept., CH-8092, Z\"urich, Switzerland}
\email{sebastian.herrero.m@gmail.com}
\thanks{S.~Herrero's research is supported by ANID/CONICYT FONDECYT Iniciaci\'on grant 11220567 and by SNF grant 200021-185014}

\author{\"O.~Imamo\=glu}
\address{ETH, Mathematics Dept., CH-8092, Z\"urich, Switzerland}
\email{ozlem@math.ethz.ch}
\thanks{\"O.~Imamo\=glu's research  is supported by SNF grant 200021-185014}

\author{A.-M.~von Pippich}
\address{
University of Konstanz, Dept. of Mathematics and Statistics, Universit\"atsstra{\ss}e 10, 78464 Konstanz, Germany}
\email{anna.pippich@uni-konstanz.de}
\thanks{A.-M.~von Pippich's research is supported by the LOEWE research unit \emph{Uniformized Structures in Arithmetic and Geometry}}

\author{M.~Schwagenscheidt}
\address{ETH, Mathematics Dept., CH-8092, Z\"urich, Switzerland}
\email{mschwagen@ethz.ch}
\thanks{M.~Schwagenscheidt's research is supported by SNF grant PZ00P2{\_}202210}

\begin{abstract}
  	Gross, Kohnen and Zagier proved an averaged version of the algebraicity conjecture for special values of higher Green's functions on modular curves. In this work, we study an analogous problem for special values of Green's functions on hyperbolic $3$-space. We prove that their averages can be computed in terms of logarithms of primes and logarithms of units in real quadratic fields. Moreover, we study twisted averages of special values of Green's functions, which yield algebraic numbers instead of logarithms. 
 \end{abstract}



\maketitle


\tableofcontents

\section{Introduction}

In a seminal paper, Gross and Zagier \cite{grosszagier} related the central values of derivatives of $L$-functions of elliptic curves to heights of Heegner points on modular curves. Moreover, they found a striking factorization of the norms of differences of singular moduli \cite{grosszagier-SM}. The proof uses the fact that the function $\log|j(\tau_1)-j(\tau_2)|$ is essentially given by the constant term at $s = 1$ of the automorphic Green's function 
\[
G^{\SL_2(\Z)}_s(\tau_1,\tau_2) = \sum_{\gamma \in \SL_2(\Z)}Q_{s-1}\left(1 + \frac{|\tau_1 - \gamma\tau_2|^2}{2\Im(\tau_1)\Im(\gamma \tau_2)}\right)
\]
on the modular curve $X(1) = \SL_2(\Z)\backslash \H$. Here $Q_{s-1}(x)$ denotes the Legendre function of the second kind and $\Im(\tau)$ denotes the imaginary part of the point $\tau$ in the complex upper-half plane $\H$. The Green's function converges absolutely for $s\in \C$ with $\Re(s) > 1$ and has a simple pole at $s = 1$. Moreover, it is an eigenfunction of the hyperbolic Laplacian in both variables, and has a logarithmic singularity along the diagonal in $X(1) \times X(1)$.

 At the end of \cite{grosszagier}, Gross and Zagier made a deep conjecture about the algebraicity properties of the values of the level $N\geq 1$ Green's function $G_s^{\Gamma_0(N)}(\tau_1,\tau_2)$ at positive integer values $s = k \geq 2$ and CM points $\tau_1,\tau_2$. They predicted that (under some technical conditions) the values at CM points $\tau_1,\tau_2$ of the ``higher Green's function'' $G_{k}^{\Gamma_0(N)}(\tau_1,\tau_2)$ are essentially given by logarithms of absolute values of algebraic numbers. There has been a lot of interesting work on this conjecture over the last years (see \cite{bruinierehlenyang,bruinierliyang,li1, mellit,viazovska1,viazovska2}). A major step in this direction  was taken by Li \cite{li2} who proved the conjecture in the case of level 1 and when one of the discriminants is fundamental. The general case   has recently been solved   by Bruinier, Li, and Yang \cite{bruinierliyang}.

In the present work we study an analogous problem for Green's functions on the hyperbolic $3$-space $\H^3$. A major obstacle in this case is the fact that $\H^3$ does not have a complex structure. In particular, the theory of complex multiplication, which played an important role in the proof of the Gross--Zagier algebraicity conjecture over modular curves, is not available on $\H^3$. Hence, it is not clear whether the special values of Green's functions on $\H^3$ still have good algebraic properties. In fact, our results suggest that the analogue of the Gross--Zagier conjecture for the individual special values of Green's functions on $\H^3$ might not be true; see Example~\ref{example 2 intro} below.

In order to obtain a convenient algebraicity result, we study an averaged version of the problem. Gross, Kohnen and Zagier \cite{grosskohnenzagier} proved an averaged version of the algebraicity conjecture over modular curves by summing $\tau_1$ and $\tau_2$ over all classes of CM points of fixed discriminants $d_1$ and $d_2$. They showed that these ``double traces'' are essentially given by logarithms of rational numbers. Inspired by this result, we consider the double traces of Green's functions on hyperbolic $3$-space. We show that they are given by algebraic linear combinations of logarithms of primes and logarithms of units in real quadratic fields. For the proof of our results we follow ideas of Bruinier, Ehlen, and Yang \cite{bruinierehlenyang}, who proved a partially averaged version of the algebraicity conjecture over modular curves: they fix $\tau_1$ and sum $\tau_2$ over all classes of discriminant $d_2$. It will become clear during our proof why we cannot fix one of the variables in the hyperbolic $3$-space case, but instead need to take the double trace.

Let us describe our results in some more detail. We use the setup of \cite{elstrodt}. Let $\Q(\sqrt{D})$ be an imaginary quadratic field of discriminant $D < 0$, and let $\mathcal{O}_D$ be its ring of integers. The group 
\[
\Gamma = \PSL_2(\mathcal{O}_D)
\]
acts on the hyperbolic $3$-space
\[
\H^3 = \{P = z + r j \, : \, z \in \C, r \in \R^+\}
\] 
(viewed as a subset of the quaternions $\R[i,j,k]$) by fractional linear transformations, and this action preserves the hyperbolic distance $d(P_1,P_2)$. For $P_1,P_2 \in \H^3$ which are not in the same $\Gamma$-orbit, and $s \in \C$ with $\Re(s) > 1$, the \textit{automorphic Green's function} for $\Gamma$ is defined by\footnote{We use a different normalization than \cite{elstrodt} and \cite{herreroimamogluvonpippichtoth} to obtain nicer algebraicity results.}
\begin{equation}\label{eq Green's function}
G_s(P_1,P_2) = \pi\sum_{\gamma \in \Gamma}\varphi_s\big(\cosh(d(P_1,\gamma P_2))\big),
\end{equation}
with the function
\begin{equation}\label{eq:varphi_s}
\varphi_s(t) = \left(t+\sqrt{t^2-1}\right)^{-s} (t^2-1)^{-1/2}.
\end{equation}
The Green's function converges absolutely for $\Re(s) > 1$, is $\Gamma$-invariant in both variables, and symmetric in $P_1,P_2$. Moreover, it satisfies 
\[
(\Delta_{P_1} - (1-s^2)) G_s(P_1,P_2) = 0,
\]
where $\Delta_{P_1}$ is the usual invariant Laplacian on $\H^3$, taken with respect to the variable $P_1$. The Green's function has meromorphic continuation to $s\in \C$ with $s = 1$ a simple pole, and singularities precisely at the points $P_2$ in the $\Gamma$-orbit of $P_1$.

We want to investigate the algebraic properties of the Green's function $G_s(P_1,P_2)$ at special points $P_1,P_2 \in \H^3$, and positive integer values $s \geq 2$. The special points we consider are defined as follows. For a number $m$ in $\frac{1}{|D|}\N$ (where $\N=\{1,2,\ldots\}$) we let
\begin{equation}\label{eq:L_m^+}
L_{m}^+ = \left\{X = \begin{pmatrix}a & b \\ \overline{b} & c \end{pmatrix} \, :\, a,c \in \N,\, b \in \mathfrak{d}_D^{-1}, \, \det(X)=m \right\}
\end{equation}
be the set of positive definite integral binary hermitian forms of determinant $m$ over $\Q(\sqrt{D})$. Here $\mathfrak{d}_D^{-1}$ denotes the inverse different of $\Q(\sqrt{D})$. The group $\Gamma$ acts on $L_{m}^+$ by $\gamma.X = \gamma X \overline{\gamma}^t$, and $\Gamma \backslash L_m^+$ has finitely many classes. To a form $X \in L_{m}^+$ we associate the \emph{special point}
\[
P_X = \frac{b}{c} + \frac{\sqrt{m}}{c}j \in \H^3.
\]
We may view it as an analogue of a CM point on $\H^3$. Moreover, we define the $m$-th \emph{trace} of the Green's function by
\[
\tr_{m}(G_{s}(\,\cdot\,,P)) = \sum_{X \in \Gamma \backslash L_{m}^+}\frac{1}{|\Gamma_X|}G_s(P_X,P),
\]
where $\Gamma_X$ denotes the stabilizer of $X$ in $\Gamma$. Here we need to be careful not to evaluate the Green's function at a singularity, so $P$ should not be in the same $\Gamma$-orbit as one of the special points $P_X$. We remark that, unlike the case of binary quadratic forms, the order $|\Gamma_X|$ really depends on the individual hermitian form $X$, not just on its determinant $m$. If we also take the trace in the second variable, we obtain the \emph{double trace}
\begin{equation}\label{def double trace}
 \tr_{m}\tr_{m'}(G_{s}) = \sum_{\substack{X \in \Gamma \backslash L_{m}^+ \\ Y \in \Gamma \backslash L_{m'}^+}}\frac{1}{|\Gamma_X|}\frac{1}{|\Gamma_Y|}G_s(P_X,P_Y).   
\end{equation}
In the following we will tacitly assume that $m$ and $m'$ are chosen in such a way that we do not evaluate the Green's function at a singularity.

From now on we will work with \emph{prime discriminants}, so either $D = -4,D=-8,$ or $D = -\ell$ for an odd prime $\ell \equiv 3 \pmod 4$. Moreover, for simplicity we will assume in the introduction that a certain space of cusp forms is trivial. Namely, for an odd positive integer $k$ we consider the space $S_{k}^+(\Gamma_0(|D|),\chi_{D})$ of cusp forms $f = \sum_{n > 0}a_f(n)q^n$ of weight $k$ for $\Gamma_0(|D|)$ and character $\chi_{D} = \left( \frac{D}{\cdot}\right)$ satisfying the ``plus space'' condition $a_f(n) = 0$ if $\chi_{D}(n) = -1$. Assuming $S_{k}^+(\Gamma_0(|D|),\chi_{D})=\{0\}$ we have the following algebraicity result. In the body of the paper we lift this restriction by taking suitable linear combinations of double traces of the Green's function; see Corollary~\ref{corollary that implies main theorem}. 

\begin{theorem}\label{main theorem}
	Let $D < 0$ be a prime discriminant and let $n \geq 1$ be a natural number such that $S_{1+2n}^+(\Gamma_0(|D|),\chi_{D}) = \{0\}$. 
	Let $m,m' \in \frac{1}{|D|}\N$ such that $mm'$ is not a rational square. 
	Then the double trace 
	\[
	\frac{1}{\sqrt{|D|mm'}}\tr_m \tr_{m'}(G_{2n})
	\]
	 of the Green's function at positive even $s = 2n$ is a rational linear combination of $\log(p)$ for some primes $p$ and of $\log(\varepsilon_\Delta)/\sqrt{\Delta}$ for some fundamental discriminants $\Delta > 0$, where $\varepsilon_\Delta$ denotes the smallest totally positive unit $> 1$ in $\Q(\sqrt{\Delta})$.
\end{theorem}

\begin{remark}\label{remark on fund disc and primes appearing}
The numbers $\log(\varepsilon_\Delta)/\sqrt{\Delta}$ in the theorem arise as special $L$-values $L(\chi_\Delta,1)$ via Dirichlet's class number formula. The fundamental discriminants that appear are among the discriminants of the real quadratic fields $\Q(\sqrt{(4mm'D^2-r^2)|D|})$ with $r\in \Z$ satisfying~$|r|<2|D|\sqrt{mm'}$. The logarithms of rational primes $p$ can appear only if $(4mm'D^2-r^2)|D|$ is a square for some $r\in \Z$, and if  $p$ is the only prime divisor of $2m'|D|$ satisfying $(-m',D)_p=-1$; see Remark \ref{rem:primes_in_traces}. Note that when $(4mm'D^2-r^2)|D|$ is a square, we have $(-m,D)_p=(-m',D)_p$
by properties of Hilbert symbols. In this case $(-m',D)_p=-1$ implies that $p$ also divides $2m|D|$, which shows that these conditions are symmetric in $m$ and $m'$.
\end{remark}

One can give an explicit evaluation of $\tr_m \tr_{m'}(G_{2n})$ using Theorem~\ref{theorem evaluation Green's function}. The following example is an illustration of such an evaluation.

\begin{example}\label{example 1 intro}
	Let $D = -4$, $m=1$ and $m' = 1/2$. For $n = 1$ we have $S_{3}^+(\Gamma_0(4),\chi_{-4}) = \{0\}$, hence there exists a weight $-1$ weakly holomorphic modular form~$f\in M_{-1}^{!}(\Gamma_0(4),\chi_{-4})$ with principal part $q^{-1}$. Using the formula given in Theorem~\ref{theorem evaluation Green's function} we find that
	\[
	\frac{1}{\sqrt{2}}\tr_1 \tr_{1/2}(G_{2}) = L(\chi_8,1) = \frac{\log(\varepsilon_8)}{\sqrt{8}},
	\]
	where $\varepsilon_8 = 3+\sqrt{8}$ is the smallest totally positive unit $> 1$ in $\Q(\sqrt{8})$. Modulo $\Gamma$, there are unique special points of determinants $1$ and $1/2$, given by $j$ and $\frac{1+i}{2}+ \frac{\sqrt{2}}{2}j$, with stabilizers of size $4$ and $12$, respectively. Hence, we find
	\[
	\frac{1}{\sqrt{2}}G_2\left( j, \frac{1+i}{2}+\frac{\sqrt{2}}{2}j\right) = 48L(\chi_8,1). 
	\]
	Similarly, we can compute
	\[
	\frac{1}{\sqrt{2}}G_4\left(j, \frac{1+i}{2}+\frac{\sqrt{2}}{2}j\right)=48\log(2)-48L(\chi_8,1).
	\]
	We refer to Example~\ref{example 1} for more details. Note that the fundamental discriminant $\Delta=8$ equals the discriminant of $\Q(\sqrt{4mm'|D|^2-r^2)|D|})$ when $r=0$. Moreover, when~$r=4$ the number~$(4mm'|D|^2-r^2)|D|$ is a square, and $2$ is the only prime dividing $2m'|D|$ satisfying $(-m',D)_2=-1$, in accordance with Remark~\ref{remark on fund disc and primes appearing}.
\end{example}

The proof of Theorem~\ref{main theorem} uses a method of Bruinier, Ehlen, and Yang \cite{bruinierehlenyang}, and consists of the following four major steps. 
\begin{enumerate}
	\item First, we show that the trace $\tr_m\big(G_{2n}(\,\cdot\,, P)\big)$ can be written as a regularized theta lift of a harmonic Maass form $F_m$ of weight $1-2n$. More explicitly, we have an integral representation
	\[
	\tr_m\big(G_{2n}(\,\cdot\,, P)\big) = \int_{\mathcal{F}}^{\reg}\left(R_{1-2n}^{n}F_m\right)(\tau)\Theta(\tau,P)d\mu(\tau),
	\]
	where the integral over the fundamental domain $\mathcal{F}$ for $\SL_2(\Z)$ has to be regularized as explained by Borcherds \cite{borcherds} or Bruinier \cite{bruinier}, $d\mu(\tau)$ denotes the usual invariant measure on $\H$, and $\Theta(\tau,P)$ is a real-analytic Siegel theta function associated with a suitable lattice of signature $(1,3)$. Here $R_{\kappa}^n = R_{\kappa+2n-2}\circ \dots \circ R_{\kappa}$ denotes the iterated version of the raising operator $R_{\kappa} = 2i\frac{\partial}{\partial \tau} + \kappa v^{-1}$ where $v=\Im(\tau)$. See Theorem~\ref{theorem higher Green's function as theta lift} for the details.	
	\item Now we evaluate $\tr_m\big(G_{2n}(\,\cdot\,, P_0)\big)$ at a special point $P_0$. Then the Siegel theta function $\Theta(\tau,P_0)$ essentially splits as a product of a holomorphic theta function of weight $1/2$ and the complex conjugate of a holomorphic theta function of weight $3/2$, 
	\[
	\Theta(\tau,P_0) = \Theta_{1/2}(\tau) \cdot \overline{\Theta_{3/2}(\tau)}v^{3/2}.
	\]
	If we now take the trace $\tr_{m'}$ over $P_0$, an application of the Siegel--Weil formula allows us to replace the theta function $\Theta_{3/2}(\tau)$ by an Eisenstein series $E_{3/2}(\tau)$, and we obtain the splitting
	\[
	\tr_{m'}(\Theta(\tau,\,\cdot \,)) = \Theta_{1/2}(\tau)\cdot \overline{E_{3/2}(\tau)}v^{3/2}.
	\]	
	We refer to Theorem~\ref{theorem splitting theta} for the precise statement.
	\item If we plug the splitting of $\tr_{m'}(\Theta(\tau,\,\cdot \,))$ into the  theta lift representation of $\tr_{m}(G_{2n})$ from step $(1)$, we obtain
	\begin{align*}
	\tr_{m}\tr_{m'}(G_{2n}) = \int_{\mathcal{F}}^{\reg}\left(R_{1-2n}^{n}F_m\right)(\tau)\Theta_{1/2}(\tau)\overline{E_{3/2}(\tau)}v^{3/2}d\mu(\tau).
	\end{align*}
	The right-hand side can be interpreted as the regularized Petersson inner product of the form $\left(R_{1-2n}^{n}F_m\right)(\tau)\Theta_{1/2}(\tau)$ and the Eisenstein series $E_{3/2}(\tau)$. By the fundamental results of Bruinier and Funke \cite{bruinierfunke}, there exists a harmonic Maass form $\widetilde{E}_{1/2}(\tau)$ of weight $1/2$ with shadow $E_{3/2}(\tau)$, and an application of Stokes' Theorem shows that the above Petersson inner product can be evaluated in terms of the coefficients of $F_m(\tau)$ and $\Theta_{1/2}(\tau)$, and the coefficients of the holomorphic part of $\widetilde{E}_{1/2}(\tau)$. Here we need the assumption~$S_{1+2n}^+(\Gamma_0(|D|),\chi_D) \neq \{0\}$, which implies that~$F_m$ is weakly holomorphic.
	\item Since the coefficients of $\Theta_{1/2}(\tau)$ and $F_m(\tau)$ are rational numbers (here we again use that $F_m$ is weakly holomorphic), the algebraic properties of $\tr_{m}\tr_{m'}(G_{2n})$ are controlled by the coefficients of the harmonic Maass form $\widetilde{E}_{1/2}(\tau)$. In general, the algebraic nature of the coefficients of harmonic Maass forms is a deep open problem, compare \cite{bruinierono}. However, since the shadow of $\widetilde{E}_{1/2}(\tau)$ is an Eisenstein series, $\widetilde{E}_{1/2}(\tau)$ itself can be constructed as a Maass Eisenstein series, and its Fourier coefficients can be computed very explicitly using a method of Bruinier and Kuss \cite{bruinierkuss}; see Theorem~\ref{eisenstein series fourier expansion}. It turns out that the coefficients of the holomorphic part of $\widetilde{E}_{1/2}(\tau)$ are given by simple multiples of logarithms of primes, or by the special value $L(\chi_{\Delta},1)$ of Dirichlet $L$-functions  for positive fundamental discriminants $\Delta > 0$. By Dirichlet's class number formula, the latter $L$-values are rational multiples of $\log(\varepsilon_\Delta)/\sqrt{\Delta}$. 
\end{enumerate}

\begin{remark}
	Step (4) explains why we need to take the double trace $\tr_{m}\tr_{m'}(G_{2n})$ to obtain a convenient algebraicity result. Steps (1) to (3) work without taking the trace $\tr_{m'}$, and show that, for each fixed special point $P_0$, we can interpret the single trace $\tr_{m}(G_{2n}(\,\cdot\,,P_0))$ as a regularized Petersson inner product of $\left(R_{1-2n}^{n}F_m\right)(\tau)\Theta_{1/2}(\tau)$ with a weight $3/2$ holomorphic theta function $\Theta_{3/2}(\tau)$ associated to an even lattice of rank $3$. However, it is believed that the harmonic Maass forms corresponding to these ternary theta functions in general do not have good algebraic properties. Taking the additional trace $\tr_{m'}$ allows us to apply the Siegel--Weil formula and replace $\Theta_{3/2}(\tau)$ by an Eisenstein series $E_{3/2}(\tau)$, whose corresponding harmonic Maass form $\widetilde{E}_{1/2}(\tau)$ has better algebraic properties. In contrast, in the modular curve case, the ternary theta function $\Theta_{3/2}$ is replaced with a binary theta function $\Theta_1$ of weight $1$. Using the theory of complex multiplication, Duke and Li \cite{dukeli} and Ehlen \cite{ehlen} proved that these binary theta functions possess corresponding harmonic Maass forms whose coefficients are essentially given by logarithms of absolute values of algebraic numbers. Hence, in the modular curve case it is not necessary to take the double trace to obtain an algebraicity result.
\end{remark}

Next, we consider \emph{twisted} double traces of the Green's function, which are defined as in \eqref{def double trace}, but with additional signs $\chi_D(X),\chi_D(Y) \in \{-1,0,1\}$. Recall that we assume that $D$ is a prime discriminant. Let $\ell$ denote the unique prime dividing $D$. For $X = \left(\begin{smallmatrix}a & b \\ \overline{b} & c \end{smallmatrix} \right) \in L^+_{m|D|}$ with $m \in \N$, we define
\[
\chi_D(X) =  \begin{cases}
\left( \frac{D}{a}\right), & \text{if } \ell \nmid a, \\
\left( \frac{D}{c}\right), & \text{if } \ell \nmid c, \\
0, & \text{otherwise}.
\end{cases}
\]
This function was previously considered in \cite{bruinieryang, ehlen} for $D > 0$ in order to study twisted Borcherds produts on Hilbert modular surfaces. One can check that $\chi_D(X)$ is well-defined and invariant under the action of $\Gamma$. In particular, for $m,m' \in \N$ the \emph{doubly-twisted double trace}
\[
\tr_{m|D|,\chi_D}\tr_{m'|D|,\chi_D}(G_{s}) = \sum_{\substack{X \in \Gamma \backslash L_{m|D|}^+ \\ Y \in \Gamma \backslash L_{m'|D|}^+}}\frac{\chi_D(X)}{|\Gamma_X|}\frac{\chi_D(Y)}{|\Gamma_Y|}G_s(P_X,P_Y)
\]
is well-defined. Similarly, we can consider the \emph{partially-twisted double trace}
\[
\tr_{m|D|,\chi_D}\tr_{m'}(G_{s}) = \sum_{\substack{X \in \Gamma \backslash L_{m|D|}^+ \\ Y \in \Gamma \backslash L_{m'}^+}}\frac{\chi_D(X)}{|\Gamma_X|}\frac{1}{|\Gamma_Y|}G_s(P_X,P_Y),
\]
where we only twist one of the traces. Note that here $m'$ can be a rational number in $\frac{1}{|D|}\N$.

There are two main reasons why we are interested in the twisted double traces of the Green's function. Firstly, in Theorem~\ref{main theorem} we considered the \emph{non-twisted} double trace of $G_s$ for \emph{even} $s = 2n$, but we did not get information for odd $s$. However, we will compute the \emph{twisted} double traces of $G_s$ for \emph{odd} $s = 2n+1$. Secondly, we can sometimes compute both the non-twisted and the twisted double traces of $G_{s}$ and use this to get formulas for some \emph{individual} special values of the Green's function, even if the class numbers of the involved binary hermitian forms are not equal to $1$. For an instance of such a case, we refer to Example~\ref{example 2 intro} below.

We have the following algebraicity results for the twisted double traces of the Green's function. Once again, for simplicity, we assume that certain spaces of cusp forms are trivial. In the general case we take suitable linear combinations of twisted double traces; see Corollaries~\ref{corollary that implies theorem 1.4(2) n even}, \ref{corollary that implies theorem 1.4(2) n odd} and \ref{corollary that implies theorem 1.4(1)}.

As usual,  for an integer $k$ we denote by $S_{k}(\SL_2(\Z))$ the space of cusp forms of weight $k$ for the full modular group.

\begin{theorem}\label{main theorem twisted}
	Let $D < 0$ be a prime discriminant.
	\begin{enumerate}
		\item Let $n \geq 1$ be a natural number such that $S_{2+2n}(\SL_2(\Z)) = 0$. Let $m,m' \in \N$ such that $mm'$ is not a square. Then the doubly-twisted double trace 
		\[
		\tr_{m|D|,\chi_D} \tr_{m'|D|,\chi_D}(G_{2n+1})
		\]
	 is a rational linear combination of $\log(p)$ for some primes $p$ and of $\log(\varepsilon_\Delta)/\sqrt{\Delta}$ for some fundamental discriminants $\Delta > 0$, where $\varepsilon_\Delta$ denotes the smallest totally positive unit $> 1$ in $\Q(\sqrt{\Delta})$.
	 	\item Let $n \geq 2$ be a natural number. If $n$ is even, assume that $S_{1+n}^+(\Gamma_0(|D|),\chi_D) = \{0\}$, and if $n$ is odd, assume that $S_{1+n}(\SL_2(\Z)) = 0$. Let $m \in \N , m'\in \frac{1}{|D|}\N$ such that $mm'|D|$ is not a square.  Then the partially-twisted double trace
		\[
		\tr_{m|D|,\chi_D} \tr_{m'}(G_{n})
		\]
		is a rational multiple of $\pi  \big(\sqrt{mm'}\big)^{1+n}$.
	\end{enumerate}
\end{theorem}

	The proofs of these results, which can be found in Section~\ref{section twisted traces}, use the same strategy as the proof of Theorem~\ref{main theorem} sketched above. However, in step (1) of the proof sketch we use a twisted theta lift, and in step (2) we apply a twisted version of the Siegel--Weil formula, which is due to Snitz \cite{snitz}. Roughly speaking, this twisted Siegel--Weil formula says that a certain twisted integral of a theta function is a \emph{distinguished cusp form}, namely a weight 3/2 unary theta function, instead of an Eisenstein series. These unary theta functions admit harmonic Maass forms with \emph{rational} Fourier coefficients in the holomorphic part, up to a fixed square root factor (see \cite{bruinierschwagenscheidt, lischwagenscheidt}). Hence, in part (4) of the proof sketch we obtain rational numbers instead of logarithms (up to a factor of $\pi$ and possibly some square roots). This explains the different algebraic properties of the doubly-twisted and the partially-twisted traces in Theorem~\ref{main theorem twisted}.

\begin{example}\label{example 2 intro}
	By combining our explicit evaluations of the twisted and non-twisted double traces of the Green's function we can sometimes compute the individual values $G_{2n}(P_1,P_2)$, even if the class numbers of the two special points $P_1,P_2$ are not equal to $1$. For example, modulo $\Gamma=\mathrm{PSL}_2(\Z[i])$ there is one form $\left(\begin{smallmatrix}1 & 0 \\ 0 & 1\end{smallmatrix} \right)$ of determinant $m = 1$, with corresponding special point $j$, and there are two primitive forms $\left(\begin{smallmatrix}4 & 0 \\ 0 & 1\end{smallmatrix}\right)$ and $\left(\begin{smallmatrix}3 & 1+i \\ 1-i & 2 \end{smallmatrix}\right)$ of determinant $m' = 4$, with corresponding special points $2j$ and $\frac{1+i}{2} + j$. Using the formula for the non-twisted double trace of $G_2$ from Theorem~\ref{theorem evaluation Green's function}, we find
	\[
	G_2\left(j,2j \right) + G_2\left(j,\frac{1+i}{2}+j \right) =  32 L(\chi_{12},1) -8 L(\chi_{28},1) + 56L(\chi_{60},1).
	\]
	On the other hand, the formula for the partially-twisted double trace of $G_2$ from Theorem~\ref{theorem evaluation Green's function single twist even} yields
	\[
	G_2\left(j,2j \right) - G_2\left(j,\frac{1+i}{2}+j \right) = -4\pi.
	\]
	Combining these two evaluations, we obtain
	\[
	G_{2}\left(j,2j \right) = 16 L(\chi_{12},1) - 4 L(\chi_{28},1) + 28 L(\chi_{60},1)-2\pi.
	\]
	In particular, the algebraicity results for the double traces of $G_2$ are in general not true for the individual values. We refer to Section~\ref{section example} for the details.
\end{example}

This work is organized as follows. In Section~\ref{section preliminaries} we discuss the necessary preliminaries about vector-valued harmonic Maass forms for the Weil representation associated with an even lattice, and some basic properties of Siegel theta functions. In Section~\ref{section eisenstein series} we recall the construction of vector-valued holomorphic and harmonic Maass Eisenstein series of half-integral weight, and we give the Siegel--Weil formula in our setup. The results of this section are mostly well known. However, in Theorem~\ref{eisenstein series fourier expansion} we write out an explicit formula for the Fourier coefficients of a harmonic Eisenstein series of weight $1/2$, which is not readily available in the literature and might be of independent interest. Section~\ref{section traces green's function} is the heart of the paper. Here we first show that the trace of the Green's function can be written as a theta lift (Theorem~\ref{theorem higher Green's function as theta lift}). The next key step is to rewrite the double traces into an adelic language in order to apply the Siegel--Weil formula and determine a precise splitting of the Siegel theta function at special points (Theorem~\ref{theorem splitting theta}). Finally, we give our explicit evaluation of the double traces in terms of Fourier coefficients of Maass Eisenstein series (Theorem~\ref{theorem evaluation Green's function}). In Section~\ref{section twisted traces green's function} we consider the twisted double traces of the Green's function. The proofs of the results in this last section are very similar to the proofs of their non-twisted counterparts in Section~\ref{section traces green's function}, so we will skip some details there. \\

\noindent \textbf{Acknowledgments.} We are grateful to  Yingkun Li for pointing out the paper of Snitz \cite{snitz} to us.

\section{Preliminaries}\label{section preliminaries}

Throughout this section, we let $L$ be an even lattice of signature $(p,q)$ with   bilinear form $(\cdot,\cdot)$ and associated quadratic form $Q(x)=\frac{1}{2}(x,x)$. The dual lattice of $L$ will be denoted by $L'$. The abelian group $L'/L$ is finite and of cardinality $|\mathrm{det}(L)|$, where $\mathrm{det}(L)$ denotes the determinant of the Gram matrix of $L$.

\subsection{Vector-valued modular forms for the Weil representation}\label{section modular forms} Let $\C[L'/L]$ be the group ring of $L$, which is generated by the standard basis vectors $\e_{\gamma}$ for $\gamma \in L'/L$. We let
\[
\bigg(\sum_{\gamma \in L'/L}a_\gamma \e_\gamma\bigg)\cdot \bigg( \sum_{\gamma \in L'/L}b_\gamma \e_\gamma\bigg) = \sum_{\gamma \in L'/L}a_\gamma b_\gamma
\]
be the natural bilinear pairing on $\C[L'/L]$.

 Let $\Mp_{2}(\Z)$ be the metaplectic double cover of $\SL_{2}(\Z)$, realized as the set of pairs $(M,\phi)$ with $M = \left(\begin{smallmatrix}a & b \\ c & d \end{smallmatrix} \right) \in \SL_{2}(\Z)$ and $\phi:\H\to\C$ a holomorphic function with $\phi^{2}(\tau) = c\tau + d$. The Weil representation $\rho_{L}$ of $\Mp_{2}(\Z)$ associated to $L$ is defined on the generators $T = \left(\left(\begin{smallmatrix}1 & 1 \\ 0 & 1 \end{smallmatrix} \right), 1 \right)$ and $S = \left(\left(\begin{smallmatrix}0 & -1 \\ 1 & 0 \end{smallmatrix} \right), \sqrt{\tau} \right)$ by 
\begin{align*}
\rho_{L}(T)\e_{\gamma} = e(Q(\gamma))\e_{\gamma},  \qquad  \rho_{L}(S)\e_{\gamma}= \frac{e((q-p)/8)}{\sqrt{|L'/L|}}\sum_{\beta \in L'/L}e(-(\beta,\gamma))\e_{\beta},
\end{align*}
where we put $e(x) = e^{2\pi i x}$ for $x \in \C$.  We let $\overline{\rho}_L$ denote the dual Weil representation. Note that $\overline{\rho}_L$ is the Weil representation $\rho_{L^{-}}$ associated to the lattice $L^{-} = (L,-Q)$. 

For $k \in \frac{1}{2}\Z$ we let $A_{k,L}$ be the set of all functions $f: \H \to \C$ that transform like modular forms of weight $k$ for $\rho_{L}$, which means that $f$ is invariant under the slash operator 
\begin{align}\label{slash operator}
f|_{k,L}(M,\phi)= \phi(\tau)^{-2k}\rho_{L}(M,\phi)^{-1}f(M\tau)
\end{align}
for $(M,\phi) \in \Mp_{2}(\Z)$. If $K \subseteq L$ is a sublattice of finite index, we can naturally view modular forms for $\rho_L$ as modular forms for $\rho_K$ as follows. We have the inclusions $K \subseteq L \subseteq L' \subseteq K'$ and thus $L/K \subseteq L'/K \subseteq K'/K$. We have the natural projection $L'/K \to L'/L, \gamma \mapsto \overline{\gamma}$. There are maps
\begin{align*}
\res_{L/K}&: A_{k,L} \to A_{k,K}, \quad f\mapsto f_{K}, \\
\tr_{L/K}&: A_{k,K} \to A_{k,L}, \quad g\mapsto g^{L},
\end{align*}
which are defined for $f \in A_{k,L}$ and $\gamma \in K'/K$ by
\begin{align*}
(f_{K})_{\gamma} = \begin{cases} f_{\overline{\gamma}}, & \text{if } \gamma \in L'/K, \\ 0, & \text{if } \gamma \notin L'/K,\end{cases}
\end{align*}
and for $g \in A_{k,K}$ and $\overline{\gamma} \in L'/L$ by
\begin{align*}
(g^{L})_{\overline{\gamma}} = \sum_{\beta \in L/K}g_{\beta+\gamma}.
\end{align*}
They are adjoint with respect to the bilinear pairings on $\C[L'/L]$ and $\C[K'/K]$. We refer the reader to \cite[Lemma~3.1]{bruinieryang} for more details.

\subsection{Harmonic Maass forms}\label{section harmonic maass forms}

Recall from \cite{bruinierfunke} that a harmonic Maass form of weight $k \in \frac{1}{2}\Z$ for $\rho_{L}$ is a smooth function $f: \H \to \C$ which is annihilated by the weight $k$ Laplace operator 
\[
\Delta_{k} = -v^{2}\left(\frac{\partial^{2}}{\partial u^{2}}+\frac{\partial^{2}}{\partial v^{2}} \right)+ikv\left(\frac{\partial}{\partial u}+i\frac{\partial}{\partial v} \right), \qquad (\tau = u + iv \in \H),
\]
which transforms like a modular form of weight $k$ for $\rho_{L}$, and which is at most of linear exponential growth at the cusp $\infty$. The space of harmonic Maass forms of weight $k$ for $\rho_{L}$ is denoted by $H_{k,L}$. We let $M_{k,L}^{!}$ be subspace of weakly holomorphic modular forms, which consists of the forms that are holomorphic on $\H$. The antilinear differential operator
\[
\xi_{k} = 2iv^{k}\overline{\frac{\partial}{\partial \overline{\tau}}}
\]
maps $H_{k,L}$ onto $M_{2-k,L^{-}}^{!}$. We let $H_{k,L}^{\hol}$ and $H_{k,L}^{\cusp}$ 
 be the subspace of $H_{k,L}$ which is mapped to the space $M_{2-k,L^{-}}$ of holomorphic modular forms or the space $S_{2-k,L^{-}}$ of cusp forms under $\xi_{k}$, respectively. For $k \neq 1$ every $f \in H_{k,L}^{\hol}$ decomposes as a sum $f = f^{+} + f^{-}$ of a holomorphic and a non-holomorphic part, having Fourier expansions of the form
\begin{align}\label{eq Fourier expansion}
\begin{split}
f^{+}(\tau) &= \sum_{\gamma \in L'/L}\sum_{\substack{n \in \Q \\ n \gg -\infty}}a_{f}^{+}(n,\gamma)q^{n}\e_{\gamma}, \\
f^{-}(\tau) &=  \sum_{\gamma \in L'/L}\bigg(a_{f}^{-}(0,\gamma)v^{1-k}+\sum_{\substack{n \in \Q \\ n < 0}}a_{f}^{-}(n,\gamma)\Gamma(1-k,4\pi|n|v)q^{n}\bigg)\e_{\gamma},
\end{split}
\end{align}
where $a_{f}^{\pm}(n,\gamma) \in \C$, $q = e^{2\pi i \tau}$, and $\Gamma(s,x) = \int_{x}^{\infty}e^{-t}t^{s-1}dt$ is the incomplete Gamma function. Note that $f \in H_{k,L}^{\cusp}$ is equivalent to $a_{f}^{-}(0,\gamma) = 0$ for all $\gamma \in L'/L$.

Examples of harmonic Maass forms can be constructed using Maass Poincar\'e series, compare \cite[Section~1.3]{bruinier}. For $k \in \frac{1}{2}\Z$, $s \in \C$ and $v > 0$ we let
\begin{equation}\label{eq M Whittaker}
\mathcal{M}_{k,s}(v) = v^{-k/2}M_{-k/2,s-1/2}(v),
\end{equation}
with the usual $M$-Whittaker function. For $\mu \in L'/L$ and $m \in \Z - Q(\mu)$ with $m > 0$, and $s \in \C$ with $\Re(s) > 1$ we define the Maass Poincar\'e series
\[
F_{k,m,\mu}(\tau,s) = \frac{1}{2\Gamma(2s)}\sum_{(M,\phi) \in \widetilde{\Gamma}_{\infty}\backslash \Mp_{2}(\Z)}\mathcal{M}_{k,s}(4\pi m v)e(-mu)\e_{\mu}|_{k,L}(M,\phi),
\]
where $\widetilde{\Gamma}_{\infty}$ is the subgroup of $\Mp_{2}(\Z)$ generated by $T=\left(\left(\begin{smallmatrix}1 & 1 \\ 0 & 1 \end{smallmatrix} \right), 1 \right)$. It converges absolutely for $\Re(s) > 1$, it transforms like a modular form of weight $k$ for $\rho_{L}$, and it is an eigenform of the Laplace operator $\Delta_{k}$ with eigenvalue $s(1-s)+(k^{2}-2k)/4$. Hence, for $k < 0$ the special value
\[
F_{k,m,\mu}(\tau) = F_{k,m,\mu}\left(\tau,1-\frac{k}{2}\right)
\]
defines a harmonic Maass form in $H_{k,L}^{\cusp}$ whose Fourier expansion starts with
\[
F_{k,m,\mu}(\tau) = q^{-m}(\e_{\mu}+\e_{-\mu}) + O(1).
\]
In particular, for $k < 0$ every harmonic Maass form $f \in H_{k,L}^{\cusp}$ with Fourier expansion as in \eqref{eq Fourier expansion} can be written as a linear combination
\begin{align}\label{eq linear combination Maass Poincare}
f(\tau) = \frac{1}{2}\sum_{\mu \in L'/L}\sum_{m > 0}a_{f}^{+}(-m,\mu)F_{k,m,\mu}(\tau)
\end{align}
of Maass Poincar\'e series.

The Maass raising and lowering operators on smooth functions on $\H$ are defined by
\[
R_{k}f = 2i \frac{\partial}{\partial \tau} + kv^{-1}, \qquad L_k = -2iv^{2}\frac{\partial}{\partial \overline{\tau}}.
\]
They raise or lower the weight of an automorphic form of weight $k$ by $2$, respectively. Also note that we have $L_k = v^{2-k}\overline{\xi_{k}}$. We also define the iterated raising operator by
\[
R_{k}^n = R_{k+2n-2}\circ \dots \circ R_{k}, \qquad R_{k}^0 = \mathrm{id}.
\]
The action of the raising operator on Maass Poincar\'e series is given as follows; see, e.g., \cite[Proposition~3.4]{bruinierehlenyang}.

\begin{lemma}\label{lemma raising poincare series}
	We have
	\[
	R_{k}F_{k,m,\mu}(\tau,s) = 4\pi m (s+k/2) F_{k+2,m,\mu}(\tau,s).
	\]
\end{lemma}

\subsection{Rankin--Cohen brackets}\label{section rankin cohen brackets}

Let $K$ and $L$ be even lattices. For $n\in\N_0$ and functions $f \in A_{k,K}$ and $g \in A_{\ell,L}$ with $k,\ell \in \frac{1}{2} \Z$ we define the $n$-th \emph{Rankin--Cohen bracket}
\begin{align*}
[f,g]_n = \sum_{s = 0}^n (-1)^s \binom{k+n-1}{s}\binom{\ell+n-1}{n-s}  f^{(n-s)} \otimes g^{(s)},
\end{align*}
where $f^{(s)} = \frac{1}{(2\pi i)^s}\frac{\partial^s}{\partial\tau^s}f$, and the tensor product of two vector-valued functions $f = \sum_{\mu}f_\mu \e_\mu \in A_{k,K}$ and $g = \sum_{\nu}g_\nu \e_\nu \in A_{\ell,L}$ is defined by $f \otimes g =\sum_{\mu , \nu} f_\mu g_\nu \e_{\mu + \nu} \in A_{k+\ell,K \oplus L}$. We can write the Rankin--Cohen bracket in terms of the raising operator as
\begin{align*}
[f,g]_n = \frac{1}{(-4\pi)^n}\sum_{s = 0}^n (-1)^s \binom{k+n-1}{s}\binom{\ell+n-1}{n-s}  R_{k}^{n-s}f \otimes R_\ell^{s}g,
\end{align*}
(see \cite[Equation (3.7)]{bruinierehlenyang}) which implies that we have
\[
[f,g]_n \in A_{k + \ell+2n,K \oplus L}.
\]
We will need the following formula from \cite[Proposition~3.6]{bruinierehlenyang} for the action of the lowering operator on Rankin--Cohen brackets.

\begin{proposition}\label{proposition rankin cohen bracket}
	Let $f \in H_{k,K}$ and $g \in H_{\ell,L}$ be harmonic Maass forms. For $n \in \N_0$ we have
	\begin{align*}
	(-4\pi)^n L_{k+\ell+2n}[f,g]_n= \binom{\ell+n-1}{n} R_k^n f \otimes L_\ell g + (-1)^n \binom{k+n-1}{n}L_k f \otimes R_\ell^n g.
	\end{align*}
\end{proposition}

\subsection{Rational relations for spaces of cusp forms}\label{sec rational relations}

A sequence
\begin{equation}\label{eq rational relation}
   \{\lambda(m,\mu)\}_{m\in \Q^+,\,\mu \in L'/L}\subseteq \mathbb{Q} 
\end{equation}
is called a \emph{rational relation} for~$S_{k,L}$ if the following conditions are satisfied:
\begin{enumerate}
    \item For each~$\mu\in L'/L$ we have~$\lambda(m,-\mu)=\lambda(m,\mu)$,
    \item For each~$\mu\in L'/L$ we have~$\lambda(m,\mu)=0$ for all but finitely many~$m\in \Q^+$,
    \item $\sum_{\mu \in L'/L}\sum_{m > 0}\lambda(m,\mu)c(m,\mu)=0$ for all cusp forms~$\sum_{\mu \in L'/L}\sum_{m > 0}c(m,\mu)q^m\e_{\mu}$ in~$ S_{k,L}$.
\end{enumerate}
Similarly, given a discriminant $D<0$, one defines rational relations~$\{\lambda(t)\}_{t\in \N}\subseteq \Q$ for the space of scalar-valued cusp forms~$S^+_{k}(\Gamma_0(|D|),\chi_D)$ as in \cite[p.~316]{grosszagier}.

We now recall the well known fact (see, e.g., \cite[Theorem~1.17]{bruinier}) that a sequence as in~\eqref{eq rational relation} satisfying conditions~$(1)$ and~$(2)$ above, is a rational relation for~$S_{k,L}$ if and only if there is a form $f\in M_{2-k,L^-}^{!}$ with Fourier coefficients $a_f(m,\mu)$ such that $\lambda(m,\mu)=a_f(-m,\mu)$ for all~$m>0$ and all~$\mu \in L'/L$.

For a  prime discriminant $D<0$ and the lattice $L$ of determinant $|D|$, the assignment
$$\sum_{\mu\in L'/L}f_{\mu}(\tau)\e_{\mu} \mapsto \sum_{\mu\in L'/L}f_{\mu}(|D|\tau)$$
defines a linear map~$S_{k,L}\to S^+_{k}(\Gamma_0(|D|),\chi_D)$ (see \cite{bruinierbundschuh}). This implies the following lemma.

\begin{lemma}\label{lem vector-valued to scalar mfs}
Assume  $D<0$ is a prime discriminant and the even lattice $L$ has determinant $|D|$. Then, for every rational relation~$\{\lambda(t)\}_{t\in \N}$ for~$S^+_{k}(\Gamma_0(|D|),\chi_D)$ the sequence~$\lambda'(m,\mu)=\lambda(m|D|)$ defines a relation for~$S_{k,L}$ (that is independent of $\mu\in L'/L$).
\end{lemma}

\subsection{Siegel theta functions and special points}\label{section theta functions}

As before, we let $L$ be an even lattice of signature $(p,q)$. Let $\Gr(L)$ be the Grassmannian of positive definite $p$-dimensional subspaces of $V(\R) = L \otimes \R$. The Siegel theta function associated to $L$ is defined by
\begin{align*}
\Theta_{L}(\tau,v) = \Im(\tau)^{q/2}\sum_{\gamma \in L'/L}\sum_{X \in L+\gamma}e\left(Q(X_{v})\tau+Q(X_{v^{\perp}})\overline{\tau}\right)\e_{\gamma},
\end{align*}
where $\tau \in \H$ and $v \in \Gr(L)$, and $X_{v}$ denotes the orthogonal projection of $X$ to $v$. The Siegel theta function transforms like a modular form of weight $\frac{p-q}{2}$ for $\rho_{L}$ in $\tau$ (see \cite[Theorem~4.1]{borcherds}) and is invariant in $v$ under the subgroup of $\mathrm{O}(L)$ fixing $L'/L$. 

We call $v\in \Gr(L)$ a \emph{special point} if it is defined over $\Q$, that is, if there exists $v_0\subseteq L \otimes \Q$ such that $v=v_0\otimes \R$. 
For a special point $v \in \Gr(L)$ its orthogonal complement $v^{\perp}$ in $V(\R)$ is also defined over $\Q$ and we obtain the rational splitting $L \otimes \Q = v \oplus v^{\perp}$ which yields the positive and negative definite lattices
\[
P = L \cap v, \qquad N = L \cap v^{\perp}.
\]
Note that $P \oplus N$ is a sublattice of $L$ of finite index. The Siegel theta functions associated to $L$ and $P \oplus N$ are related by
\begin{align}\label{eq theta relation}
\Theta_{L} = \left(\Theta_{P\oplus N}\right)^{L},
\end{align}
with the trace operator defined in Section~\ref{section modular forms}. Moreover, by  identifying $\C[(P\oplus N)'/(P\oplus N)]$ with $\C[P'/P] \otimes \C[N'/N]$ the Siegel theta function associated to $P \oplus N$ splits as a tensor product
\begin{align*}\label{splitting siegel theta}
\Theta_{P \oplus N}(\tau,v) = \Theta_{P}(\tau) \otimes \Theta_{N}(\tau),
\end{align*}
where
\begin{equation}\label{eq usual holomorphic theta}
    \Theta_{P}(\tau) = \sum_{\gamma \in P'/P}\sum_{X \in P+\gamma}e(Q(X)\tau)\e_\gamma \in M_{p/2,P}
\end{equation}
is the usual holomorphic (vector-valued) theta series associated with $P$, and
\[
\Theta_N(\tau) =  \Im(\tau)^{q/2}\overline{\Theta_{N^-}(\tau)}
\]
with the holomorphic theta series $\Theta_{N^-} \in M_{q/2,N^-}$, where $N^-=(N,-Q)$.

\section{Vector-valued Eisenstein series and the Siegel--Weil formula}\label{section eisenstein series}

In this section, we construct vector-valued holomorphic and harmonic Eisenstein series on positive definite lattices. The holomorphic Eisenstein series appear in the Siegel--Weil formula, which we will state at the end of this section. The Fourier coefficients of the harmonic Eisenstein series will be needed for the explicit evaluation of the double traces of our Green's function. 

Throughout this section, we let $(L,Q)$ denote a positive definite even lattice of rank $r \geq 1$, and $L^-=(L,-Q)$.

\subsection{Holomorphic Eisenstein series}\label{section holomorphic eisenstein series} Following \cite{bruinierkuehn}, for $k \in \frac{1}{2}\Z$ with $2k+r \equiv 2 \pmod 4$ and $s \in \C$ we consider the $\C[L'/L]$-valued non-holomorphic Eisenstein series
\[
E_{k,L}(\tau,s) = \frac{1}{4}\sum_{(M,\phi) \in \widetilde{\Gamma}_\infty \backslash \Mp_2(\Z)}(v^s \e_0) |_{k,L}(M,\phi),
\]
where $\widetilde{\Gamma}_\infty$ is the subgroup generated by $T = \left( \left(\begin{smallmatrix}1 & 1 \\ 0 & 1 \end{smallmatrix}\right),1\right) \in \Mp_2(\Z)$, and $|_{k,L}$ denotes the vector-valued slash operator defined in \eqref{slash operator}. Note that we multiplied the Eisenstein series from \cite{bruinierkuehn} by $1/2$, and we work with the Weil representation instead of its dual. 

The Eisenstein series converges for $\Re(s) > 1-k/2$, transforms like a modular form of weight $k$ for $\rho_L$, and satisfies the Laplace equation
\[
\Delta_k E_{k,L}(\tau,s) = s(1-k-s)E_{k,L}(\tau,s).
\]
In particular, for $k>2$ the special value
\[
E_{k,L}(\tau) = E_{k,L}(\tau,0)
\]
defines a holomorphic modular form of weight $k$ for $\rho_L$. It is normalized such that the constant term at the $\e_0$-component is $1$.

We will be particularly interested in the case that the weight $k$ of the Eisenstein series equals $r/2$, with small rank $r$. In this case, the analytic continuation of the Eisenstein series $E_{k,L}(\tau,s)$ to $s = 0$ was proved by Rallis (see \cite[Proposition~4.3]{rallis}) in the course of extending the Siegel--Weil formula to lattices of small ranks. It also follows from the Siegel--Weil formula of Rallis that the special value $E_{k,L}(\tau) = E_{k,L}(\tau,0)$ defines a holomorphic modular form in $\tau$. Note that Rallis worked in an adelic setup, and considered Eisenstein series $E(\varphi,s)$ associated with certain Schwartz functions $\varphi$. However, as explained in \cite[Section~4.2]{kudlaextensionsiegelweil} or \cite[Section~2.2]{bruinieryang}, the components of our classical Eisenstein series $E_{k,L}(\tau,s)$ are given by the Eisenstein series $E(\varphi,s)$ for suitable choices of Schwartz functions $\varphi$. Hence, we obtain the following result.

\begin{theorem}\label{thm Rallis holomorphic Eisenstein series}
	If $k = r/2$ with $r \geq 1$ then $E_{k,L}(\tau,s)$ has an analytic continuation to $s = 0$, and $E_{k,L}(\tau) = E_{k,L}(\tau,0)$ is a holomorphic modular form of weight $k$ for $\rho_L$.
\end{theorem}

\begin{example}\label{example rank 1}
	For $r = 1$ and $k = 1/2$, the Eisenstein series $E_{k,L}(\tau)$ is a holomorphic modular form of weight $1/2$ for $\rho_L$, and hence a linear combination of unary theta functions (see \cite[Lemma~2.1]{bruinierschwagenscheidt}). For instance, if we take the rank $1$ lattice $L = \Z$ with $Q(x) = x^2$, we have $L'/L \cong \Z/2\Z$, and
	\[
	E_{1/2,L}(\tau) = \Theta_{L}(\tau), \qquad \Theta_{L}(\tau) = \sum_{r \!\!\!\!\!\pmod 2}\sum_{\substack{n \in \Z \\ n \equiv r \!\!\!\!\! \pmod 2}}q^{n^2/4}\e_r.
	\]
	Note that vector-valued modular forms for $\rho_L$ can be identified with scalar-valued modular forms for $\Gamma_0(4)$ in the Kohnen plus space via $\sum_{r(2)}f_r(\tau)\e_r \mapsto f_0(4\tau) + f_1(4\tau)$, and under this map the function $\Theta_{L}(\tau)$ corresponds to the usual Jacobi theta function $\theta(\tau) = \sum_{n \in \Z}q^{n^2}$.
\end{example}	

\begin{example}\label{example rank 3}
	For $r = 3$ and $k = 3/2$, the coefficients of $E_{3/2,L}(\tau)$ are essentially class numbers of imaginary quadratic fields; see Remark~\ref{remark eisenstein series}(4). 
 Moreover, by the Siegel--Weil formula (see Theorem~\ref{siegel weil formula} below), $E_{3/2,L}(\tau)$ is a linear combination of theta functions associated with ternary lattices in the genus of $L$. For instance, if we take the ternary lattice $L = \Z^3$ with $Q(x_1,x_2,x_3) = x_1^2 + x_2^2 + x_3^2$, we have $L'/L \cong (\Z/2\Z)^3$ and
	\[
	E_{3/2,L}(\tau) = \Theta_{L}(\tau), \qquad \Theta_{L}(\tau) = \sum_{r_1,r_2,r_3 \!\!\!\!\! \pmod 2}\sum_{\substack{x_1,x_2,x_3 \in \Z \\ x_i \equiv r_i \!\!\!\!\! \pmod 2}}q^{(x_1^2 + x_2^2 + x^3)/4}\e_{(r_1,r_2,r_3)}.
	\]
	This identity also follows from the fact that there are no non-trivial cusp forms of weight $3/2$ for $\rho_{L}$. Note that the $\e_0$-component of $\Theta_{L}(\tau)$ is just $\theta^3(\tau)$, the third power of the Jacobi theta function. As an amusing application of these facts, one can derive Gauss' formula for the number of representations as sums of three squares in terms of class numbers of imaginary quadratic fields; compare \cite[Example~5]{williamseisenstein}.

\end{example}
	
	We remark that Williams \cite{williamseisenstein} also investigated the vector-valued Eisenstein series $E_{k,L}(\tau,s)$ for small weights $k \in \{\frac{1}{2},1,\frac{3}{2},2\}$. Moreover, the Fourier coefficients of the holomorphic Eisenstein series $E_{k,L}(\tau)$ for $k \geq 3/2$ can be computed numerically using Williams' powerful WeilRep program \cite{williamsweilrep}. For instance, one can check Example~\ref{example rank 3} using the sageMath code
	\begin{verbatim}
    from weilrep import *
    w = WeilRep(diagonal_matrix([-2,-2,-2]))
    print(w.cusp_forms_basis(3/2))
    print(w.eisenstein_series(3/2,prec=10))
    print(w.theta_series(prec=10))
\end{verbatim}

\subsection{Maass Eisenstein series}\label{section maass eisenstein series}

Now we turn to the construction of harmonic Maass Eisenstein series $\widetilde{E}_{\kappa,L^-}(\tau)$, which yield $\xi$-preimages of the holomorphic Eisenstein series $E_{k,L}(\tau)$ constructed above. As before, $(L,Q)$ denotes a positive definite even lattice of rank $r \geq 1$.

For $\kappa \in \frac{1}{2}\Z$ with $2\kappa+r \equiv 0 \pmod 4$ we put
\[
\widetilde{E}_{\kappa,L^-}(\tau,s) = \frac{1}{4}\sum_{(M,\phi) \in \widetilde{\Gamma}_\infty \backslash \Mp_2(\Z)}(v^s \e_0) |_{\kappa,L^-}(M,\phi),
\]
where the slash operator $ |_{\kappa,L^-}$ involves the dual Weil representation $\overline{\rho}_L$. The Eisenstein series converges for $\Re(s) > 1-\kappa/2$, transforms like a modular form of weight $\kappa$ for $\overline{\rho}_L$, and is an eigenform of the Laplace operator with
\[
\Delta_{\kappa}\widetilde{E}_{\kappa,L^-}(\tau,s) = s(1-\kappa-s) \widetilde{E}_{\kappa,L^-}(\tau,s).
\]
Moreover, if we put $\kappa =  2-k$, a direct computation show that
\[
\xi_{2-k} \widetilde{E}_{2-k,L^-}(\tau,s) = \overline{s}E_{k,L}(1-k+\overline{s}).
\]
Hence, for $k>2$, the special value
\[
\widetilde{E}_{2-k,L^-}(\tau) = \widetilde{E}_{2-k,L^-}(\tau,k-1)
\]
defines a harmonic Maass form of weight $2-k$ for the dual Weil representation $\overline{\rho}_L$ with
\[
\xi_{2-k}\widetilde{E}_{2-k,L^-}(\tau) = (k-1)E_{k,L}(\tau).
\]
Again, if $k\leq 2$ the analytic continuation of $\widetilde{E}_{2-k,L^-}(\tau,s)$ to $s = k-1$ in the case $k = r/2$ follows from the work of Rallis \cite{rallis}.

\begin{theorem}
	If $k = r/2$ with $r \geq 1$ then $\widetilde{E}_{2-k,L^-}(\tau,s)$ has an analytic continuation to $s = k-1$, and $\widetilde{E}_{2-k,L^-}(\tau) = \widetilde{E}_{2-k,L^-}(\tau,k-1)$ is a harmonic Maass form of weight $2-k$ for $\overline{\rho}_L$ with 
	\begin{align*}
	\xi_{2-k}\widetilde{E}_{2-k,L^-}(\tau) = (k-1)E_{k,L}(\tau).
	\end{align*}
\end{theorem} 

\begin{example}
	As in Example~\ref{example rank 1} we take the rank $1$ lattice $L = \Z$ with $Q(x)= x^2$ and view $\widetilde{E}_{3/2,L^-}(\tau)$ as a scalar-valued modular form for $\Gamma_0(4)$ in the Kohnen plus space. Then we obtain Zagier's weight $3/2$ Eisenstein series
	\[
	\frac{1}{4\pi}\widetilde{E}_{3/2,L^-}(\tau) = \sum_{n =0}^{\infty}H(n)q^n + \frac{1}{4\sqrt{\pi}}\sum_{n=1}^\infty n \Gamma\left(-\tfrac{1}{2},4\pi n^2 v \right)q^{-n^2} + \frac{1}{8\pi \sqrt{v}},
	\]
	where $H(0) = -\frac{1}{12}$ and $H(n)$ for $n > 0$ is the usual Hurwitz class number of discriminant $-n$, and $\Gamma(s,x) = \int_{x}^\infty e^{-t}t^{s-1}dt$ is the incomplete gamma function. This can be proved by plugging in $s = -1/2$ into the Fourier expansion of the non-holomorphic Eisenstein series computed in \cite[Proposition~3.2]{bruinierkuehn} (see also the proof of Theorem~\ref{eisenstein series fourier expansion} below for the Fourier expansion) and simplifying the coefficients as in \cite[Example~2]{bruinierkuss}.
\end{example}

\subsection{The Fourier expansion of the Maass Eisenstein series}

In this section we recall the Fourier expansion of the non-holomorphic Eisenstein series $\widetilde{E}_{\kappa,L^-}(\tau,s)$ computed in \cite{bruinierkuehn}, and specialize it to $\kappa = 1/2$ and lattices $L$ of rank $r = 3$. Let us introduce the relevant quantities.

For $\gamma \in L'/L$, $n \in \Z-Q(\gamma)$  and~$a\in \N$ we consider the representation number
\[
N_{\gamma,n}(a) = \#\{x \in L/aL \, : \, Q(x-\gamma) + n \equiv 0 \pmod a \}
\]
modulo $a$, and the corresponding $L$-series
\[
L_{\gamma,n}(s) = \sum_{a = 1}^\infty N_{\gamma,n}(a)a^{-s},
\]
which converges for $s\in \C$ with $\Re(s)\gg 0$ and has meromorphic continuation to $s\in \C$.
Note that $N_{\gamma,n}(a)$ is multiplicative in $a$, so we have an Euler product
\[
L_{\gamma,n}(s) = \zeta(s-r+1)\prod_{p}L_{\gamma,n}^{(p)}(p^{-s}),
\]
with the local Euler factors
\begin{equation}\label{eq:local_L_factor_Igusa}
 L_{\gamma,n}^{(p)}(X) = (1-p^{r-1} X)\sum_{m = 0}^{\infty}N_{\gamma,n}(p^m)X^m.   
\end{equation}
For $n \neq 0$ these local Euler factors can be simplified as follows. For $\gamma \in L'/L$ we let $d_\gamma \in \N$ be the order of $\gamma$ in $L'/L$. Then $2d_\gamma n$ is a non-zero integer, and for a prime $p$ we put
\begin{equation}\label{eq w_p}
 w_p = 1 + 2\nu_p(2d_\gamma n) \in \N,   
\end{equation}
where $\nu_p$ denotes the usual $p$-adic valuation on $\Q$. Then, for $\nu \geq w_p$ we have $N_{\gamma,n}(p^{\nu+1}) = p^{r-1}N_{\gamma,n}(p^\nu)$, which implies that for $n \neq 0$ the Euler factor $L_{\gamma,n}^{(p)}(X)$ becomes the polynomial
\begin{align}\label{local polynomial}
L_{\gamma,n}^{(p)}(X) = N_{\gamma,n}(p^{w_p})X^{w_p} + (1-p^{r-1}X)\sum_{\nu=0}^{w_p-1}N_{\gamma,n}(p^{\nu})X^\nu. 
\end{align}

The Fourier expansion of $\widetilde{E}_{\kappa,L^-}(\tau,s)$ has been computed in \cite[Section~3]{bruinierkuehn}. Specializing their results to $\kappa = 1/2$ and $s = 1/2$, we obtain the following Fourier expansion.

\begin{theorem}[\cite{bruinierkuehn}]\label{eisenstein series fourier expansion}
	Let $L$ be a positive definite even lattice of rank $r = 3$. We have the Fourier expansion
	\begin{align*}
	\widetilde{E}_{1/2,L^-}(\tau) &= \sum_{\gamma \in L'/L}\sum_{\substack{n \in \Z-Q(\gamma) \\ n \geq 0}} c^+(n,\gamma)q^n \e_\gamma  \\
	&\quad + \sqrt{y}\e_0 + \sum_{\gamma \in L'/L}\sum_{\substack{n \in \Z-Q(\gamma) \\ n < 0}} c^-(n,\gamma)\Gamma(1/2,4\pi|n|v)q^n\e_\gamma,
	\end{align*}
	with coefficients
	\begin{align*}
	c^+(0,\gamma) &= -\frac{2^{3/2}\cdot 3}{\sqrt{|L'/L|}\pi}\lim_{s \to 1/2}\left(\zeta(4s-1)\prod_{p \mid 2\det(L)}\frac{1}{1+p^{-1}}L_{\gamma,0}^{(p)}(p^{-1-2s})\right),
	\end{align*}
	and, for $n > 0$,
	\begin{align*}
	c^+(n,\gamma) &= -\frac{2^{3/2}\cdot 3}{\sqrt{|L'/L|}\pi} \\
	&\quad\times\begin{dcases}
	 L(\chi_{\Delta_0},1)\sigma_{\gamma,n} \prod_{p \mid 2\det(L)}\frac{1-\chi_{\Delta_0}(p)p^{-1}}{1-p^{-2}}p^{-2w_p}N_{\gamma,n}(p^{w_p}), & \text{if $\Delta$ is not a square},\\ 
	 \lim_{s\to 1/2}\bigg(\zeta(2s)\prod_{p \mid 2\det(L)}\frac{1}{1+p^{-1}}L_{\gamma,n}^{(p)}(p^{-1-2s}) \bigg), & \text{if $\Delta$ is a square}, 
	 \end{dcases}
	 \end{align*}
	 and, for $n < 0$,
	 \begin{align*}
	c^-(n,\gamma) &=  -\frac{2^{3/2}\cdot 3}{\sqrt{|L'/L|}\pi^{3/2}}L(\chi_{\Delta_0},1)\sigma_{\gamma,n}\prod_{p \mid 2\det(L)}\frac{1-\chi_{\Delta_0}(p)p^{-1}}{1-p^{-2}}p^{-2w_p}N_{\gamma,n}(p^{w_p}).
	\end{align*}
	Here we wrote $\Delta = 2d_\gamma^2 n \det(L) = \Delta_0 f^2$ with $\Delta_0$ a fundamental discriminant and $f \in \N$. Moreover, $w_p$ is defined by \eqref{eq w_p} and we put
	\[
	\sigma_{\gamma,n}  = \sum_{d \mid w}\mu(d)\chi_{\Delta_0}(d)d^{-1}\sigma_{-1}(\tfrac{w}{d})
	\]
	where we wrote $f = ww'$ with $\text{g.c.d.}(w,2\det(L)) = 1$ and all prime divisors of $w'$ appear in $2\det(L)$.
\end{theorem}

\begin{proof}
	Let $\kappa \in \Z+\frac{1}{2}$ be a half-integer with $2\kappa + r \equiv 0 \pmod 4$. By \cite[Proposition~3.2 and equation (3.13)]{bruinierkuehn}, the non-holomorphic Eisenstein series $\widetilde{E}_{\kappa,L^-}(\tau,s)$ has the Fourier expansion
	\[
	\widetilde{E}_{\kappa,L^-}(\tau,s) = \sum_{\gamma \in L'/L}\sum_{n \in \Z-Q(\gamma)}c_0(n,\gamma,s,v)e(nx)\e_\gamma,
	\]
	with Fourier coefficients $c_0(n,\gamma,s,v)$ given by
\[
\begin{dcases}
\delta_{0,\gamma}v^s + 2^{2-\kappa-2s}\pi v^{1-\kappa-s}\frac{\Gamma(\kappa+2s-1)}{\Gamma(\kappa+s)\Gamma(s)}\frac{(-1)^{(2\kappa+r)/4}}{\sqrt{|L'/L|}}\prod_{p}L_{\gamma,0}^{(p)}(p^{1-r/2-\kappa-2s}), & \text{if $n = 0$}, \\
\frac{(-1)^{(2\kappa+r)/4}2^\kappa \pi^{s+\kappa}|n|^{s+\kappa-1}}{\sqrt{|L'/L|}\Gamma(s+\kappa)}\cdot \frac{L(\chi_{\Delta_0},2s+\kappa-1/2)}{\zeta(4s+2\kappa-1)}\sigma_{\gamma,n}(2s+\kappa)\mathcal{W}_s(4\pi n v), & \text{if $n > 0$}, \\
\frac{(-1)^{(2\kappa+r)/4}2^\kappa \pi^{s+\kappa}|n|^{s+\kappa-1}}{\sqrt{|L'/L|}\Gamma(s)}\cdot \frac{L(\chi_{\Delta_0},2s+\kappa-1/2)}{\zeta(4s+2\kappa-1)}\sigma_{\gamma,n}(2s+\kappa)\mathcal{W}_s(4\pi n v), & \text{if $n < 0$},
\end{dcases}
\]
with
\[
\sigma_{\gamma,n}(s) = \prod_{p \mid \Delta}\frac{1-\chi_{\Delta_0}(p)p^{1/2-s}}{1-p^{1-2s}}L_{\gamma,n}^{(p)}(p^{1-r/2-s}),
\]
and with the usual Whittaker function
\[
\mathcal{W}_{s}(y) = |y|^{-\kappa/2}W_{\sgn(y)\kappa/2,(1-\kappa)/2-s}(|y|).
\]
Note that we multiplied the coefficients in \cite{bruinierkuehn} by $1/2$. The relevant special values of the Whittaker function are given by
\[
\mathcal{W}_{1-\kappa}(y) = \mathcal{W}_{0}(y)=\begin{cases}
e^{-y/2}, & \text{if $y > 0$}, \\
e^{-y/2}\Gamma(1-\kappa,|y|), & \text{if $y < 0$}.
\end{cases}
\]

We first rewrite the coefficients of index $n = 0$. For the ``generic'' primes $p \nmid 2\det(L)$ the local Euler factors $L_{\gamma,0}^{(p)}(p^{-s})$ can be computed as explained in \cite[Remark~22]{williamspoincare} (see also \cite[Section 3]{williamseisenstein}), and are given by
\[
L_{\gamma,0}^{(p)}(p^{-s}) = \frac{1-p^{r-1-2s}}{1-p^{r-2s}},
\]
which implies that we can write
\[
\prod_{p}L_{\gamma,0}^{(p)}(p^{1-r/2-\kappa-2s}) = \frac{\zeta(4s+2\kappa-2)}{\zeta(4s+2\kappa-1)}\prod_{p\mid 2\det(L)}\frac{1-p^{2-2\kappa-4s}}{1-p^{1-2\kappa-4s}}L_{\gamma,0}^{(p)}(p^{1-r/2-\kappa-2s}).
\]
If we plug in $r = 3$, $\kappa = 1/2$, and $s = 1/2$, this gives the coefficients of index $n = 0$ as stated above.

For $n \neq 0$ we can split the product over the primes $p \mid \Delta$ into the primes with $p\nmid 2\det(L)$ and the ones with $p \mid 2\det(L)$. Note that by \eqref{local polynomial} we have
\[
L_{\gamma,n}^{(p)}(p^{1-r}) = p^{(1-r)w_p}N_{\gamma,n}(p^{w_p}).
\] 
For $p \nmid 2\det(L)$ the representation numbers $N_{\gamma,n}(p^{w_p})$ were computed by Siegel (see \cite[Theorem~6(ii) with $\alpha=w_p$]{bruinierkuss}). Now we plug in $r = 3$, $\kappa = 1/2$, and $s = 1/2$, and obtain\footnote{In the notation of \cite[Theorem~6(ii)]{bruinierkuss} we have $\mathcal{D}=\Delta_0w'^2$, hence for primes $p \nmid 2\det(L)$ we get $\chi_{\mathcal{D}}(p)=\chi_{\Delta_0}(p)$.} as in the proof of \cite[Theorem~11]{bruinierkuss} that
\begin{align*}
\prod_{\substack{p \mid \Delta \\ p \nmid 2\det(L)}}\frac{1-\chi_{\Delta_0}(p)p^{-1}}{1-p^{-2}}L_{\gamma,n}^{(p)}(p^{-2}) &= \prod_{p \mid w}\left(\sigma_{-1}(p^{\nu_p(w)})-\chi_{\Delta_0}(p)p^{-1}\sigma_{-1}(p^{\nu_p(w)-1}) \right) \\
&= \sum_{d \mid w}\mu(d)\chi_{\Delta_0}(d)d^{-1}\sigma_{-1}(\tfrac{w}{d}) = \sigma_{\gamma,n}.
\end{align*}
Note that for $\Delta_0 = 1$ we have $\sigma_{\gamma,n} = 1$ and $\frac{1-\chi_{\Delta_0}(p)p^{-1}}{1-p^{-2}} = \frac{1}{1+p^{-1}}$. This gives the formula in the theorem.
\end{proof}

\begin{remark}\label{remark eisenstein series}
	\begin{enumerate}
		\item The coefficients with non-square $\Delta$ can be further simplified using Dirichlet's class number formula
		\[
		L(\chi_{\Delta_0},1) = \begin{cases}
		\frac{\log \varepsilon_{\Delta_0}}{\sqrt{\Delta_0}}h(\Delta_0), & \text{if $\Delta_0 > 0$}, \\
		\frac{2\pi}{\omega(\Delta_0)\sqrt{|\Delta_0|}}h(\Delta_0), & \text{if $\Delta_0 < 0$},
		\end{cases} 
		\]
		where $h(\Delta_0)$ denotes the class number of $\Q(\sqrt{\Delta_0})$, $\omega(\Delta_0) \in \{2,4,6\}$ is the number of units, and $\varepsilon_{\Delta_0}$ denotes the smallest totally positive unit larger than $1$.
		
  \item Since~$L$ is positive definite, there exists at least one prime~$p\mid 2\det(L)$ such that~$L\otimes \Q_p$ is anisotropic. If $\Delta=2d_\gamma^2 n \det(L)$ is a square (including $n = 0$), this implies $N_{\gamma,n}(p^m)=0$ for~$m$ large enough, hence~$L_{\gamma,n}^{(p)}(p^{-1-2s})$ vanishes at $s = 1/2$; see \eqref{eq:local_L_factor_Igusa}. The zero of this factor at $s = 1/2$ cancels out with the pole of the Riemann zeta function and contributes with a rational multiple of~$\log(p)$. Moreover, from~\eqref{local polynomial} it follows that~$L_{\gamma,n}^{(p)}(p^{-2})$ is rational when~$n\neq 0$ for all primes~$p$. 
  In general, it is well known that $L_{\gamma,n}^{(p)}(p^{-s})$ (including $n=0$) can be written in terms of the local Igusa zeta function associated with a quadratic polynomial, and hence, by a fundamental result of Igusa, it is a rational function in $p^{-s}$ (see, e.g., \cite{cowankatzwhite}). This implies that the limits $\lim_{s \to 1/2}(\ldots)$ appearing in the above coefficients are rational multiples of $\log(p)$ if $L_{\gamma,n}^{(p)}(p^{-1-2s})$ vanishes at $s = 1/2$, and equal to $0$ if there are two such primes.
	 	\item Summarizing, the coefficients $c^+(n,\gamma)$ are of the form
		\[
		c^+(n,\gamma) = \frac{\sqrt{2}}{\sqrt{|L'/L|}\pi} \times \text{rational number}  \times 
		\begin{cases} 
		\log(\varepsilon_{\Delta_0})/\sqrt{\Delta_0}, & \text{if $\Delta$ is not a square}, \\
		\log(p), & \text{if $\Delta$ is a square},
		\end{cases}
		\]
		for some prime $p \mid 2\det(L)$ in the square case.
		\item Using $\xi_{1/2}\widetilde{E}_{1/2,L^-} = \frac{1}{2}E_{3/2,L}$ and applying the $\xi$-operator term-wise to the above Fourier expansion, we see that the coefficients of the holomorphic Eisenstein series $E_{3/2,L}$ are essentially given by class numbers of imaginary quadratic fields. 
		\item Rhoades and Waldherr \cite{rhoadeswaldherr} constructed a scalar-valued weight $1/2$ harmonic Maass form which maps to $\theta^3$ under the $\xi$-operator. By choosing $L = \Z^3$ with $Q(x_1,x_2,x_3) = x_1^2 + x_2^2 + x_3^2$ and taking the $\e_0$-component of $\widetilde{E}_{1/2,L^-}(\tau)$, we recover the results of \cite{rhoadeswaldherr}.
		\end{enumerate}
\end{remark}

\subsection{The Siegel--Weil formula}

As before, we let $(L,Q)$ be a positive definite even lattice of rank $r\geq 1$. Moreover, we let $W = L \otimes \Q$ be the surrounding rational quadratic space, and we let $H = \SO(W)$ be its special orthogonal group, viewed as an algebraic group over~$\Q$. Denoting by $\A_f$ the ring of finite adeles of~$\Q$, we have the $\C[L'/L]$-valued theta function on $\H \times H(\A_f)$ defined by
\begin{equation}\label{eqn-theta-L}
\Theta_L(\tau,h) = \sum_{\gamma \in L'/L}\sum_{X \in h(L+\gamma)}e(Q(X)\tau)\e_\gamma.
\end{equation}
It is a holomorphic modular form of weight $k = r/2$ for the Weil representation $\rho_L$. 

\begin{remark}\label{remark theta functions genus}
For $h = 1$ this is the usual holomorphic theta function $\Theta_L(\tau)$ associated with $L$ defined by \eqref{eq usual holomorphic theta}. More generally, we can view $\Theta_L(\tau,h)$ as a theta function corresponding to a lattice in the same genus as $L$: If we write $h = (h_p)$ with $h_p \in H(\Q_p)$, and $L_p =  L\otimes \Z_p$, then $H(\A_f)$ acts on lattices in the genus of $L$ by $hL = \bigcap_p h_pL_p \cap W$. Since $(hL)'/(hL) \cong L'/L$ (non-canonically), we can view $\Theta_{hL}(\tau)$ as a modular form for $\rho_L$.  Moreover, if we denote by $\O(L'/L)$ the orthogonal group of the finite quadratic module $(L'/L,Q \text{ mod }\Z)$, then there is an isometry $\sigma \in \O(L'/L)$, depending on the choice of the isomorphism $(hL)'/(hL) \cong L'/L$, such that 
\[
\Theta_L(\tau,h) = \Theta_{hL}^\sigma(\tau),
\]
where $\O(L'/L)$ acts on $\C[L'/L]$-valued functions by $f^{\sigma}(\tau) = \sum_{\gamma}f_\gamma(\tau)\e_{\sigma(\gamma)}$. This means that, up to a possible permutation of the components, $\Theta_L(\tau,h)$ can be viewed as the theta function corresponding to the lattice $hL$. 
\end{remark}

We state the Siegel--Weil formula in our setup. We refer the reader to \cite[Theorem~4.1]{kudla} for the general formula. It can be translated into our notation as explained in \cite[Section~2.2]{bruinieryang}.

\begin{theorem}[Siegel--Weil]\label{siegel weil formula}
	Let $L$ be a positive definite even lattice of rank $r \geq 1$, and let $k = r/2$. Then we have
	\[
	\int_{H(\Q)\backslash H(\A_f)}\Theta_L(\tau,h)dh = 2E_{k,L}(\tau),
	\]
	where $dh$ is the Tamagawa measure on $H(\A_f)$ normalized such that $\vol(H(\Q)\backslash H(\A_f)) = 2$, and $E_{k,L}(\tau)$ is the holomorphic Eisenstein series of weight $k$ for $\rho_L$ defined in Section~\ref{section holomorphic eisenstein series}.
\end{theorem}

The following well known result (see, e.g., \cite[Lemma~2.13]{schofer}) allows us to write the integral in the Siegel--Weil formula as a finite sum. We give a proof for the convenience of the reader.

\begin{lemma}\label{sum lemma}
	Let $K \subset H(\A_f)$ be a compact open subgroup and let $f(h)$ be a  function on $H(\Q) \backslash H(\A_f)$ which is $K$-invariant from the right. Then we have
	\[
	\int_{H(\Q)\backslash H(\A_f)}f(h)dh = \vol(K)\sum_{\alpha \in H(\Q)\backslash H(\A_f)/K}\frac{1}{|\Gamma_{\alpha}|}f(\alpha)
	\]
	where $\Gamma_{\alpha} = H(\Q) \cap (\alpha K \alpha^{-1})$ and 
	\[
	\vol(K) = 2\left(\sum_{\alpha \in H(\Q)\backslash H(\A_f)/K}\frac{1}{|\Gamma_{\alpha}|}\right)^{-1}.
	\]
\end{lemma}

\begin{proof}
	We have a finite disjoint double coset decomposition
	\[
	H(\A_f) = \bigcup_{\alpha \in H(\Q)\backslash H(\A_f)/K}H(\Q)\alpha K,
	\]
	so we can write the integral as
	\[
	\int_{H(\Q)\backslash H(\A_f)}f(h)dh = \sum_{\alpha \in H(\Q)\backslash H(\A_f)/K}\int_{\Gamma_\alpha \backslash\alpha K}f(h)dh.
	\]
	 Since $\Gamma_\alpha$ is finite and $f$ is invariant under $K$, we have
	\[
	\int_{\Gamma_\alpha \backslash\alpha K}f(h)dh = \frac{1}{|\Gamma_\alpha|}\int_{\alpha K}f(h)dh = \vol(K)\frac{1}{|\Gamma_{\alpha}|}f(\alpha).
	\]
	This gives the stated formula. In order to compute $\vol(K)$, apply the lemma to $f(h) = 1$. 
 \end{proof}

	\begin{example}
		Let us rewrite the integral in the Siegel--Weil formula as a finite sum over lattices in the genus of $L$ of symmetrized theta functions. We choose as $K$ the compact open subgroup
		\[
		\widetilde{U} = \{h \in H(\A_f): hL = L\},
		\]
		the stabilizer of $L$ in $H(\A_f)$. The map sending $\alpha\in H(\Q)\backslash H(\A_f)/\widetilde{U}$ to the isometry class $M$ of $\alpha L$ is a bijection between $H(\Q)\backslash H(\A_f)/\widetilde{U}$ and the set~$\gen(L)$  of isometry classes of lattices in the genus of $L$. Moreover, we have $\Gamma_\alpha = \SO(\alpha L)= \SO(M)$, where $\SO(M)=H(\Q)\cap \widetilde{U}$. However, we cannot directly apply Lemma~\ref{sum lemma} since $\widetilde{U}$ might act non-trivially on $L'/L$, so $\Theta_{L}(\tau,h)$ is not generally $\widetilde{U}$-invariant. To circumvent this problem, we note that the symmetrized theta function 
  $$\Theta_{L}^{\mathrm{sym}}(\tau,h)=\frac{1}{|\O(L'/L)|}\sum_{\sigma \in \O(L'/L)}\Theta_{L}^{\sigma}(\tau,h)$$
  is $\widetilde{U}$-invariant, so by Lemma~\ref{sum lemma} and Remark~\ref{remark theta functions genus} we can write
  \begin{align*}
      \int_{H(\Q)\backslash H(\A_f)}\Theta^{\mathrm{sym}}_L(\tau,h)dh & = \vol(\widetilde{U})\sum_{\alpha \in H(\Q)\backslash H(\A_f)/\widetilde{U}}\frac{1}{|\Gamma_{\alpha}|}\Theta_L^{\mathrm{sym}}(\tau,\alpha) \\
      &  =
  \frac{2}{\mathrm{mass}(L)}\sum_{M \in \gen(L)}\frac{1}{|\SO(M)|}
  \Theta_{M}^{\mathrm{sym}}(\tau),
  \end{align*}
  where $\mathrm{mass}(L) = \sum_{M \in \gen(L)}\frac{1}{|\SO(M)|}$ is the mass of the genus of $L$. Now by the Siegel--Weil formula and the symmetry of the Eisenstein series $E_{k,L}$ we obtain 
		\begin{align}\label{siegel weil finite sum}
   \frac{1}{\mathrm{mass}(L)}\sum_{M \in \gen(L)}
  \Theta_{M}^{\mathrm{sym}}(\tau)
  = E_{k,L}(\tau),
		\end{align}
  where $k = r/2$. Here we had identified $M'/M \cong L'/L$ for each~$M\in  \gen(L)$, but  the sum on the left-hand side is independent of the choice of such isomorphisms due to the symmetrization over $\O(L'/L)$.
  	
		Finally, we remark that the classical Siegel--Weil formula is usually stated for the scalar-valued theta function $\theta_L(\tau) = \sum_{x \in L}e(Q(x)\tau)$, which is just the $\e_0$-component of $\Theta_L(\tau)$. If we take the $\e_0$-component on both sides of \eqref{siegel weil finite sum}, and use that any $\sigma \in \O(L'/L)$ fixes $\e_0$, we obtain the scalar-valued Siegel--Weil formula
		\[
		\frac{1}{\mathrm{mass}(L)}\sum_{M \in \gen(L)}\frac{1}{|\SO(M)|}\theta_{M}(\tau) = E_{k,L,0}(\tau),
		\]
		where $E_{k,L,0}(\tau)$ denotes the $\e_0$-component of $E_{k,L}(\tau)$, which is a scalar-valued Eisenstein series for some congruence subgroup.
	\end{example}

\section{Traces of special values of the Green's function}\label{section traces green's function}

\subsection{The trace of the Green's function as a theta lift}\label{section theta lift}

In this section we show that the traces of the Green's function $G_{s}(P_1,P_2)$ defined in \eqref{eq Green's function} can be obtained as a regularized theta lift on a lattice of signature $(1,3)$. 

\subsubsection{The orthogonal model of hyperbolic $3$-space}\label{sec orthogonal model}

	We fix an imaginary quadratic field $\Q(\sqrt{D})$ of fundamental discriminant $D < 0$. We let $\mathcal{O}_D$ be its ring of integers and $\mathfrak{d}_D = \sqrt{D}\mathcal{O}_D$ be its different. We consider the lattice
\begin{align}\label{lattice}
	L = \left\{ X= \begin{pmatrix}a & b \\ \overline{b} & c \end{pmatrix} \, : \, a,c \in \Z, \, b \in \mathcal{O}_D\right\}
\end{align}
	with the quadratic form
	\[
	Q(X) = \det(X) = ac - |b|^2
	\]
	and corresponding bilinear form
	\[
	(X_1,X_2) = a_1 c_2 + a_2 c_1 - \tr(b_1 \overline{b}_2).
	\]
	It is an even lattice of signature $(1,3)$, with dual lattice
	\[
L' = \left\{ X= \begin{pmatrix}a & b \\ \overline{b} & c \end{pmatrix} \, :  \, a,c \in \Z, \, b \in \mathfrak{d}_D^{-1}\right\}.
	\]
	In particular, we have 
	\[
	L'/L \cong \mathfrak{d}_D^{-1}/\mathcal{O}_D. 
	\]
	The group $\Gamma = \PSL_2(\mathcal{O}_D)$ acts on $L$ via $\gamma.X = \gamma X \overline{\gamma}^t$, fixing the classes of $L'/L$. Note that we may identify the elements $X = \left(\begin{smallmatrix}a & b \\ \overline{b} & c \end{smallmatrix} \right) \in L'$ with integral binary hermitian forms 
	\[
	[a,b,c](x,y) = a|x|^2 + \tr(bx\overline{y}) + c|y|^2
	\]
	of determinant $ac - |b|^2 = Q(X)$, and this identification is compatible with the corresponding actions\footnote{$\Gamma$ acts on hermitian forms via $(\gamma [a,b,c]) (x,y)=[a,b,c]( (x,y)\cdot \overline{\gamma})$.} of $\Gamma$.
	
	The Grassmannian $\Gr(L)$ corresponding to $L$ is the set of all positive definite lines in $V(\R) = L \otimes \R$. We can identify the hyperbolic $3$-space $\H^3$ with $\Gr(L)$ by mapping $P = z+rj \in \H^3$ to the positive line $\R v(P) \in \Gr(L)$ spanned by the vector
\begin{align}\label{eq identification}
v(P) = \frac{1}{\sqrt{2}r}\begin{pmatrix}r^2+|z|^2 & z \\ \overline{z} & 1\end{pmatrix} \in V(\R).
\end{align}
Conversely, each positive line in $\Gr(L)$ is spanned by some vector $X = \left(\begin{smallmatrix}a & b \\ \overline{b} & c \end{smallmatrix} \right) \in V(\R)$ with $Q(X) = \det(X) > 0$. If we put
\begin{align}\label{eq point for X}
P_X = \frac{b}{c} + \frac{\sqrt{\det(X)}}{|c|}j \in \H^3,
\end{align}
then one can check that $v(P_X) = \sgn(c)\sqrt{2\det(X)}^{-1}X$, so $v(P_X)$ spans the given positive line. We call $P_X$ the \emph{point corresponding to $X$}, and if $X \in V(\Q)$ is a rational vector, we call $P_X$ a \emph{special point}.

Recall that the hyperbolic distance $d(P_1,P_2)$ on $\H^3$ is given by
\[
\cosh(d(P_1,P_2)) = \frac{|z_1-z_2|^2 + r_1^2 + r_2^2}{2r_1 r_2}, \qquad \text{ where }P_i=z_i+r_ij \text{ for }i\in \{1,2\}.
\]

The bilinear form on $V(\R)$ can be expressed in terms of the hyperbolic distance as follows.

\begin{lemma}\label{lemma hyperbolic distance}
	Given $X_1,X_2 \in V(\R)$ with $D_1 = Q(X_1) > 0$ and $D_2 = Q(X_2) > 0$, and corresponding points $P_{1},P_{2} \in \H^3$, we have
	\[
	(X_1,X_2) = 2\, \sgn(c_1 c_2)\sqrt{D_1 D_2}\cdot\cosh(d(P_{1},P_{2})).
	\]
\end{lemma}

\begin{proof}
	By definition we have
	\[
	(X_1,X_2) = a_1 c_2 + a_2 c_1 - \tr(b_1 \overline{b}_2).
	\]
	One the other hand, using the explicit formula for the points $P_1,P_2$ from \eqref{eq point for X}, we have
	\begin{align*}
	\cosh(d(P_{1},P_{2})) &= \frac{\left|\frac{b_1}{c_1}-\frac{b_2}{c_2}\right|^2+\frac{D_1}{c_1^2}+\frac{D_2}{c_2^2}}{2\frac{\sqrt{D_1}}{|c_1|}\frac{\sqrt{D_2}}{|c_2|}}
	= \sgn(c_1 c_2)\frac{a_1 c_2 + a_2 c_1 - \tr(b_1 \overline{b}_2)}{2\sqrt{D_1 D_2}}.
	\end{align*}
	This gives the stated formula.
\end{proof}

\subsubsection{Traces of $\Gamma$-invariant functions on $\H^3$}\label{section traces} For $\mu \in L'/L$ and $m \in \Z+Q(\mu)$ we let
\[
L_{m,\mu} = \{X \in L + \mu, Q(X) = m\}.
\]
If $m > 0$, we let $L_{m,\mu}^+$ the subset of those $X  =\left(\begin{smallmatrix}a & b \\ \overline{b} & c \end{smallmatrix} \right) \in L_{m,\mu}$ with $a > 0$ (or, equivalently, $c > 0$). Note that the elements $X \in L_{m,\mu}^+$ correspond to positive definite binary hermitian forms $[a,b,c]$ of determinant $m$ with $b \in \mathcal{O}_D + \mu$.
For $m > 0$ and a $\Gamma$-invariant function $f$ on $\H^3$, we define its \emph{trace of index $(m,\mu)$} by
\[
\tr_{m,\mu}(f) = \sum_{X \in \Gamma \backslash L_{m,\mu}^+}\frac{1}{|\Gamma_X|}f(P_X).
\]
Here $\Gamma_X$ denotes the stabilizer of $X$ in $\Gamma$. Clearly, we have
\begin{equation*}\label{eq-fulltrace}
\tr_{m}(f) = \sum_{X \in \Gamma \backslash L_{m}^+}\frac{1}{|\Gamma_X|}f(P_X) =\sum_{\mu\in L'/L}\tr_{m,\mu}(f),    
\end{equation*}
where $L_{m}^+$ is defined in \eqref{eq:L_m^+}.

We call $X \in L'$ \emph{primitive} if $\frac{1}{r}X \notin L'$ for every integer $r > 1$, and we let $L_{m,\mu}^{+,0}$ be the subset of $L_{m,\mu}^{+}$ consisting of primitive vectors. Correspondingly, we define the \emph{primitive trace of index $(m,\mu)$} by
\[
\tr_{m,\mu}^0(f) = \sum_{X \in \Gamma \backslash L_{m,\mu}^{+,0}}\frac{1}{|\Gamma_X|}f(P_X).
\]
Each $X \in L_{m,\mu}^{+}$ can be written in a unique way as $X = rX_0$ for some $r \in \N$ and a primitive $X_0 \in L'$ with $a > 0$. Then we have $P_{X} = P_{X_0}$ and $\Gamma_X = \Gamma_{X_0}$. Since $DQ(X) \in \Z$ for any $X \in L'$, we must have $r^2 \mid Dm$. Moreover, if $X_0 \in L+\nu$, then we necessarily have $r\nu = \mu$ in $L'/L$. Hence, we can write
\[
\tr_{m,\mu}(f) = \sum_{\substack{r\in \N \\ r^2 \mid  Dm}}\sum_{\substack{\nu \in L'/L \\  r\nu = \mu}}\tr_{m/r^2,\nu}^0(f).
\]
Conversely, using an inclusion-exclusion argument, we see that the primitive trace $\tr_{m,\mu}^0(f)$ can be written as an integral linear combination of the traces $\tr_{m/r^2,\nu}(f)$, for $r\in \N$ with $r^2 \mid Dm$ and $\nu \in L'/L$ with $r\nu = \mu$.

\subsubsection{Green's functions and theta lifts}

We let $\Theta_L(\tau,v)$ be the Siegel theta function on $\H \times \Gr(L)$ as in Section~\ref{section theta functions}. Using the identification $\H^3 \cong \Gr(L), P \mapsto v = \R v(P),$ given by \eqref{eq identification}, we can view 
\[
\Theta_L(\tau,P) = \Theta_L(\tau,\R v(P))
\]
as a function on $\H \times \H^3$. It is $\Gamma$-invariant in $P$, and since $L$ has signature $(1,3)$, it transforms like a modular form of weight $-1$ for the Weil representation $\rho_L$ in $\tau$.

Let $n\in \N_0$. For a harmonic Maass form $f \in H_{1-2n,L^-}^{\cusp}$ of weight $1-2n$ for the dual Weil representation $\overline{\rho}_L$, we define the regularized theta lift
\[
\Phi_L^{(1-2n)}(f,P) = \int_{\mathcal{F}}^{\mathrm{reg}} \left(R_{1-2n}^{n}f\right)(\tau) \cdot \Theta_L(\tau,P) d\mu(\tau),
\]
where the integral over the fundamental domain $\mathcal{F}$ for $\SL_2(\Z)$ is regularized as explained by Borcherds \cite{borcherds} or Bruinier \cite{bruinier}, the product in the integral is the bilinear pairing on $\C[L'/L]$, and $d\mu(\tau)$ is the invariant measure on $\H$.

For $n\in \N_0$, $\mu \in L'/L$ and $m \in \Z + Q(\mu)$ with $m > 0$ we let $F_{1-n,m,\mu}(\tau,s)$ be the weight $1-n$ Maass Poincar\'e series for $\overline{\rho}_L$ defined in Section~\ref{section harmonic maass forms} (for the lattice $L^-$).

\begin{theorem}\label{proposition higher Green's function as theta lift}
	For $n\in \N_0$, $\mu \in L'/L$ and $\mu \in \Z + Q(\mu)$ with $m > 0$ and $\Re(s) > 1$ we have
	\[
	\Phi_L^{(1-2n)}\left(F_{1-2n,m,\mu}(\,\cdot\,,s),P\right) = C_n(s)\frac{2}{\pi\Gamma(s+1/2)\sqrt{m}}\tr_{m,\mu}\big( G_{2s-1}(\, \cdot \,,P)\big),
	\]
	with the constant $C_n(s) = (4\pi m)^{n}(s+1/2-n)(s+1/2-n+1)\cdots (s-1/2)$ for $n>0$ and $C_0(s)=1$. 
\end{theorem}

\begin{proof}
	Using Lemma~\ref{lemma raising poincare series} we see that
	\[
	\Phi_L^{(1-2n)}\left(F_{1-2n,m,\mu}(\,\cdot\,,s),P\right) = C_n(s)\Phi_L^{(1)}\left(F_{1,m,\mu}(\,\cdot\,,s),P\right),
	\]
	so it suffices to show the result for $n = 0$. Then the theta lift can be computed by unfolding against the Poincar\'e series as in \cite[Theorem~2.14]{bruinier}. We need to apply this theorem to the lattice $L^- = (L,-Q)$, which has signature $(3,1)$, and with weight $k = \frac{b^+-b^-}{2} = 1$. Then we obtain
	\[
	\Phi_L^{(1)}(F_{1,m,\mu}(\,\cdot\,,s),P) = \frac{2\Gamma(s)}{\Gamma(2s)\cdot (4\pi m)^{1/2-s}}\sum_{\substack{X \in L+\mu \\ Q(X) = m}}(4\pi Q(X_v))^{-s} \ _2 F_1\left(s,s+\frac{1}{2},2s; \frac{m}{Q(X_v)}\right),
	\]
	where we put $v = \R v(P)$ for brevity. Since $(v(P),v(P)) = 1$, we have
	\[
	Q(X_v) = Q\left(\frac{(X,v(P))}{(v(P),v(P))}v(P)\right) = \frac{1}{2}(X,v(P))^2.
	\]
	Hence, using Lemma~\ref{lemma hyperbolic distance}, we can write $Q(X_v) = m \, \cosh(d(P,P_X))^2$.	Moreover, the hypergeometric function simplifies to
	\begin{equation}\label{eq simplification of hypergeometric function}
	    \ _2 F_1\left(s,s+\frac{1}{2},2s; x\right) = 2^{2s-1}(\sqrt{1-x}+1)^{1-2s} (1-x)^{-1/2}. 
	\end{equation}
	Combining these two facts, we obtain after a short computation
	\begin{align}\label{eq hypergeometric equals varphi}
	\ _2 F_1\left(s,s+\frac{1}{2},2s; \frac{m}{Q(X_v)}\right)
	&= 2^{2s-1}\cosh(d(P,P_X))^{2s}\varphi_{2s-1}(\cosh(d(P,P_X))),
	\end{align}
 where $\varphi_s$ is defined by \eqref{eq:varphi_s}. Taking everything together, we arrive at
	\[
	\Phi_L^{(1)}(F_{1,m,\mu}(\,\cdot\,,s),P) = \frac{2^{2s}\Gamma(s)}{\Gamma(2s)\cdot (4\pi m)^{1/2}}\sum_{\substack{X \in L+\mu \\ Q(X) = m}}\varphi_{2s-1}(\cosh(d(P_X,P))).
	\]
	By the duplication formula for the Gamma function we have $\frac{2^{2s}\Gamma(s)}{\Gamma(2s)} = \frac{2\sqrt{\pi}}{\Gamma(s+1/2)}$. Splitting the sum modulo $\Gamma$, and replacing $X$ with $-X$ if $a < 0$, gives the stated formula.
\end{proof}

Next we recall that for~$n\geq 1$ the function $F_{1-2n,m,\mu}(\tau) = F_{1-2n,m,\mu}\left(\tau,n+\frac{1}{2}\right)$ defines a harmonic Maass form of weight $1-2n$ for $\overline{\rho}_L$ with principal part $q^{-m}(\e_\mu + \e_{-\mu})$. As an application  of Theorem~\ref{proposition higher Green's function as theta lift} we obtain the next theorem which is the main result of this section. 

\begin{theorem}\label{theorem higher Green's function as theta lift}
	Let $\mu \in L'/L$ and $m \in \Z + Q(\mu)$ with $m > 0$. For $n \geq 1$ we have
	\[
	\Phi_L^{(1-2n)}\left(F_{1-2n,m,\mu},P\right) = (4\pi m)^n\frac{2}{\pi \sqrt{m}}\tr_{m,\mu}\big(G_{2n}(\,\cdot\,,P)\big).
	\]
\end{theorem}

\begin{proof}
One can show, as in  the proof of \cite[Proposition~2.11]{bruinier}, that
	\begin{align}\label{eq plug in s0}
	\Phi_L^{(1-2n)}\left(F_{1-2n,m,\mu}(\,\cdot\,,s),P\right)\big|_{s = n+\frac{1}{2}} = \Phi_L^{(1-2n)}\left(F_{1-2n,m,\mu}\left(\,\cdot\,,n+\frac{1}{2}\right),P\right).
	\end{align}
 Here we use that $n \geq 1$, so $F_{1-2n,m,\mu}(\,\cdot\,,s)$ converges at $s=n+\frac{1}{2}$. Now the theorem follows from~\eqref{eq plug in s0} and Theorem~\ref{proposition higher Green's function as theta lift}.
\end{proof}

\begin{remark}
	It would be interesting to extend Theorem~\ref{theorem higher Green's function as theta lift} to $n = 0$. However, one would need to show that $F_{1,m,\mu}(\,\cdot\,,s)$ has an analytic continuation to $s = 1/2$, and that~\eqref{eq plug in s0} still holds, which seem to be difficult problems.
\end{remark}

\subsection{Splitting of the Siegel theta function at special points}

\subsubsection{The special orthogonal group}\label{sec orthogonal and spinor groups} We let $L$ be the lattice of integral binary hermitian forms over $\Q(\sqrt{D})$ defined in \eqref{lattice}, with the quadratic form $Q(X) = \det(X)$. We let 
\[
V = L \otimes \Q = \left\{X = \begin{pmatrix}a & b \\ \overline{b}& c \end{pmatrix} : a,c \in \Q,  b\in \Q(\sqrt{D})\right\}
\]
be the surrounding rational quadratic space, and $G = \SO(V)$ the corresponding special orthogonal group. As usual, there is a short exact sequence
$$1\to \mathbb{G}_m \to \mathrm{GSpin}_V\to G\to 1$$
of algebraic groups over~$\Q$, where $\mathrm{GSpin}_V$ denotes the general spinor group associated to $V$, and $\mathbb{G}_m$ is the multiplicative group.

For an integral domain $R$ of characteristic zero we  consider the commutative ring $\mathcal{O}_D\otimes R=R\oplus R\omega_D$, where $\omega_D = \frac{D+\sqrt{D}}{2}$. On $\mathcal{O}_D\otimes R$ there is a unique $R$-linear involution denoted by $z\mapsto \overline{z}$ satisfying $\overline{\sqrt{D}} = -\sqrt{D}$, which induces an involution on the group $\GL_2(\mathcal{O}_D\otimes R)$. Note that $\mathcal{O}_D\otimes R$ is an integral domain if and only if $D$ is not a square in $R$. 

A  study of the Clifford algebra associated to $(V,Q)$ (see, e.g., \cite[Chapter~1]{shimurabook}) shows that, for any field extension $F$ of $\Q$, we have
\begin{equation}\label{eq GSpin(F) description}
\mathrm{GSpin}_V(F) \cong \{g \in \GL_2(\mathcal{O}_D\otimes F): \det(g) \in F^\times\}, 
\end{equation}
with spinor norm $\nu$ corresponding to $\mathrm{det}(g)$. Moreover, the action of $\mathrm{GSpin}_V(F)$ on $V\otimes F$ corresponds to
\[
g.X = \frac{1}{\det(g)}gX\overline{g}^t.
\]
Hence, we get
\begin{equation}\label{eq: description of G(F)}
G(F) \cong  \mathrm{GSpin}_V(F)/F^\times \cong \{g \in \GL_2(\mathcal{O}_D\otimes F): \det(g) \in F^\times\}/F^{\times}.
\end{equation}

We now consider the compact open subgroup
\[
\widetilde{U}=\{g \in G(\A_f) \, : \, gL = L\}\subset G(\A_f)
\]
and its subgroup
\[
U = \{g \in \widetilde{U} \, : \, \text{$g$ fixes the classes in $L'/L$} \}.
\]
Note that $\widetilde{U}=\prod_p\widetilde{U}_p$ and $U=\prod_pU_p$ where
\[
\widetilde{U}_p=\{g_p \in G(\mathbb{Q}_p) \, : \, g_pL_p = L_p\}  \text{ and }
U_p = \{g_p \in \widetilde{U}_p \, : \, \text{$g_p$ fixes the classes in $L'_p/L_p$}\},
\]
with $L_p=L\otimes \Z_p$ and $L'_p=L'\otimes \Z_p$.

For the rest of this section we assume that $D<0$ is a prime discriminant, and denote by $\ell$ the unique prime dividing $D$. We then define
 $$z_D=\begin{cases}
    1+\sqrt{-1}, & \text{if $D=-4$},  \\
   \sqrt{-2}, & \text{if $D=-8$},  \\
   \sqrt{-\ell}, & \text{if $D=-\ell, \ell\equiv 3$ (mod 4)},
\end{cases}$$
and put $g_D=\left(\begin{smallmatrix}
    z_D& 0 \\
       0  & \overline{z_D}
\end{smallmatrix}\right)$ as an element of $\mathrm{GSpin}_V(\mathbb{Q})$ according to \eqref{eq GSpin(F) description}. We let $T_D$ denote the image of $g_D$ under the natural map $\mathrm{GSpin}_V(\mathbb{Q})\to G(\mathbb{Q})$. A straightforward computation shows that $T_D\in \widetilde{U}_{\ell}$ and that $T_D$ induces an automorphism of order two of $L'_{\ell}/L_{\ell}$.

\begin{lemma}\label{l:Dp_and_Up}
For every prime $p$ we have
$$U_p\cong  \{g \in \GL_2(\mathcal{O}_D\otimes \Z_p): \det(g) \in \Z_p^\times\}/\Z_p^{\times}.$$
Moreover, we have $\widetilde{U}_p= U_p$ if $p\neq \ell$, and $\widetilde{U}_{\ell}=
 U_{\ell}\cup T_DU_{\ell}$.
\end{lemma}
\begin{proof}
It is easy to check that every element in $\{g \in \GL_2(\mathcal{O}_D\otimes \Z_p): \det(g) \in \Z_p^\times\}$ acts on $V(\mathbb{Q}_p)$ as a transformation in $U_p$. In order to prove the converse, let $T\in U_p$. By \eqref{eq: description of G(F)} we have that $T$ is induced by the action of a matrix $g\in \GL_2(\mathcal{O}_D\otimes \Q_p)$ with $\det(g) \in \Q_p^\times$. By multiplying $g$ by an appropriate power of $p$ we can assume that $g\in M_2(\mathcal{O}_D\otimes \Z_p)\setminus pM_2(\mathcal{O}_D\otimes \Z_p)$. Then $\mathrm{det}(g)\in p^{m}\Z_p^{\times}$ for some integer $m \geq 0$. We claim that $m=0$. Indeed, using the $\Z_p$-basis 
$$\left\{ \begin{pmatrix}
    1 & 0 \\
    0 &  0
\end{pmatrix},\begin{pmatrix}
    0 & 0 \\
    0 &  1
\end{pmatrix},\begin{pmatrix}
    0 & 1 \\
    1 &  0
\end{pmatrix},\begin{pmatrix}
    0 & \omega_D \\
    \overline{\omega_D} &  0
\end{pmatrix} \right\}$$
of $L_p$, we see that $gL_pg^{-1}$ is generated over $\Z_p$ by
    $$\left\{ \begin{pmatrix}
        \alpha\overline{\alpha} & \alpha \overline{\gamma} \\
       \gamma \overline{\alpha}  & \gamma \overline{\gamma}
    \end{pmatrix},\begin{pmatrix}
        \beta\overline{\beta} & \beta \overline{\delta} \\
       \delta \overline{\beta}  & \delta \overline{\delta}
    \end{pmatrix},
    \begin{pmatrix}
        \alpha\overline{\beta}+\beta\overline{\alpha} & \beta\overline{\gamma}+\alpha \overline{\delta} \\
       \gamma\overline{\beta}+\delta \overline{\alpha}  & \delta\overline{\gamma}+\gamma\overline{\delta}
    \end{pmatrix},
    \begin{pmatrix}
        \omega_D\alpha\overline{\beta}+\overline{\omega_D}\beta\overline{\alpha} & \overline{\omega_D}\beta\overline{\gamma}+\omega_D\alpha \overline{\delta} \\
       \omega_D\gamma\overline{\beta}+\overline{\omega_D}\delta \overline{\alpha}  & \overline{\omega_D}\delta\overline{\gamma}+\omega_D\gamma\overline{\delta}
    \end{pmatrix} \right\}.
    $$
Since $gL_pg^{-1}=\mathrm{det}(g)L_p$, if follows that these four matrices have all their coefficients in $p^m\mathcal{O}_D\otimes \Z_p$. We distinguish three cases according to the value of $\left(\frac{D}{p}\right)$. First, assume $\left(\frac{D}{p}\right)=-1$. Then $\mathcal{O}_D\otimes \Z_p=\Z_p[\omega_D]$ is a quadratic unramified extension of $\Z_p$ with maximal ideal $p\Z_p[\omega_D]$. Since $z\overline{z}\in p^m\Z_p[\omega_D]$ for all $z\in\{\alpha,\beta,\gamma,\delta\}$, we have $\alpha,\beta,\gamma,\delta \in  p^m\Z_p[\omega_D]$. Since $g\not \in pM_2(\Z_p[\omega_D])$ we get $m=0$. Now, assume $\left(\frac{D}{p}\right)=0$, i.e.~$p=\ell$. Then $\Z_p[\omega_D]$ is a quadratic ramified extension of $\Z_p$ with maximal ideal $z_D\Z_p[\omega_D]$. Since $z\overline{z}\in z_D^{2m}\Z_p[\omega_D]$ for all $z\in\{\alpha,\beta,\gamma,\delta\}$, we have $\alpha,\beta,\gamma,\delta \in  z_D^m\Z_p[\omega_D]=\overline{z_D}^m\Z_p[\omega_D]$. But then $g=g_D^mh$ for some matrix $h\in \mathrm{GL}_2(\Z_p[\omega_D])$ with $\mathrm{det}(h)\in \Z_p^{\times}$. Since $g$ and $h$ act as transformations in $U_{\ell}$ and $g_D$ does not, we must have $m$ even. This implies  $\alpha,\beta,\gamma,\delta\in p^{m/2}\Z_p[\omega_D]$. Since $g\not \in pM_2(\Z_p[\omega_D])$ we get $m=0$, as before. Finally, assume $\left(\frac{D}{p}\right)=1$. Using that $p\nmid D$ we deduce from the generators of $gL_pg^{-1}$ computed above, that $z\overline{w}\in p^m\mathcal{O}_D\otimes \Z_p$ for all $z,w\in \{\alpha,\beta,\gamma,\delta\}$. Since $\left(\frac{D}{p}\right)=1$, there exists $\zeta \in \Z_p$ such that $\zeta^2=D$, and the map $p^m\mathcal{O}_D\otimes \Z_p\to \Z_p\times \Z_p$ given by 
\begin{equation}\label{ring_isomorphism}
    z=x+y\omega_D\mapsto (z_1,z_2)=\left(x+y\frac{D+\zeta}{2},x+y\frac{D-\zeta}{2}\right)
\end{equation}
is a ring isomorphism. We then have $(z_1w_2,z_2w_1)\in p^m( \Z_p\times \Z_p)$ for all $z,w\in \{\alpha,\beta,\gamma,\delta\}$. If $z_i\in \Z_p^{\times}$ for some index $i\in \{1,2\}$ and some $z\in  \{\alpha,\beta,\gamma,\delta\}$, then we get $w_j\in p^m\Z_p$ for the other index $j\in \{1,2\},j\neq i$, for all $w\in \{\alpha,\beta,\gamma,\delta\}$. In particular $\alpha_j\delta_j-\beta_j\gamma_j\in p^{2m}\Z_p$. But $\mathrm{det}(g)=\alpha\delta-\beta\gamma\in \Z_p$, hence its image under \eqref{ring_isomorphism} lies in the diagonal subring of $\Z_p\times \Z_p$. This implies $\mathrm{det}(g)\in p^{2m}\Z_p^{\times}$, which is only possible for $m=0$. This proves that $m=0$ in all possible cases, and completes the proof of the first statement. 

Regarding the second statement, for $p\neq \ell$ we have $L'_p=L_p$ and the result follows. The case $p=\ell$ follows from the arguments used in the previous paragraph since  any $T\in \widetilde{U}_{\ell}$ is in $U_{\ell}$, hence it acts as a matrix of the form $g_D^mh$ where $m\geq 0$ is an integer and $h\in \GL_2(\mathcal{O}_D\otimes \Z_p)$ with $ \det(h) \in \Z_p^\times$. Since $h$ defines a transformation in $U_{\ell}$ and $g_D$ defines the transformation $T_D$ with $T_D^2\in U_{\ell}$, we get $\widetilde{U}_{\ell}=U_{\ell}\cup T_DU_{\ell}$ as wanted. This completes the proof.
\end{proof}

We now prove the following approximation result.

\begin{lemma}\label{lem:strong_approximation}
	We have $G(\A_f) = G(\Q)U$.
\end{lemma}

\begin{proof}
	The number of double cosets in
	\[
	G(\Q) \backslash G(\A_f)/\widetilde{U}
	\]
	is equal to the number of classes in the genus of $L$. Since $L$ is indefinite of determinant $D$, and $D$ is a prime discriminant, we have by \cite[Chapter 15, Theorem 21]{conwaysloane} that the genus of $L$ consist of a single class, thus
	\[
	G(\A_f) = G(\Q)\widetilde{U}.
	\]
From Lemma \ref{l:Dp_and_Up} we get $\widetilde{U}=U\cup T_DU$. Since $T_D\in G(\mathbb{Q})$ we conclude that
\[
	G(\A_f) = G(\Q)(U\cup T_DU)=G(\Q)U,
	\]
as claimed.
\end{proof}

We let 
\[
\Gamma_U = G(\Q) \cap U,
\]
which is a discrete subgroup of $G(\Q)$. From Lemma \ref{l:Dp_and_Up} it follows that
\[
\Gamma_U \cong \{g \in \GL_2(\mathcal{O}_D): \det(g) = \pm 1\}/\{\pm 1\}.
\]

\begin{remark}\label{rmk:det=+-1_to_det=1}
	Since an element in $\{g \in \GL_2(\mathcal{O}_D): \det(g) = \pm 1\}$ of determinant $-1$ can be written as $g = \varepsilon g_0$ with $\varepsilon = \left(\begin{smallmatrix}-1 & 0 \\ 0 & 1 \end{smallmatrix}\right)$ and $g_0 \in \SL_2(\mathcal{O}_D)$, we have
	\[
	\{g \in \GL_2(\mathcal{O}_D): \det(g) = \pm 1\} = \SL_2(\mathcal{O}_D) \cup  \SL_2(\mathcal{O}_D)\varepsilon,
	\]
	as a disjoint union. Also note that $\varepsilon$ acts on $V(\Q)$ by
	\[
	\begin{pmatrix}a & b \\ \overline{b} & c \end{pmatrix} \mapsto \begin{pmatrix}-a &  b \\ \overline{b} & -c \end{pmatrix}.
	\]
	Hence, we have
	\[
	\Gamma_U \backslash L_{m,\mu}^0 = \Gamma \backslash L_{m,\mu}^{0,+},
	\]
	where the $+$ indicates that we only take the positive definite (i.e. $a > 0$) hermitian forms in $L_{m,\mu}^0$ (recall that $\Gamma=\PSL_2(\mathcal{O}_D)$ and $L_{m,\mu}^0$ denotes the set of all $X \in L+\mu$ with $Q(X) = m$ which are primitive in $L'$). Moreover, for every $X\in V(\Q)$ the stabilizers subgroups $\Gamma_X$ and $\Gamma_{U,X} = \{\alpha \in \Gamma_U: \alpha X = X\}$ satisfy $|\Gamma_X|=|\Gamma_{U,X}|$.
\end{remark}

\begin{lemma}\label{lem:one_U_orbit}
For any fixed $X_0 \in L_{m,\mu}^0$ we have
\[
L_{m,\mu}^0 = V(\Q) \cap U X_0.
\]
\end{lemma}

\begin{proof}
	Since $U$ preserves $L$ and fixes the classes of $L'/L$, we have $V(\Q) \cap UX_0 \subseteq L_{m,\mu}^0$. Conversely, given $b_0\in \mu$, for each prime $p$ we will show that any point 
	\[
	X = \begin{pmatrix}a & b \\ \overline{b} & c \end{pmatrix} \in L_{m,\mu}^0
	\]
	can be transformed into the ``principal form''
	\[
	 \begin{pmatrix} 1 & b_0 \\ \overline{b_0} & * \end{pmatrix} \in L_{m,\mu}^0,
	\]
	(where $*$ is uniquely determined from $\det(X_0) = m$) using a  transformation in $G(\Q)\cap U_p$. First, since $X$ is primitive in $L'$, we can act by some element in $\PSL_2(\mathcal{O}_D) \subset G(\Q)\cap U_p$ to assume that $a$ is coprime to $p$. Indeed, if $a$ is already coprime to $p$ then there is nothing to prove. If $c$ is coprime to $p$, we can act by $\left(\begin{smallmatrix}0 & -1 \\ 1 & 0 \end{smallmatrix}\right)$. If $a$ and $c$ are divisible by $p$, then acting by $\left(\begin{smallmatrix}1 & z \\ 0 & 1 \end{smallmatrix}\right)$, with $z\in \mathcal{O}_D$, transforms the $a$-component of $X$ to $a+\tr(z\overline{b})+c|z|^2$. If this number were divisible by $p$ for all $z\in \mathcal{O}_D$, then we would have $b/p\in \mathfrak{d}^{-1}$, hence $\frac{1}{p}X\in L'$. This contradicts the fact that $X$ is primitive in $L'$. Now, since $p\nmid a$, the matrix $\left(\begin{smallmatrix}1/a & 0 \\ 0 & 1 \end{smallmatrix} \right)$ acts as a transformation in $G(\Q)\cap U_p$ (by Lemma \ref{l:Dp_and_Up}) with
	\[
	\begin{pmatrix}1/a & 0 \\ 0 & 1 \end{pmatrix}.X = \begin{pmatrix}1 & b \\ \overline{b} & ac \end{pmatrix}. 
	\]
	 Finally, for $\beta\in \mathcal{O}_D$ the matrix $\left(\begin{smallmatrix}1 & 0 \\ \overline{\beta} & 1 \end{smallmatrix} \right)$  also acts as a transformation in $G(\Q)\cap U_p$, with
	\[
	\begin{pmatrix}1 & 0 \\ \overline{\beta} & 1 \end{pmatrix}.\begin{pmatrix}1 & b \\ \overline{b} & ac \end{pmatrix}
 = \begin{pmatrix}1 & b + \beta \\ \overline{b+\beta} & * \end{pmatrix}.
	\]
	Hence,  by choosing $\beta = b_0-b \in \mathcal{O}_D$, we arrive at the desired principal form. This completes the proof of the lemma.
\end{proof}

From now on, we fix some primitive $X_0 \in L_{m,\mu}^0$. We let 
\[
H = \{g \in G: g X_0 = X_0\}
\]
be the stabilizer of $X_0$ in $G$, which can be identified with $\SO(W)$ where $W = (\Q X_0)^\perp$ is a negative definite three-dimensional rational quadratic space. In $H(\A_f)$ we have the compact open subgroup 
\[
K = H(\A_f) \cap U.
\]
Since $G(\A_f) = G(\Q)U$ and $H(\A_f) \subset G(\A_f)$, every element $h \in H(\A_f) \subset G(\A_f)$ can be written as a product $h = gu$ for some $g \in G(\Q)$ and $u \in U$. The following bijection, that we extract from \cite{shimura}, allows us to translate the set $\Gamma \backslash L_{m,\mu}^0$ into an adelic setting. Recall that $\Gamma_U=G(\Q)\cap U$ and $\Gamma_{U,X} = \{\alpha \in \Gamma_U: \alpha X = X\}$ is the stabilizer of a point $X \in V(\Q)$ in $\Gamma_U$.

\begin{theorem}\label{shimura theorem}
 We have a bijection
	\begin{align*}
	H(\Q) \backslash H(\A_f) / K \quad &\to  \quad \Gamma_U \backslash L_{m,\mu}^0, \\
	h = gu \quad &\mapsto \quad g^{-1} X_0.
	\end{align*}
	Moreover, if we let $\Gamma_h = H(\Q) \cap h K h^{-1}$ then we have $|\Gamma_h| = |\Gamma_{U,g^{-1}X_0}|$.
	\end{theorem}
	
	\begin{proof}
  Since $U \subset G(\A_f)$ is open and compact, the map
		\begin{align}\begin{split}\label{shimura bijection}
		H(\Q)\backslash (H(\A_f) \cap G(\Q)U)/(H(\A_f) \cap U) &\to (G(\Q)\cap U)\backslash (V(\Q) \cap UX_0) \\
		h = gu & \mapsto g^{-1}X_0,
		\end{split}
		\end{align}
  is a bijection by \cite[Theorem 2.2(ii)]{shimura}.
  By Lemma \ref{lem:strong_approximation} we have $G(\A_f) = G(\Q)U$ and hence $H(\A_f) \cap G(\Q)U = H(\A_f)$. Thus the left-hand side in \eqref{shimura bijection} is $H(\Q) \backslash H(\A_f)/K$. Moreover, by Lemma~\ref{lem:one_U_orbit} we have
		\[
		(G(\Q)\cap U)\backslash (V(\Q) \cap UX_0) = \Gamma_U \backslash L_{m,\mu}^0,
		\]
		This gives the stated bijection. We leave it to the reader to verify that the map
		\[
		\Gamma_h \to \Gamma_{U,g^{-1}X_0}, \quad \beta \mapsto g^{-1}\beta g
		\]
		is a bijection.
	\end{proof}
	
	We obtain the following splitting of the trace of the theta function.
	
	\begin{theorem}\label{theorem splitting theta}
		Fix a point $X_0 \in L_{m,\mu}^0$, and define sublattices
		\[
		P = L \cap \Q X_0, \qquad N = L \cap (\Q X_0)^\perp,
		\]
		which are one-dimensional positive definite and three-dimensional negative definite, respectively. Let $N^- = (N,-Q)$. Then we have
		\[
		\tr_{m,\mu}^0\left(\Theta_L(\tau,\, \cdot \,)\right) = \tr_{m,\mu}^0(1) \cdot \left(\Theta_{P}(\tau) \otimes \overline{E_{3/2,N^-}(\tau)}v^{3/2}\right)^L,
		\]
		where $\Theta_{P}(\tau)$ is the weight $1/2$ holomorphic theta function for $P$ and $E_{3/2,N^-}(\tau)$ is the weight $3/2$ holomorphic Eisenstein series for $N^-$, and the superscript $L$ denotes the operator defined in Section~\ref{section modular forms}.
	\end{theorem}
	
	\begin{proof}
		Using Remark~\ref{rmk:det=+-1_to_det=1} and the bijection from Theorem~\ref{shimura theorem}, we can write the left-hand side as
		\[
		\tr_{m,\mu}^0(\Theta_L(\tau,\,\cdot\,)) = \sum_{X \in \Gamma_U \backslash L_{m,\mu}^0}\frac{1}{|\Gamma_{U,X}|}\Theta_L(\tau,P_X) = \sum_{h \in H(\Q)\backslash H(\A_f)/K}\frac{1}{|\Gamma_h|}\Theta_L(\tau,P_0,h),
		\]
		where $P_0 \in \H^3$ is the special point corresponding to $X_0$, and
		\[
		\Theta_L(\tau,P_0,h) = \Im(\tau)^{3/2}\sum_{\mu \in L'/L}\sum_{X \in h(L + \mu)}e(Q(X_{v(P_0)})\tau + Q(X_{v(P_0)^\perp})\overline{\tau})\e_\mu.
		\]
		Since $K=H(\A_f) \cap U$ acts trivially on $L'/L$, the function $f(h) = \Theta_L(\tau,P_0,h)$ on $H(\Q) \backslash H(\A_f)$ is invariant under $K$ from the right, so Lemma \ref{sum lemma} leads to
		\[
		\sum_{h \in H(\Q)\backslash H(\A_f)/K}\frac{1}{|\Gamma_h|}\Theta_L(\tau,P_0,h) = \frac{1}{\vol(K)}\int_{H(\Q) \backslash H(\A_f)}\Theta_L(\tau,P_0,h)dh.
		\]
		By \eqref{eq theta relation} the theta functions for $L$ and $P \oplus N$ are related by
		\[
		\Theta_L(\tau,P_0,h) = \Theta_{P \oplus N}(\tau,P_0,h)^L.
		\]
		Note that the operator $f \mapsto f^L$ commutes with the integral, so we get
		\[
		\int_{H(\Q) \backslash H(\A_f)}\Theta_L(\tau,P_0,h)dh = \left(\int_{H(\Q) \backslash H(\A_f)}\Theta_{P \oplus N}(\tau,P_0,h)dh \right)^L.
		\]
		Next, we use the splitting
		\[
		\Theta_{P \oplus N}(\tau,P_0,h) = \Theta_{P}(\tau) \otimes v^{3/2}\overline{\Theta_{N^-}(\tau,h)}.
		\]
		Note that $\Theta_{P}(\tau)$ is independent of $h$ since $H(\A_f)$ acts trivially on $P$. Thus we get
		\[
		\int_{H(\Q) \backslash H(\A_f)}\Theta_{P \oplus N}(\tau,P_0,h)dh = \Theta_{P}(\tau) \otimes v^{3/2}\overline{\int_{H(\Q) \backslash H(\A_f)}\Theta_{N^-}(\tau,h)dh}.
		\]
		By the Siegel--Weil formula (Theorem~\ref{siegel weil formula}) we have
		\[
		\int_{H(\Q) \backslash H(\A_f)}\Theta_{N^-}(\tau,h)dh = 2E_{3/2,N^-}(\tau).
		\]
		It remains to compute $\vol(K)$. Using the formula for $\vol(K)$ from Lemma~\ref{sum lemma}, and the bijection from Theorem~\ref{shimura theorem} again, we obtain
		\[
		\frac{2}{\vol(K)} = \sum_{h \in H(\Q)\backslash H(\A_f)/K}\frac{1}{|\Gamma_h|} = \sum_{X \in \Gamma_U \backslash L_{m,\mu}^0}\frac{1}{|\Gamma_{U,X}|} =  \tr_{m,\mu}^0(1).
		\]
		Taking everything together, we obtain the stated formula.
	\end{proof}

\begin{remark}\label{remark determinant N}
    The unary lattice~$P$ associated to~$X_0\in L^0_{m,\mu}$ equals $\Z d_{\mu} X_0$ and has determinant~$2d_{\mu}m$, where~$d_{\mu}$ is the order of~$\mu$ in~$L'/L$. The ternary lattice~$N^-$ has determinant~$2m|D|$. Indeed, by Lemma~\ref{lem:one_U_orbit} one can assume that~$X_0$ is the ``principal form'' $\left(\begin{smallmatrix}
        1 & b_0 \\ \overline{b_0} & c_0
    \end{smallmatrix}\right)$, where~$b_0\in \mu$ and~$c_0=m+|b_0|^2$, and then a direct computation using the equality
    $$N=\left\{ \begin{pmatrix}
        a & b \\
        \overline{b} & \tr(b\overline{b_0})-ac_0
    \end{pmatrix}: a\in \Z, b\in \mathcal{O}_D\right\}$$
    gives the result. 
\end{remark}

\subsection{Traces of special values of Green's functions}

We fix a primitive rational point $X_0 \in L'$ of determinant $m' = Q(X_0) > 0$, and we let $P_0 \in \H^3$ be the corresponding special point. We define the sublattices
\[
P =  L \cap \Q X_0 ,\quad N = L \cap (\Q X_0)^\perp,
\]
which are positive definite one-dimensional and negative definite three-dimensional, respectively. Then $N^-= (N,-Q)$ is positive definite three-dimensional. We let $E_{3/2,N^-} \in M_{3/2,N^-}$ be the holomorphic weight $3/2$ Eisenstein series for $\overline{\rho}_N$ defined in Section~\ref{section holomorphic eisenstein series}. Moreover, we let $\widetilde{E}_{1/2,N} \in H_{1/2,N}^{\hol}$ be the harmonic Maass Eisenstein series of weight $1/2$ for $\rho_{N}$ defined in Section~\ref{section maass eisenstein series}, which satisfies $\xi_{1/2}\widetilde{E}_{1/2,N} = \frac{1}{2}E_{3/2,N^-}$.

We obtain the following explicit formula involving double traces of the Green's function.

\begin{theorem}\label{theorem evaluation Green's function}
	Let $n \geq 1$ be a positive integer, and let 
	\[
	f = \sum_{\mu \in L'/L}\sum_{m \in \Z-Q(\mu)}a_f(m,\mu)q^m \e_\mu \in M_{1-2n,L^-}^!
	\]
	be a weakly holomorphic modular form of weight $1-2n$ for the Weil representation $\overline{\rho}_L$. Let $\mu' \in L'/L$ and $-m' \in \Z-Q(\mu')$ with $m' > 0$ such that $a_f(-m'r^2,\mu'r) = 0$ for all integers $r\geq 1$. Then we have
	\begin{align}\label{eq theorem evaluation Green's function}
	&\frac{1}{2}\sum_{\mu \in L'/L}\sum_{m > 0}m^{n-1/2}a_f(-m,\mu)\tr_{m',\mu'}^0\tr_{m,\mu}(G_{2n})  \\
	&\qquad = \frac{4^n \pi}{\binom{2n}{n}} \tr_{m',\mu'}^0(1)\CT\left(f_{P \oplus N} \cdot \left[\Theta_{P},\widetilde{E}_{1/2,N}^+\right]_n \right), \nonumber
	\end{align}
	where $\widetilde{E}_{1/2,N}^+ $ denotes the holomorphic part of $\widetilde{E}_{1/2,N}$, $[\cdot,\cdot]_n$ denotes the $n$-th Rankin--Cohen bracket as defined in Section~\ref{section rankin cohen brackets}, with $k = \ell = 1/2$, and $\CT(\cdot)$ denotes the constant term of a holomorphic $q$-series. 
\end{theorem}

\begin{remark}\label{remark after theorem evaluation Green's function}
	\begin{enumerate}
		\item The condition $a_f(-m'r^2,\mu'r) = 0$ for all integers $r\geq 1$
  ensures that we never evaluate the Green's function $G_{2n}(P_1,P_2)$ along the diagonal. 
		\item By writing $\tr_{m',\mu'}$ in terms of primitive traces as explained in Section~\ref{section traces}, we can also get a closed formula for the non-primitive double trace $\tr_{m',\mu'} \tr_{m,\mu}(G_{2n})$.
		\item The Rankin--Cohen bracket appearing in the theorem is given by
	\[
	[g,h]_n = \sum_{j = 0}^n (-1)^j \binom{n-1/2}{j}\binom{n-1/2}{n-j} g^{(n-j)}h^{(j)},
	\]
	with $g^{(j)} = \left(\frac{1}{2\pi i }\frac{\partial}{\partial \tau}\right)^j g$. 
	The product of binomial coefficients can be written as
	\[
	\binom{n-1/2}{j}\binom{n-1/2}{n-j} = \frac{1}{4^n}\binom{2n}{n}\binom{2n}{2j}.
	\]
	\item The right-hand side of~\eqref{eq theorem evaluation Green's function} equals
\begin{align*}
	\pi \tr_{m',\mu'}^0(1)\sum_{\substack{\mu'\in L'/L \\ m>0}}a_f(-m,\mu)\sum_{\substack{\alpha \in P'/P \\ \beta \in N'/N \\ \alpha + \beta = \mu \!\!\!\!\!\pmod{L}}}\sum_{\ell \in \Z+Q(\alpha)}\kappa_{m,n}(\ell) c_{\Theta_P}(\ell,\alpha)c_{\widetilde{E}_{1/2,N}}^+(m-\ell,\beta),
	\end{align*}
 with the rational constants $\kappa_{m,n}(\ell) = \sum_{j=0}^n (-1)^j \binom{2n}{2j}\ell^{n-j}(m-\ell)^j$.
 Suppose that $S_{1+2n,L} = \{0\}$. Then we can choose $f = F_{1-2n,m,\mu} = q^{-m}(\e_{\mu} + \e_{-\mu}) + O(1)$ for some fixed $\mu \in L'/L$ and $m \in \Z+Q(\mu)$ with $m > 0$. In this case, formula~\eqref{eq theorem evaluation Green's function} simplifies to
	\begin{align*}
	&m^{n-1/2}\tr_{m',\mu'}^0 \tr_{m,\mu}(G_{2n}) \\
	&= 2\pi \tr_{m',\mu'}^0(1)\sum_{\substack{\alpha \in P'/P \\ \beta \in N'/N \\ \alpha + \beta = \mu \!\!\!\!\!\pmod{L}}}\sum_{\ell \in \Z+Q(\alpha)}\kappa_{m,n}(\ell) c_{\Theta_P}(\ell,\alpha)c_{\widetilde{E}_{1/2,N}}^+(m-\ell,\beta).
	\end{align*}
	\end{enumerate}
\end{remark}

\begin{proof}[Proof of Theorem~\ref{theorem evaluation Green's function}]
	Since the weight $1-2n$ is negative, we can write $f$ as a linear combination of Maass Poincar\'e series,
	\[
	f(\tau) = \frac{1}{2}\sum_{\mu \in L'/L}\sum_{m > 0}a_f(-m,\mu)F_{1-2n,m,\mu}(\tau),
	\]
	see \eqref{eq linear combination Maass Poincare} in Section~\ref{section harmonic maass forms}. Writing the Green's function as a theta lift using Theorem~\ref{theorem higher Green's function as theta lift}, we have
	\[
	\frac{1}{2}\sum_{\mu \in L'/L}\sum_{m > 0}m^{n-1/2}a_f(-m,\mu)\tr_{m,\mu}\big(G_{2n}(\,\cdot\,,Q)\big)  = \frac{1}{(4\pi)^{n}}\frac{\pi}{2}\Phi_L^{(1-2n)}\left(f,Q\right).
	\]
	Taking the trace $\tr_{m',\mu'}^0$ in $Q$, and using the splitting of the trace of the Siegel theta function from Theorem~\ref{theorem splitting theta} (which involves an application of the Siegel--Weil formula), we obtain
	\[
	\tr_{m',\mu'}^0\left(\Phi_L^{(1-2n)}\left(f,\, \cdot \,\right)\right) = \tr_{m',\mu'}^0(1)\int_{\mathcal{F}}^{\reg}\big(R_{1-2n}^n f_{P\oplus N}\big)(\tau)\cdot \Theta_{P}(\tau) \otimes \overline{E_{3/2, N^-}(\tau)}v^{3/2}d\mu(\tau). 
	\]
	Next, using the ``self-adjointness'' of the raising operator, we obtain
	\begin{align*}
	&\int_{\mathcal{F}}^{\reg}\big(R_{1-2n}^n f_{P\oplus N}\big)(\tau)\cdot \Theta_{P}(\tau) \otimes \overline{E_{3/2,N^-}(\tau)}v^{3/2}d\mu(\tau) \\
	&= (-1)^n\int_{\mathcal{F}}^{\reg} f_{P\oplus N}(\tau)\cdot R_{-1}^n\bigg(\Theta_{P}(\tau) \otimes \overline{E_{3/2,N^-}(\tau)}v^{3/2}\bigg)d\mu(\tau) \\
	&= (-1)^n\int_{\mathcal{F}}^{\reg} f_{P\oplus N}(\tau)\cdot \left(R_{1/2}^n \Theta_{P}(\tau)\right) \otimes \overline{E_{3/2,N^-}(\tau)}v^{3/2}d\mu(\tau).
	\end{align*}
	In the last step we used that $E_{3/2,N^-}$ is holomorphic (Theorem \ref{thm Rallis holomorphic Eisenstein series}). Moreover, one has to check that certain boundary integrals vanish, which is straightforward but tedious. Next, Proposition~\ref{proposition rankin cohen bracket} implies that
	\[
	\left(R_{1/2}^n\Theta_{P}(\tau)\right)\otimes \overline{E_{3/2,N^-}(\tau)}v^{3/2} = 2\frac{(-4 \pi)^n}{\binom{n-1/2}{n}} L_{1+2n} \left[\Theta_{P},\widetilde{E}_{1/2,N}\right]_n,
	\]
	where we used that $\Theta_P$ is holomorphic. Note that $\binom{n-1/2}{n} = \frac{1}{4^n}\binom{2n}{n}$. Now a standard application of Stokes' Theorem as in \cite[Proposition~3.5]{bruinierfunke} gives
	\begin{align*}
	&\int_{\mathcal{F}}^{\reg} f_{P\oplus N}(\tau)\cdot \left(R_{1/2}^n \Theta_{P}(\tau)\right) \otimes \overline{E_{3/2,N^-}(\tau)}v^{3/2}d\mu(\tau) \\
	&\qquad = 2\frac{(-4\pi)^n}{\frac{1}{4^n}\binom{2n}{n}}\CT\left(f_{P\oplus N}\cdot \left[\Theta_{P},\widetilde{E}_{1/2,N}^+\right]_n \right).
	\end{align*}
	Here we used that $f$ (and hence $f_{P\oplus N}$) is holomorphic on $\H$. Taking everything together, we obtain the stated formula.
\end{proof}

\begin{example}\label{example 1}
	Here we give the details for Example~\ref{example 1 intro} from the introduction. We work over the field $ \Q(i)$ with discriminant $D = -4$. Then 
	\[
	L'/L \cong \tfrac{1}{2}\Z[i]/\Z[i] \cong \Z/2\Z \times \Z/2\Z.
	\]
	We will write the elements of $L'/L$ as $(b_1,b_2)$ with $b_1,b_2 \in \Z/2\Z$, and understand that this tuple corresponds to the class of the matrix $\left(\begin{smallmatrix}0 & (b_1+ib_2)/2 \\ (b_1-ib_2)/2 & 0 \end{smallmatrix}\right) \in L'$.

Let us take $n = 1$. Since $S_{3,L} = \{0\}$, the Maass Poincar\'e series $f_{-1,m,\mu}$ is weakly holomorphic for every $m > 0$ and $\mu \in L'/L$, so we can just take $f = f_{-1,m,\mu}$. Then the formula from Theorem~\ref{theorem evaluation Green's function} simplifies to
\begin{align}\label{example 1 step 1}
	&\sqrt{m}\tr_{m',\mu'}^0\tr_{m,\mu}(G_{2n}) = 2\pi \tr_{m',\mu'}^0(1)\CT\left(f_{P \oplus N} \cdot \left[\Theta_{P},\widetilde{E}_{1/2,N}^+\right]_1 \right).
	\end{align}

 We take $m = 1/2$ and $\mu = (1,1)$, such that the trace $\tr_{m,\mu}$ has only one summand for the form $\left(\begin{smallmatrix}1 & (1+i)/2 \\ (1-i)/2 & 1 \end{smallmatrix} \right)$ with corresponding point $\frac{1+i}{2}+\frac{\sqrt{2}}{2}j$, with stabilizer of order $12$. Moreover, we choose $m' = 1$ and $\mu' = (0,0)$, such that $\tr_{m',\mu'}^0$ has only one summand for the form $X_0 = \left(\begin{smallmatrix}1 & 0 \\ 0 & 1 \end{smallmatrix} \right)$ with corresponding point $j$, with stabilizer of order $4$. In particular, $\tr_{m',\mu'}^0(1) = \frac{1}{4}$. Hence, \eqref{example 1 step 1} becomes
 \begin{align}\label{example 1 step 2}
	\frac{\sqrt{1/2}}{48}G_2\left(j,\frac{1+i}{2}+\frac{\sqrt{2}}{2}j\right) = \frac{\pi}{2}\CT\left(f_{P \oplus N} \cdot \left[\Theta_{P},\widetilde{E}_{1/2,N}^+\right]_1 \right).
	\end{align}
	
Next, we compute the sublattices $P$ and $N$. We have
\begin{align*}
P &= L \cap \Q X_0 = \left\{n\begin{pmatrix}1 & 0 \\ 0 & 1 \end{pmatrix}: n \in \Z\right\} \cong (\Z,n^2), \\
N &= L \cap (\Q X_0)^\perp = \left\{\begin{pmatrix}a & b \\ \overline{b} & -a \end{pmatrix}: a \in \Z, b \in \Z[i]\right\} \cong (\Z^3,-a^2-b_1^2-b_2^2).
\end{align*}
Their dual lattices are given by
\begin{align*}
P' &= \frac{1}{2}P = \left\{\frac{n}{2}\begin{pmatrix}1 & 0 \\ 0 & 1 \end{pmatrix}: n \in \Z\right\}, \\
N' &= \frac{1}{2}N = \left\{\frac{1}{2}\begin{pmatrix}a & b \\ \overline{b} & -a \end{pmatrix}: a \in \Z, b \in \Z[i]\right\}.
\end{align*}
Note that for $\alpha \in P'$ and $\beta \in N'$ we have $\alpha + \beta \in L'$ if and only if $n \equiv a \pmod 2$. We have
\[
P'/P \cong \Z/2\Z, \quad N'/N \cong (\Z/2\Z)^3.
\]
In particular, we will write the elements of $(P \oplus N)'/(P \oplus N) \cong P'/P \times N'/N$ as $(n,(a,b_1,b_2))$ with $n,a,b_1,b_2 \in \Z/2\Z$. Such an element is in $L'/L$ if and only if $n \equiv a \pmod 2$, and in this case it is in the class of the tuple $(b_1,b_2) \in L'/L$.

The first Rankin--Cohen bracket of two forms of weight $1/2$ is given by
	\[
	[f,g]_1 = \frac{1}{2}(f' g - f g'),
	\]
	where $f' = \frac{1}{2\pi i }\frac{d}{d\tau}f$. Moreover, by construction, $\Theta_{P}$ and $\widetilde{E}_{1/2,N}^+$ do not have any terms of negative index, which implies that $[\Theta_{P},\widetilde{E}_{1/2,N}^+]_1$ has vanishing principal part. Since $f_{-1,m,\mu}$ has principal part $q^{-m}(\e_\mu + \e_{-\mu})$, we obtain
	\begin{align*}
&\CT\left((f_{-1,m,\mu})_{P \oplus N} \cdot \left[\Theta_{P},\widetilde{E}_{1/2,N}^+\right]_1\right) \\
&\quad= 2c_{\left[\Theta_{P},\widetilde{E}_{1/2,N}^+\right]_1}(m,(0,(0,\mu))) + 2c_{\left[\Theta_{P},\widetilde{E}_{1/2,N}^+\right]_1}(m,(1,(1,\mu))).
\end{align*}
Hence, it remains to compute the coefficients of the Rankin--Cohen bracket. Recall that we take $m = 1/2$ and $\mu = (1,1)$. The coefficient at $q^{1/2}$ of the component of index $(0,(0,1,1))$ is given by $-1/2$ times the constant coefficient of $\Theta_P$ at $\e_0$ (which equals $1$) times $1/2$ times the coefficient at $q^{1/2}\e_{(0,1,1)}$ of $\widetilde{E}_{1/2,N}^+$ (here we use that the constant coefficient of the derivative of $\Theta_P$ vanishes, and that the coefficient at $q^{1/2}$ of $\Theta_P$ vanishes). Using Theorem~\ref{eisenstein series fourier expansion} and the Dirichlet class number formula, this latter coefficient can be computed as
\[
c_{\widetilde{E}_{1/2,N}^+}(1/2,(0,1,1)) = -\frac{4}{\pi}L(\chi_8,1) = -\frac{4}{\pi}\frac{\log(3+\sqrt{8})}{\sqrt{8}} = -\frac{2}{\pi}\frac{\log(3+\sqrt{8})}{\sqrt{2}}.
\]
Hence, we get
\[
c_{\left[\Theta_{P},\widetilde{E}_{1/2,N}^+\right]_1}\big(1/2,(0,(0,1,1))\big) = \frac{1}{\pi}L(\chi_8,1).
\]
Similarly, we see that
\[
c_{\left[\Theta_{P},\widetilde{E}_{1/2,N}^+\right]_1}\big(1/2,(1,(1,1,1))\big) = 0
\]
due to a cancellation in the Rankin--Cohen bracket. Putting this into \eqref{example 1 step 2}, we finally obtain
\[
\frac{1}{\sqrt{1/2}}G_2\left(j,\frac{1+i}{2}+\frac{\sqrt{2}}{2}j\right) = 96L(\chi_8,1).
\]

 For $n = 2$ we have $S_{5,L} = \{0\}$, so we can compute in a similar way that
\[
\frac{1}{\sqrt{1/2}}G_4\left(j,\frac{1+i}{2}+\frac{\sqrt{2}}{2}j\right) = 96\log(2) - 96L(\chi_8,1).
\]
\end{example}

We summarize the algebraic properties of the double traces of the Green's function.

\begin{theorem}\label{theorem algebraicity greens function 2}
	Let  $f\in M_{1-2n,L^-}^! $, $\mu' \in L'/L$ and $-m' \in \Z-Q(\mu')$ be  as in Theorem~\ref{theorem evaluation Green's function}. Suppose that the coefficients $a_f(-m,\mu)$ for $m > 0$ are rational. Then the linear combination of double traces
	\begin{align*}
	\frac{1}{\sqrt{m'|D|}}\sum_{\mu \in L'/L}\sum_{m > 0}m^{n-1/2}a_f(-m,\mu)\tr_{m',\mu'}^0\tr_{m,\mu}(G_{2n})
	\end{align*}
	is a rational linear combination of $\log(p)$ for some primes $p$ and $\log(\varepsilon_\Delta)/\sqrt{\Delta}$ for some fundamental discriminants $\Delta > 0$, where $\varepsilon_\Delta$ is the smallest totally positive unit $>1$ in $\Q(\sqrt{\Delta})$.
\end{theorem}

\begin{proof}
	We need to analyze the right-hand side of~\eqref{eq theorem evaluation Green's function}. Since~$n\geq 1$, and $\Theta_P$ and $\widetilde{E}_{1/2,N}^+$ have Fourier expansions supported on non-negative indices, the constant term of~$f_{P \oplus N} \cdot \left[\Theta_{P},\widetilde{E}_{1/2,N}^+\right]_n$ is a linear combination with coefficients in $\Q$ of products of negative index coefficients of $f_{P \oplus N}$ (hence of $f$), and positive index coefficients of $\Theta_{P}$ and $\widetilde{E}_{1/2,N}^+$. Moreover, the theta function $\Theta_P$ has integral Fourier coefficients by construction. Now, by Theorem~\ref{eisenstein series fourier expansion} and the subsequent Remark~\ref{remark eisenstein series}(3), the coefficients of $\widetilde{E}_{1/2,N}^+$ are of the form
	\[
	c_{\widetilde{E}_{1/2,N}}^+(n,\gamma) = \frac{\sqrt{2}}{\sqrt{|N'/N|}\pi} \times \text{rational number}  \times 
		\begin{cases} 
		\log(\varepsilon_{\Delta_0})/\sqrt{\Delta_0}, & \text{if $\Delta$ is not a square}, \\
		\log(p), & \text{if $\Delta$ is a square},
		\end{cases}
		\]
		for some prime $p \mid 2\det(N)$ in the square case. Here we use the same notation as in Theorem~\ref{eisenstein series fourier expansion}. The factor $\pi$ in the denominator on the right cancels out with the factor $\pi$ in the numerator on the right-hand side of~\eqref{eq theorem evaluation Green's function}. Moreover, by Remark~\ref{remark determinant N} we have $|N'/N|=2m'|D|$. Taking everything together, we obtain the stated result.
	\end{proof}

\begin{remark}\label{rem:primes_in_traces}
It follows from Remarks~\ref{remark determinant N} and~\ref{remark after theorem evaluation Green's function}(4) that  the right-hand side of~\eqref{eq theorem evaluation Green's function} is a linear combination of coefficients~$c_{\widetilde{E}_{1/2,N}}^+(n,\beta)$ with~$n=m-\ell$ and~$\ell$ of the form~$\frac{t^2}{4d_{\mu}^2m'}$ ($t\in \Z$), hence the relevant discriminants are of the form~$\Delta=4d_{\beta}^2nm'|D|$ and $\Delta_0$ is the discriminant of~$\Q\left(\sqrt{(4mm'd_{\mu}^2-t^2)|D|}\right)$. Moreover, when~$\Delta$ is a square, the coefficient~$c_{\widetilde{E}_{1/2,N}}^+(n,\beta)$ is a rational multiple of~$\log(p)$ provided~$p$ is the unique prime divisor of~$2m'|D|$ for which~$N\otimes \Q_p$ is anisotropic, and it vanishes if there are two or more such primes by Remark~\ref{remark eisenstein series}$(2)$. By using properties of Hilbert symbols it is easy to check that $N\otimes \Q_p$ is anisotropic if and only if $(-m,D)_p=-1$. This explains Remark~\ref{remark on fund disc and primes appearing}.
\end{remark}

We note that Theorem \ref{theorem algebraicity greens function 2} can be rephrased in terms of linear combinations of traces with coefficients coming from rational relations for spaces of cusp forms (see Section \ref{sec rational relations}). More precisely, in the language of rational relations, Theorem \ref{theorem algebraicity greens function 2} says that if~$\{\lambda(m,\mu)\}_{m\in \frac{1}{|D|}\N,\,\mu \in L'/L}$ is a rational relation for~$S_{1+2n,L}$, then for any $\mu' \in L'/L$ and $-m' \in \Z-Q(\mu')$ such that  $m' > 0$ and $\lambda (m'r^2,\mu'r) = 0$ for all integers $r\geq 1$, we have that the linear combination of double traces
	\begin{align*}
	\frac{1}{\sqrt{m'|D|}}\sum_{\mu \in L'/L}\sum_{m > 0}m^{n-1/2}\lambda(m,\mu)\tr_{m',\mu'}^0\tr_{m,\mu}(G_{2n})
	\end{align*}
	is a rational linear combination of $\log(p)$ for some primes $p$ and $\log(\varepsilon_\Delta)/\sqrt{\Delta}$ for some fundamental discriminants $\Delta > 0$. 
Then, Theorem~\ref{theorem algebraicity greens function 2} and Lemma \ref{lem vector-valued to scalar mfs} 
imply the following corollary, of which Theorem~\ref{main theorem} in the introduction is a special case.

\begin{corollary}\label{corollary that implies main theorem}
Let~$\{\lambda(t)\}_{t\in \N}$ be a rational relation for~$S^+_{1+2n}(\Gamma_0(|D|),\chi_D)$, and let $m' \in \frac{1}{|D|}\N$ be such that $\lambda (m'|D|r^2) = 0$ for all integers~$r\geq 1$. Then the linear combination of double traces
	\begin{align*}
	\frac{1}{\sqrt{m'|D|}}\sum_{m > 0}m^{n-1/2}\lambda(m|D|)\tr_{m'}^0\tr_{m}(G_{2n})
	\end{align*}
	is a rational linear combination of $\log(p)$ for some primes $p$ and $\log(\varepsilon_\Delta)/\sqrt{\Delta}$ for some fundamental discriminants $\Delta > 0$, where $\varepsilon_\Delta$ is the smallest totally positive unit $>1$ in $\Q(\sqrt{\Delta})$.
\end{corollary}

\section{Twisted traces of special values of Green's functions}\label{section twisted traces green's function}

In this section we let $(L,Q)$ denote the lattice of integral binary hermitian forms over $\Q(\sqrt{D})$ defined in Section \ref{sec orthogonal model}, associated to a fundamental discriminant $D<0$. Recall that $V=L\otimes \Q$.

\subsection{The twisting function}

Let us write $D = \prod_{p \mid D} p^*$ with prime discriminants $p^*$, that is, for odd $p$ we have $p^* = (\frac{-1}{p})p$ and for $p = 2$ we have $p^* \in \{-4, \pm 8\}$.

We set
\[
L_D = \{X \in L' : Q(X) \in D\Z\}.
\]
As in \cite[Lemma~3.3]{ehlen} one can check that $L_D \subseteq L$. For $X = \left(\begin{smallmatrix}a & b \\ \overline{b} & c \end{smallmatrix} \right) \in L_D$ we define the function
\begin{equation}\label{eq chi_D(X)}
\chi_D(X) = \prod_{p \mid D}\chi_p(X), \qquad \chi_p(X) = \begin{cases}
\left( \frac{p^*}{a}\right), & \text{if } p \nmid a, \\
\left( \frac{p^*}{c}\right), & \text{if } p \nmid c, \\
0, & \text{otherwise}.
\end{cases}
\end{equation}
For $X \in V(\Q)$ with $X \notin L_D$ we put $\chi_D(X) = 0$. It is easy to check that $\chi_D$ is well-defined, $\Gamma$-invariant, and only depends on $X$ modulo $DL$. For prime discriminants $D<0$, it coincides with the function defined in the introduction.

So far, we have defined $\chi_D$ on all of $V(\Q)$. Next, we extend it to $V(\A_f)$.

\begin{definition}
For a prime $p$ dividing~$D$ and $X_p = \left(\begin{smallmatrix} a & b \\ \overline{b} & c \end{smallmatrix} \right) \in L_p'$ with $Q(X_p) \in D \Z_p$ we let
\[
\chi_{D,p}(X_p) = \begin{cases}
(a,D)_p, & \text{if $p \nmid a$}, \\
(c,D)_p, & \text{if $p \nmid c$}, \\
0, & \text{otherwise},
\end{cases}
\]
where $(c,D)_p$ denotes the $p$-adic Hilbert symbol. For all other $X_p \in V(\Q_p)$ we put $\chi_{D,p}(X_p) = 0$, so that $\chi_{D,p}$ is defined on all of $V(\Q_p)$. For primes~$p$ not dividing~$D$, we define~$\chi_{D,p}$ as the characteristic function of~$L_p'=L_p$. Once again, it is easy to check that~$\chi_{D,p}$ is well-defined. For $X = (X_p)_{p < \infty} \in V(\A_f)$ we put
\[
\chi_D(X) = \prod_{p < \infty}\chi_{D,p}(X_p).
\]
\end{definition}
It follows from the basic properties of the Hilbert symbol that $\chi_D$ agrees with the previous definition on $V(\Q)$. 

By a slight abuse of notation, we also define $\chi_D$ as a function on $\A_f^\times$ by
\begin{equation}\label{eq:chi-adelic}
 \chi_D(x) = \prod_{p < \infty}\chi_{D,p}(x_p) = \prod_{p < \infty}(x_p,D)_p.   
\end{equation}
Note that $\chi_D$ is a quadratic character on $\A_f^\times / \Q^+$.

Let $\nu$ denote the spinor norm on $\mathrm{GSpin}_V$. By \eqref{eq: description of G(F)} the map
$$\chi_{D,p}\circ \nu:\mathrm{GSpin}_V(\Q_p)\to \{\pm 1\}$$
induces a quadratic character~$\chi_{D,p}\circ \nu:G(\Q_p)\to \{\pm 1\}$ . Hence, we obtain the quadratic character
\begin{equation}\label{eq:chi-adelic-on_G}
   \chi_D\circ \nu:G(\A_f)\to \{\pm 1\}, \qquad \chi_D\circ \nu=\prod_p \chi_{D,p}\circ \nu. 
\end{equation}

We have the following transformation property of $\chi_D(X)$ under the action of the subgroup $U=\prod_p U_p$ defined  in Section \ref{sec orthogonal and spinor groups}. 

\begin{proposition}\label{proposition character invariance}
For every prime~$p$, every $g_p\in U_p$ and~$X_p\in V(\Q_p)$ we have
\[
\chi_{D,p}(g_p X_p) = \chi_{D,p}(\nu(g_p))  \chi_{D,p}(X_p).
\]
In particular, for all~$g \in U$ and $X \in V(\A_f)$ we have
\[
\chi_{D}(g X) = \chi_D(\nu(g))  \chi_{D}(X).
\]
\end{proposition}

\begin{proof} It is enough to prove the first statement since the second one follows from the fact that~$\chi_D$ is the product of~$\chi_{D,p}$ over all primes~$p$. Let~$g_p\in U_p$ and~$X_p\in V(\Q_p)$.  Recall that $g_p$ is given by the action of a matrix $h_p=\left(\begin{smallmatrix}
        \alpha & \beta \\
  \gamma       &  \delta
    \end{smallmatrix}\right)\in \GL_2(\mathcal{O}_D\otimes \Z_p)$ with $\nu(g_p)=\det(h_p)\in \Z_p^{\times}$ (Lemma~\ref{l:Dp_and_Up}). If~$p$ does not divide~$D$ then~$(\nu(g_p),D)_p=1$ and~$\chi_{D,p}$ on $V(\Q_p)$ is the characteristic function of~$L_p'=L_p$. Since~$g_p$ preserves~$L_p$, we have~$g_pX_p\in L_p$ if and only if~$X_p\in L_p$, hence the desired identity follows. Now, assume~$p$ divides $D$. Since the action of $g_p$ preserves $L_p'$ and the quadratic form~$Q$, both $\chi_{D,p}(g_pX_p)$ and $\chi_{D,p}(X_p)$ are zero if $X_p$ is not in $L_p'$ or has $Q(X_p) \not \in D \Z_p$. Hence, we can assume $X_p\in L_p'$ and $Q(X_p)\in D \Z_p$. This implies~$X_p\in L_p$. If we write $X_p=\left(\begin{smallmatrix}
   a  & b \\
    \overline{b} & c
\end{smallmatrix}\right)$, then
    $$\begin{pmatrix}
        A & B \\
        \overline{B} & C
    \end{pmatrix}=g_p X=\frac{1}{\mathrm{det}(h_p)}\begin{pmatrix}
        a|\alpha|^2+\tr(\alpha \overline{\beta}b)+c|\beta|^2 &\ast  \\
        \ast & a|\gamma|^2+\tr(\gamma\overline{\delta}b)+c|\delta|^2
    \end{pmatrix}$$
    where the norms and traces are taken in the ring~$\mathcal{O}_D\otimes \Z_p\subseteq \mathbb{Q}_p(\sqrt{D})$. Note that~$p\mid D$ implies  $\tr(z)\in p\Z_p$ for every~$z\in \mathcal{O}_D\otimes \Z_p$. In particular, $\tr(\alpha \overline{\beta}b),\tr(\gamma\overline{\delta}b)\in p\Z_p$. Hence, if $a,c\in p\Z_p$, then $\chi_{D,p}(g_p X_p)=\chi_{D,p}(X_p)=0$. We get the same conclusion if $A,C\in p\Z_p$. Assume  $a,A\in \Z_p^{\times}$. Note that
    \begin{eqnarray*}
        \mathrm{det}(h_p)Aa &=&|a\alpha+\beta \overline{b}|^2+|\beta|^2(ac-|b|^2)\\
        &=& |a\alpha+\beta \overline{b}|^2+|\beta|^2 Q(X_p)\\
        &\equiv & |a\alpha+\beta \overline{b}|^2 \, \text{ mod }D\Z_p.
    \end{eqnarray*}
    This implies~$(\mathrm{det}(h_p)Aa,D)_p=1$, hence
    $$\chi_{D,p}(g_pX_p)=(A,D)_p=(\mathrm{det}(h_p),D)_p(a,D)_p=\chi_{D,p}(\nu(g_p))\chi_{D,p}(X_p).$$
    The remaining cases are treated similarly. This proves the result.
\end{proof}

\subsection{The twisted modified Siegel theta function}\label{section twisted theta function}

In this section we construct a scalar-valued twisted Siegel theta function, which is modified with a polynomial. A similar twisted theta function was first constructed in \cite{ehlen} for lattices of signature $(2,2)$ and $D > 0$, whereas we are in the case of signature $(1,3)$ and $D < 0$.

Recall that $G=\SO(V)$ where $V=L\otimes \Q$.

\begin{definition}
For $h \in G(\A_f), \tau \in \H$ and $P \in \H^3$ we define the (scalar-valued) $D$-twisted modified theta function
\begin{equation}\label{eq (scalar-valued) D-twisted modified theta function}
   \Theta_{L,\chi_D}^*(\tau,P,h) = \Im(\tau)^{3/2}\sum_{X \in hL}\chi_D(h^{-1}X)\tfrac{1}{|D|}(X,v(P))e\left(\frac{Q(X_v)}{|D|}\tau + \frac{Q(X_{v^\perp})}{|D|}\overline{\tau}\right), 
\end{equation}
where $v = \R v(P) \in \Gr(L)$ is the positive line corresponding to $P$ as in Section \ref{sec orthogonal model}. 
\end{definition}
Note that this is not the usual Siegel theta function, but it is modified with a polynomial in $X$, which is harmonic and homogeneous of degree $(1,0)$. It is straightforward to check that the twisted theta function is $\Gamma$-invariant in $P$. The goal of this section is to prove its modularity in $\tau$.

\begin{proposition}\label{proposition twisted theta function}
	The twisted theta function $\Theta_{L,\chi_D}^*(\tau,P,h)$ transforms like a (scalar-valued) modular form of weight $0$ for the full modular group $\SL_2(\Z)$ in $\tau$.
\end{proposition}

We will deduce this result from the well known modularity of a (vector-valued) non-twisted modified Siegel theta function, using an ``intertwining operator'' $\psi_D$ of certain Weil representations. This general method for constructing twisted theta functions was 
developed in \cite{alfesehlen}.

We let $L(D)$ be the rescaled lattice $DL = \{DX: X \in L\}$ with the quadratic form $Q_D(X) = \frac{Q(X)}{|D|}$. Then the dual lattice of $L(D)$ is again $L'$ (as a set, but $L'$ is now equipped with $Q_D(X)$). We write $\mathcal{L}(D) = L'/L(D)$ for the discriminant group of $L(D)$. Note that $|\mathcal{L}(D)| = |D|\cdot |D|^4 = |D|^5$, and that $L(D) \subset L_D $ (as sets, but they have different quadratic forms). We let $\rho_{\mathcal{L}(D)}$ be the Weil representation corresponding to $\mathcal{L}(D)$. Recall that $\chi_D(X)$ defined in \eqref{eq chi_D(X)} depends only on $X$ modulo $DL$, hence it defines a function on $\mathcal{L}(D)$. Then we have the following general result.

\begin{lemma}\label{lemma intertwiner}
	Put
	\[
	\psi_D = \sum_{\delta \in \mathcal{L}(D)}\chi_D(\delta)\e_\delta = \sum_{\delta \in L_D/L(D)}\chi_D(\delta)\e_\delta \in \C[\mathcal{L}(D)].
	\]
	Then $\psi_D$ is invariant under the dual Weil representation $\overline{\rho}_{\mathcal{L}(D)}$.
\end{lemma}

\begin{proof}
	Put $\rho = \rho_{\mathcal{L}(D)}$ for brevity. Using the notation in Section \ref{section modular forms}, we need to check that $\overline{\rho}(T)\psi_D = \psi_D$ and $\overline{\rho}(S)\psi_D = \psi_D$. We have
	\[
	\overline{\rho}(T)\psi_D = \sum_{\delta \in L_D/L(D)}\chi_D(\delta)e\left( -\frac{Q(\delta)}{|D|}\right)\e_\delta = \sum_{\delta \in L_D/L(D)}\chi_D(\delta)\e_\delta = \psi_D,
	\]
	where we used that $\frac{Q(\delta)}{|D|} \in \Z$ for $\delta \in L_D$. Moreover, we have
	\begin{align*}
	\overline{\rho}(S)\psi_D &= \sum_{\delta \in L_D/L(D)}\chi_D(\delta)\overline{\rho}(S)\e_\delta \\
	&= \sum_{\delta \in L_D/L(D)}\chi_D(\delta)\frac{e((1-3)/8)}{\sqrt{|\mathcal{L}(D)|}}\sum_{\mu \in L'/L(D)}e\left(\frac{1}{|D|}(\delta,\mu)\right)\e_\mu \\
	&= -\frac{i}{|D|^{5/2}}\sum_{\mu \in L'/DL}\left(\sum_{\delta \in L_D/L(D)}\chi_D(\delta)e\left(\frac{1}{|D|}(\delta,\mu)\right)\right)\e_\mu.
	\end{align*}
	Hence it remains to show the evaluation of the Gauss sum
	\[
	\sum_{\delta \in L_D/L(D)}\chi_D(\delta)e\left(\frac{1}{|D|}(\delta,\mu)\right) = i |D|^{5/2}\chi_D(\mu).
	\]
	This can be proved by the same arguments as in \cite[Section~4.2]{ehlen}.
\end{proof}

	\begin{corollary}\label{corollary intertwiner}	
	Let $k \in \Z$, let $A_{k,\rho_{\mathcal{L}(D)}}$ be the space of functions transforming like vector-valued modular forms of weight $k$ for $\rho_{\mathcal{L}(D)}$, and let $A_{k,\SL_2(\Z)}$ be the space of functions transforming like scalar-valued modular forms of weight $k$ for $\SL_2(\Z)$. Then we have a map
	\[
	A_{k,\rho_{\mathcal{L}(D)}} \to A_{k,\SL_2(\Z)}, \qquad f = \sum_{\delta \in L'/L(D)}f_\delta \e_\delta \mapsto f \cdot \psi_D = \sum_{\delta \in L_D/L(D)}\chi_D(\delta)f_{\delta}.
	\]
	Here $f \cdot \psi_D$ denotes the bilinear pairing on $\C[\mathcal{L}(D)]$.
	\end{corollary}
	
	\begin{proof}
		We again put $\rho = \rho_{\mathcal{L}(D)}$. It is well known that the Weil representation satisfies 
		\[
		(\rho(M,\phi)\mathfrak{a}) \cdot \mathfrak{b} = \mathfrak{a} \cdot (\overline{\rho}(M,\phi)^{-1}\mathfrak{b})
		\] for any $(M,\phi) \in \Mp_2(\Z)$ and $\mathfrak{a},\mathfrak{b} \in \C[\mathcal{L}(D)]$. Since $\psi_D$ is invariant under $\overline{\rho}$, we obtain
		\[
		f(M\tau) \cdot \psi_D = \phi(\tau)^{2k}\rho(M,\phi)f(\tau) \cdot \psi_D = \phi(\tau)^{2k}f(\tau) \cdot \overline{\rho}(M,\phi)^{-1}\psi_D = \phi(\tau)^{2k} (f \cdot \psi_D),
		\]
		so $f\cdot\psi_D$ transforms like a modular form of weight $k$ for $\SL_2(\Z)$.
	\end{proof}

We now complete the proof of Proposition~\ref{proposition twisted theta function}.

\begin{proof}[Proof of Proposition~\ref{proposition twisted theta function}]
Consider the (vector-valued and non-twisted) modified theta function for $L(D)$,
\[
\Theta_{L(D)}^*(\tau,P,h) =  \Im(\tau)^{3/2}\sum_{\delta \in L'/L(D)}\sum_{X \in h(\delta+L(D))}\tfrac{1}{|D|}(X,v(P))e\left(\frac{Q(X_v)}{|D|}\tau + \frac{Q(X_{v^\perp})}{|D|}\overline{\tau}\right)\e_\delta.
\]
As a function of $\tau$ it transforms like a modular form of weight $0$ for $\rho_{\mathcal{L}(D)}$ by \cite[Theorem~4.1]{borcherds}. Moreover, it is related to the scalar-valued twisted theta function by
\[
\Theta_{L,\chi_D}^*(\tau,P,h) = \Theta_{L(D)}^*(\tau,P,h)\cdot \psi_D,
\]
which can be checked by a short computation. In particular, by Corollary~\ref{corollary intertwiner} the twisted theta function $\Theta_{L,\chi_D}^*(\tau,P,h)$ transforms like a modular form of weight $0$ for $\SL_2(\Z)$ in $\tau$. This proves the result.
\end{proof}

\subsection{The twisted traces of the Green's function as a theta lift}

Let $n \in \N_0$. For a smooth modular form $f$ of weight $-2n$ for $\SL_2(\Z)$ we define the twisted regularized theta lift
\[
\Phi_{L,\chi_D}^{*,(-2n)}(f,P) = \int_{\mathcal{F}}^{\reg}R_{-2n}^n f(\tau) \Theta_{L,\chi_D}^*(\tau,P)d\mu(\tau).
\]
Here we put $\Theta_{L,\chi_D}^*(\tau,P)=\Theta_{L,\chi_D}^*(\tau,P,1)$ in the twisted theta function \eqref{eq (scalar-valued) D-twisted modified theta function}.

For $m \in \N$ we let
\[
F_{-2n,m}(\tau,s) = \frac{1}{\Gamma(2s)}\sum_{M \in  \Gamma_{\infty}\backslash \SL_2(\Z)}\mathcal{M}_{-2n,s}(4\pi m v)e(-mu)|_{k}M
\]
be the $m$-th Maass Poincar\'e series of weight $-2n$ for $\SL_2(\Z)$, where $\Gamma_{\infty}$ is the subgroup of $\SL_2(\Z)$ generated by $\left(\begin{smallmatrix}
    1 & 1 \\ 0 & 1
\end{smallmatrix}\right)$ and $\mathcal{M}_{k,s}$ is defined in \eqref{eq M Whittaker}. By unfolding against the Poincar\'e series as in the proof of Theorem~\ref{proposition higher Green's function as theta lift}, we obtain the twisted trace of the Green's function as a theta lift.

\begin{theorem}
	Let $n\in \N_0$ and $m \in \N$. For $\Re(s) > 1$ we have
	\[
	\Phi_{L,\chi_D}^{*,(-2n)}(F_{-2n,m}(\, \cdot \,,s),P) = C_n(s)\frac{2\sqrt{2}}{\pi \Gamma(s) \sqrt{|D|}}\tr_{m|D|,\chi_D}\big( G_{2s-1}(\, \cdot \, , P)\big),
	\]
	with the constant $C_n(s) = (4\pi m)^n(s-n)(s-n+1)\cdots (s-1)$.
\end{theorem}

\begin{proof}
	By Lemma~\ref{lemma raising poincare series} we have
	\[
	\Phi_{L,\chi_D}^{*,(-2n)}(F_{-2n,m}(\, \cdot \,,s),P) = C_n(s)\Phi_{L,\chi_D}^{*,(0)}(F_{0,m}(\, \cdot \,,s),P).
	\]
	Unfolding against $F_{0,m}(\tau,s)$ and computing the integral over $u=\Re(\tau)$ as in the proof of \cite[Theorem~2.14]{bruinier}, we obtain
	\begin{align*}
	\Phi_{L,\chi_D}^{*,(0)}(F_{0,m}(\, \cdot \,,s),P) &= \frac{2}{\Gamma(2s)|D|}\sum_{\substack{X \in L \\ Q(X)  =m|D|}}\chi_D(X)(X,v(P)) \\
	& \qquad \times \int_{v = 0}^\infty M_{0,s-1/2}(4\pi m v) \exp\left(-4\pi \frac{Q(X_v)}{|D|}v+2\pi mv \right)v^{1/2}\frac{dv}{v}.
	\end{align*}
	The integral is a Laplace transform and equals
	\[
	(4\pi m)^{-1/2}\Gamma(s+1/2)\left(\frac{Q(X_v)}{m|D|}\right)^{-s-1/2} \ _2 F_1\left(s,s+\frac{1}{2},2s,\frac{m|D|}{Q(X_v)}\right).
	\]
	By Lemma~\ref{lemma hyperbolic distance}, we have
	\[
	(X,v(P)) = \sgn(c_X)\sqrt{2m|D|}\cosh(d(P,P_X)), \qquad \frac{m|D|}{Q(X_v)} = \frac{1}{\cosh(d(P,P_X))^2},
	\]
 where $c_X$ is the bottom left entry of $X$.	Moreover, the hypergeometric function simplifies as in \eqref{eq simplification of hypergeometric function}. As in \eqref{eq hypergeometric equals varphi} we obtain
	\begin{align*}
	\ _2 F_1\left(s,s+\frac{1}{2},2s; \frac{m|D|}{Q(X_v)}\right)
	&= 2^{2s-1}\cosh(d(P,P_X))^{2s}\varphi_{2s-1}(\cosh(d(P,P_X))).
	\end{align*}
 Taking everything together, we arrive at
	\begin{align*}
	\Phi_{L,\chi_D}^{*,(0)}(F_{0,m}(\, \cdot \,,s),P)  &=  \frac{2^{2s-1}\Gamma(s+1/2)\sqrt{2}}{\Gamma(2s)\sqrt{\pi|D|}}\sum_{\substack{X \in L \\ Q(X) = m|D|}}\chi_D(X)\sgn(c_X)\varphi_{2s-1}\big(\cosh(d(P,P_X))\big).
	\end{align*}
	 Note that $D<0$ implies $\chi_D(-X)=-\chi_D(X)$. Hence, by replacing $X$ with $-X$ if $\sgn(c_X) < 0$ we can restrict the sum to $X \in L_{m|D|}^+$ and get a factor of $2$. Moreover, using the Legendre duplication formula $\frac{2^{2s-1}\Gamma(s+1/2)}{\Gamma(2s)} = \frac{\sqrt{\pi}}{\Gamma(s)}$, we obtain
	\begin{align*}
	\Phi_{L,\chi_D}^{*,(0)}(F_{0,m}(\, \cdot \,,s),P)  &=  \frac{2\sqrt{2}}{\Gamma(s)\sqrt{|D|}}\sum_{\substack{X \in L_{m|D|}^+}}\chi_D(X)\varphi_{2s-1}\big(\cosh(d(P,P_X))\big).
	\end{align*}
	Splitting the sum modulo $\Gamma$ gives the stated formula.
\end{proof}

 At $s = n+1$ the function $F_{-2n,m}(\tau) = F_{-2n,m}(\tau,n+1)$ defines a harmonic Maass form of weight $-2n$ for $\SL_2(\Z)$. We obtain the following result.
 
 \begin{theorem}\label{theorem twisted theta lift}
 	Let $m \in \N$. For $n \geq 1$ we have
	\[
	\Phi_{L,\chi_D}^{*,(-2n)}(F_{-2n,m},P) = (4\pi m)^n \frac{2\sqrt{2}}{\pi \sqrt{|D|}}\tr_{m,\chi_D}\big( G_{2n+1}(\, \cdot \, , P)\big).
	\]
 \end{theorem}
 
 \begin{remark}
 	The theorem also holds for $n = 0$. In this case, $F_{0,m}(\tau,s)$ is a Niebur Poincar\'e series which can be analytically continued to $s = 1$ via its Fourier expansion. Arguing as in the proof of \cite[Proposition~2.11]{bruinier} one can check that the twisted theta lift $\Phi_{L,\chi_D}^{*,(0)}(F_{0,m}(\,\cdot\,,s),P)$ is holomorphic at $s = 1$, and agrees with $\Phi_{L,\chi_D}^{*,(0)}(F_{0,m},P)$. 
This is remarkable since it implies that the twisted trace $\tr_{m|D|,\chi_D}\big( G_{s}(\, \cdot \, , P)\big)$ is holomorphic at $s = 1$, although the Green's function $G_{s}(P_1,P_2)$ has a pole at $s  = 1$. Since the residue of the Green's function at $s = 1$ does not depend on $P_1,P_2$, the holomorphicity of $\tr_{m|D|,\chi_D}\big( G_{s}(\, \cdot \, , P)\big)$ at $s = 1$  is equivalent to
	\[
	\tr_{m|D|,\chi_D}(1) = 0.
	\]
	Indeed, this vanishing can also be proved directly. Using Theorem~\ref{shimura theorem} and Lemma~\ref{sum lemma} we see that $\tr_{m|D|,\chi_D}(1)$ is a multiple of
	\[
	\int_{H(\Q)\backslash H(\A_f)}\chi_D(\nu(h))dh,
	\]
	where $\chi_D\circ \nu$ is defined in \eqref{eq:chi-adelic-on_G}, and this integral vanishes since $\chi_D\circ \nu$ is a non-trivial character on the compact group $H(\Q) \backslash H(\A_f)$.
  \end{remark}
 \subsection{A twisted Siegel--Weil formula}\label{section twisted siegel weil}
 
 	In this section we let $M$ be a positive definite even lattice of rank $3$. Moreover, we let $W = M \otimes \Q$ the surrounding rational quadratic space, and $H = \SO(W)$. The classical Siegel--Weil formula describes the integral of the theta function $\Theta_M(\tau,h)$ over $h \in H(\Q) \backslash H(\A_f)$ as an explicit Eisenstein series; see Theorem~\ref{siegel weil formula}. Let $\nu$ denote the spinor norm on $\mathrm{GSpin}_W(\A_f)$, and associated to the field discriminant $D<0$ consider  the quadratic character $\chi_D$ of $\A_f^{\times}$ defined in \eqref{eq:chi-adelic}.   Then the function $\chi_D\circ \nu$ defines a non-trivial quadratic character of $H(\A_f)$ that is trivial on $H(\Q)$, so it makes sense to consider the \emph{twisted Siegel--Weil integral}
	\begin{equation}\label{eq twisted Siegel Weil integral}
	\vartheta_{M,\chi_D}(\tau) = \frac{1}{2}\int_{H(\Q)\backslash H(\A_f)}\Theta_M(\tau,h)\chi_D(\nu(h))dh,
	\end{equation}
 where $dh$ denotes the Tamagawa measure on $H(\A_f)$ normalized such that $\vol(H(\Q)\backslash H(\A_f)) = 2$ (as in Theorem~\ref{siegel weil formula}). This function was studied in detail by Snitz \cite{snitz} in his thesis. It follows from Snitz's work that $\vartheta_{M,\chi_D}(\tau)$ is a holomorphic cusp form of weight $3/2$ with rational Fourier coefficients whose Fourier expansion is supported on a single rational square class. More precisely, his main result implies that $\vartheta_{M,\chi_D}(\tau)$ 
 lies in the distinguished subspace generated by unary theta functions of weight $3/2$. To state his result in a convenient form, recall that $A_{k,\rho_M}$ denotes the space of functions transforming like vector-valued modular forms of weight $k$ under the Weil representation $\rho_M$. For a  lattice $P$ in a one-dimensional positive definite quadratic space $(\Q,\widetilde{Q})$ we let
 \[
 \theta_{P}^*(\tau) = \sum_{r \in P'/P}\sum_{x\in (r+P)}xe^{2\pi i \widetilde{Q}(x) \tau}\e_{r}
 \]
be the holomorphic unary theta function of weight $3/2$ associated with $P$. 

In the following theorem, we assume that the character $\chi_D$ is compatible with the quadratic space $W$ in the sense that at every prime $p$ such that the local character $\chi_{D,p}$ is trivial, we have that $W\otimes \Q_p$ is isotropic. If this compatibility condition is not satisfied, then $\vartheta_{M,\chi_D}(\tau)=0$ (see \cite[Corollary 15]{snitz}) and the following theorem holds trivially.
 
 \begin{theorem}[Twisted Siegel--Weil formula \cite{snitz}]\label{twisted siegel weil formula}
 	Let $M$ be a positive definite even lattice of rank $3$ and let $D<0$ be a field discriminant such that $\chi_D$ is compatible with the quadratic space $M\otimes \Q$ as explained above. Then there exists a positive definite unary lattice $P$ and a linear map $\Phi: A_{k,\rho_P} \to A_{k,\rho_M}$ of the form
  \[
  \Phi(f) =  \sum_{\mu \in M'/M}\bigg(\sum_{r \in P'/P}f_r\left(\tau\right) \cdot c_{r,\mu} \bigg)\e_\mu
  \]
  with rational coefficients $c_{r,\mu}$, such that
  \[
  \vartheta_{M,\chi_D}= \Phi(\theta_P^*).
  \]
\end{theorem}

 \begin{remark}\label{rmk:on-TSW-formula}
    \begin{enumerate}
    \item Snitz's results are stated for the ternary quadratic spaces $(B^0,\mathrm{nr})$ coming from the reduced norm on trace zero elements of a quaternion algebra $B$. For a general quadratic space $(W,Q)$ we have that $(W,2\, \det(W)Q)$ is isometric to $(B^0,\mathrm{nr})$ for some $B$ (see, e.g., \cite[Theorem~5.1.1]{voight}), hence one can apply the results of Snitz in our setting by keeping track of the effect of this re-scaling.
        \item It follows from \cite[Theorem 2]{snitz} that the coefficients $c_{r,\mu}$ can be expressed as products of $p$-adic orbital integrals of compactly supported and locally constant functions with rational values. However, we will not need their precise values. 
        \item Since the coefficients $c_{r,\mu}$ are rational, $\Phi$ also defines a linear map between spaces of modular forms for the dual Weil representations $\overline{\rho}_P$ and $\overline{\rho}_M$.
        \item The lattice $P$ is a lattice in the rational quadratic space $(\Q,2\, \det(M)|D|x^2)$.
    \end{enumerate}
 \end{remark}

  \begin{proof}[Proof of Theorem~\ref{twisted siegel weil formula}]
    Since the results of \cite{snitz} are stated in terms of Schwartz functions on adelic spaces, we give some details on the translation into our setting. Given a lattice $L$ we define $\hat{L}=L\otimes \hat{\Z}$ where $\hat{\Z}=\prod_{p<\infty}\Z_p$. We let $\mathcal{S}(W(\A))$ be the space of Schwartz functions $\varphi =\varphi_{\infty}\otimes \varphi_f$, 
    that is, $\varphi_\infty \in \mathcal{S}(W(\R))$ is smooth and rapidly decreasing and $\varphi_f\in \mathcal{S}(W(\A_f))$ is compactly supported and locally constant. Associated to the quadratic space $(W,Q)$ and the standard non-trivial additive character $\psi_0$ of $\A/\Q$ there is the adelic Weil representation $\omega_{W,\psi_0}=\omega_W$ of $\Mp_2(\A)$ on $\mathcal{S}(W(\A))$. Note that, for $\epsilon=2\, \mathrm{det}(W)$, the representation $\omega_W$ is equivalent to the adelic Weil representation associated to the quadratic space $(W,\epsilon Q)$ and the additive character $\psi(x)=\psi_0(x/\epsilon)$, hence we can apply the results in \cite{snitz} as explained in Remark \ref{rmk:on-TSW-formula}(1). Now, to each $\varphi\in \mathcal{S}(W(\A))$ one can associate a theta function $\Theta_W(g,h,\varphi)$ as in \cite[Equation~(1.1)]{snitz}, which is a function of $g \in \Mp_2(\A)$ and $h \in H(\A_f)$. Choosing $\varphi_\infty = e^{-2\pi Q(x)}$ as the Gaussian,  the finite part $\varphi_f$ as the characteristic function $\1_{\mu}$  of a coset $\mu \in \hat{M'}/\hat{M}\cong M'/M$, and 
    \begin{equation}\label{eq:choice-of-g-in-adelic-theta}
    g = g_\tau=\left(\begin{pmatrix}
        1 & u \\
        0 & 1
    \end{pmatrix}\begin{pmatrix}
        v^{1/2} & 1 \\
        0 & v^{-1/2}
    \end{pmatrix},1\right) \in \Mp_2(\R)
     \end{equation}
    the usual matrix that sends $i$ to $\tau=u+iv\in \H$, we obtain $v^{3/4}$ times  the $\mu$-th component of our holomorphic ternary theta function $\Theta_M(\tau,h)$ defined in \eqref{eqn-theta-L}. 
    
    Next, consider the one-dimensional quadratic space $U = (\Q,\epsilon |D|x^2)$, the corresponding space  of Schwartz functions $\mathcal{S}(U(\A))$, and the adelic Weil representation $\omega_{U,\psi_0}=\omega_U$ of $\Mp_2(\A)$ on $\mathcal{S}(W(\A))$. In a similar way as before, to each $\varphi^0 \in \mathcal{S}(U(\A))$ one can associate a theta function $\Theta_U(g,h^0,\varphi^0)$, which is a function of $g \in \Mp_2(\A)$ and $h^0 \in \O(U)(\A_f)=\prod_{p< \infty}\{\pm 1\}$. If we let $S$ denote the set of primes $p$ such that $W\otimes \Q_p$ is anisotropic, and define the character $\eta_S: \O(U)(\A_f)\to \{\pm 1\}$ as $\eta_S=\bigotimes_{p< \infty}\eta_{p}$ with $\eta_{p}(x_p)=\mathrm{sgn}(x_p)$ if $p\in S$, and $\eta_{p}(x_p)=1$ otherwise, then the results of \cite{snitz} can be summarized by saying that there exists an explicit equivariant (with respect to $\omega_W$ and $\omega_U$) linear map
    \[
    \mathcal{S}(W(\A)) \to \mathcal{S}(U(\A)), \quad \varphi \to \varphi^0=\varphi^{0}_{\infty}\otimes \varphi^{0}_{f}
    \]
    such that the  equality of twisted integrals
    \begin{equation*}
    \int_{H(\Q)\backslash H(\A_f)}\Theta_W(g,h,\varphi)\chi_D(\nu(h))dh = \int_{\O(U)(\Q)\backslash \O(U)(\A_f)} \Theta_U(g,h^0,\varphi^0)\eta_S(h^0)dh^0
    \end{equation*}
    holds. For factorizable functions $\varphi_f=\bigotimes_{p<\infty}\varphi_p$ the image $\varphi_f^0$ is also factorizable $\varphi_f^0=\bigotimes_{p<\infty}\varphi_p^0$. Moreover, our compatibility hypothesis implies  the parity property
    $\varphi^{0}_{p}(-r)=\eta_p(-1)\varphi^{0}_{p}(r)$ valid for every prime $p$ (see \cite[Corollary 38]{snitz}). Then,  as done in \cite[Section 3.2]{snitz} one can rewrite the  above equality as
   \begin{equation}\label{eq:Snitz-identity}
    \int_{H(\Q)\backslash H(\A_f)}\Theta_W(g,h,\varphi)\chi_D(\nu(h))dh = \Theta_U(g,1,\varphi_0).
    \end{equation}
    
    In our setting, we choose $g$ as in \eqref{eq:choice-of-g-in-adelic-theta}, $\varphi_\infty = e^{-2\pi Q(x)}$ as the Gaussian and $\varphi_f=\1_{\mu}$ as the characteristic function of $\mu \in \hat{M}'/\hat{M}$. Then, a direct computation using \cite[Theorem 2 with $x_0\in W(\Q)$ satisfying $\epsilon Q(x_0)=|D|$]{snitz} shows that $\varphi_{\infty}^{0}(r) = c_{\infty}\,re^{-2\pi |D|r^2/\epsilon}$ with $c_{\infty}$ the volume of $T_{x_0}(\R)\backslash H(\R)$ where $T_{x_0}$ is the stabilizer of $x_0$ in $H$. Moreover, we have that $\varphi_f$ is factorizable. As a consequence of \eqref{eq:Snitz-identity}, multiplying both sides by $\e_{\mu}$ and summing over $\mu \in \hat{M}'/\hat{M}$  we get
    \begin{equation}\label{eq:twisted-S-W-1}
    \frac{1}{2}\int_{H(\Q)\backslash H(\A_f)}  \Theta_M(\tau,h)\chi_D(\nu(h))dh = c_{\infty} \sum_{\mu \in M'/M} \sum_{r\in U(\Q)} re^{-2\pi \epsilon|D|r^2 \tau} \1_{\mu}^{0}(\epsilon r)\e_{\mu}.
    \end{equation}
    Since for every $\mu \in M'/M$ we have $\1_{\mu}^{0}\in \mathcal{S}(U(\A_f))$, there exists a lattice $P$ in $U(\Q)$ such that $\1_{\mu}^{0}(\epsilon r)$ is supported on $\hat{P'}$ and constant on each class in $\hat{P'}/\hat{P}\cong P'/P$.  Hence, the map $\varphi \to \varphi_0(\epsilon r)$ restricts to a map $\mathcal{S}_M \to \mathcal{S}_{P}$, where $\mathcal{S}_M$ and $\mathcal{S}_P$ are the subspaces of $\mathcal{S}(W(\A_f))$ and $\mathcal{S}(U(\A_f))$ spanned by characteristic functions $\{\1_{\mu}:\mu \in M'/M\}$ and $\{\1_{r}:r \in P'/P\}$, respectively. Notice that $\mathcal{S}_M$ and $\mathcal{S}_P$ can be identified with the group rings $\C[M'/M]$ and $\C[P'/P]$ via $\1_{\mu} \mapsto \e_{\mu}$ and $\1_{r} \mapsto \e_{r}$, respectively. Since the dual Weil representations $\overline{\rho}_M$ and $\overline{\rho}_{P}$ acting on $\mathcal{S}_M$ and $\mathcal{S}_{P}$ are induced by the adelic Weil representations $\omega_W$ and $\omega_{P}$, respectively (see, e.g., \cite[Section 2]{bruinieryang}), the map $\mathcal{S}_M \to \mathcal{S}_{P}$ is equivariant with respect to $\overline{\rho}_M$ and $\overline{\rho}_{P}$. Using this, it is easy to check that the map
    \[
    \sum_{r \in P'/P}f_r(\tau)\e_r \mapsto \sum_{\mu \in M'/M}\bigg(\sum_{r \in P'/P}f_r(\tau)\cdot \1_{\mu}^0(\epsilon r)\bigg) \e_{\mu}
    \]
    sends $A_{k,\rho_{P}}$ to $A_{k,\rho_M}$. Now, the explicit integral formula given in \cite[Theorem 2]{snitz} shows that $\1_{\mu}^0(\epsilon r)=c_fc_{\mu,r}$ with $c_{\mu,r}$ a positive rational number and $c_f$ the volume of $T_{x_0}(\A_f)\backslash H(\A_f)$. Finally, noting that $c_{\infty}c_f=1$ due to the normalization of measures used by Snitz (see \cite[p.~440]{snitz}), we get that the right-hand side of \eqref{eq:twisted-S-W-1} equals $\Phi(\theta_P^*)$. This finishes the proof.
  \end{proof}

 \begin{example}\label{twisted siegel weil example}
 	Let us choose $D = -4$, and consider the lattice $M = \Z^3$ with $Q(x,y,z) = 4x^2 + y^2 + z^2$. We have $M'/M \cong (\Z/8\Z) \times (\Z/2\Z)^2$, so we write its elements in the form $(\frac{x}{8}, \frac{y}{2},\frac{z}{2})$. Note that the quadratic space $W\otimes \Q_p$ is anisotropic only when $p=2$, hence $\chi_{D}$ is compatible with $M\otimes \Q$. The twisted Siegel--Weil formula asserts that $\vartheta_{M,\chi_D}$ is a cusp form whose Fourier expansion is supported on rational squares. 
  Indeed, using Williams' weilrep package for sage math \cite{williamsweilrep} one can check that the space $S_{3/2,M}$ is one-dimensional and spanned by the form
	\begin{eqnarray*}
	f &=& \theta_{4,1}^*\left(\e_{(\frac{1}{8},0,0)}+\e_{(\frac{7}{8},0,0)}-\e_{(\frac{5}{8},\frac{1}{2},\frac{1}{2})}-\e_{(\frac{7}{8},\frac{1}{2},\frac{1}{2})}\right) \\
	& &+ \theta_{4,2}^*\left(\e_{(0,\frac{1}{2},0)}+\e_{(0,0,\frac{1}{2})}-\e_{(\frac{1}{2},\frac{1}{2},0)}-\e_{(\frac{1}{2},0,\frac{1}{2})}\right) \\
	& & + \theta_{4,3}^* \left(\e_{(\frac{1}{8},\frac{1}{2},\frac{1}{2})}+\e_{(\frac{7}{8},\frac{1}{2},\frac{1}{2})}-\e_{(\frac{3}{8},0,0)}-\e_{(\frac{5}{8},0,0)}\right),
	\end{eqnarray*}
 where $\theta_{N,r}^* = \sum_{n \equiv r (2N)}nq^{n^2/4N}$ is a unary theta function of weight $3/2$, and a multiple of a component of $\theta_P^*$ for a suitable unary lattice $P$. In our case, $P = (\Z,4x^2)$. The sageMath code
	\begin{verbatim}
    from weilrep import *
    w = WeilRep(diagonal_matrix([-8,-2,-2]))
    w.cusp_forms_basis(3/2,prec=10) \end{verbatim}
	will print the components of $f$ up to $q^{10}$. The cusp form $\vartheta_{M,\chi_D}$ is a rational multiple of $f$. 
 \end{example}
 
 	\begin{theorem}\label{theorem twisted siegel weil preimage} Let the notation be as in Theorem~\ref{twisted siegel weil formula}. There exists a harmonic Maass form $\widetilde{\vartheta}_{M,\chi_D}$ of weight $1/2$ for $\rho_{M^{-}}$ with
		\[
		\xi_{1/2}\widetilde{\vartheta}_{M,\chi_D} = \frac{1}{2}\vartheta_{M,\chi_D},
		\]
		such that the holomorphic part of $\sqrt{2\det(M)|D|}\,\widetilde{\vartheta}_{M,\chi_D}$ has \emph{rational} Fourier coefficients.
 \end{theorem}

 \begin{proof}
    Since $P$ is a lattice in $(\Q,2\,\det(M)|D|x^2)$, by \cite[Theorem~1.1]{lischwagenscheidt} there exists a harmonic Maass form $\widetilde{\theta}_{P}^*$ of weight $1/2$ for the dual Weil representation $\overline{\rho}_P$ with 
    \[
    \xi_{1/2}\widetilde{\theta}_{P}^* = \frac{1}{2\sqrt{2\det(M)|D|}}\theta_{P}^*,
    \]
    such that the holomorphic part of $\widetilde{\theta}_{P}^*$ has rational coefficients. Now we choose $\widetilde{\vartheta}_{M,\chi_D} = \sqrt{2\det(M)|D|}\,\Phi(\widetilde{\theta}_{P}^*)$.
 \end{proof}

 \subsection{Splitting of the Siegel theta function, twisted versions}
 
 In this section we compute the twisted traces of the Siegel theta functions $\Theta_L(\tau,P)$ and $\Theta_{L,\chi_D}^*(\tau,P)$, similarly as in Theorem~\ref{theorem splitting theta}. Throughout, we fix $\mu \in L'/L$ and $m \in \Z + Q(\mu)$ with $m > 0$. Recall that $L_{m,\mu}^{+,0}$ denotes the set of primitive positive definite binary hermitian forms $X \in L+ \mu$ with determinant $Q(X) = m$. Throughout we consider a fixed primitive positive definite vector $X_0 \in L_{m,\mu}^{+,0}$. As before, we define sublattices
		\begin{align}\label{eq splitting lattice}
		P = L \cap (\Q X_0), \qquad N = L \cap (\Q X_0)^\perp,
		\end{align}
		which are one-dimensional positive definite and three-dimensional negative definite. Let $N^- = (N,-Q)$. As in Section~\ref{section twisted theta function} we let $P(D) = (DP, Q_D)$ and $N(D) = (DN, Q_D)$ with $Q_D(X) = \frac{1}{|D|}Q(X)$ be the corresponding rescaled lattices. Note that we have
		\[
		P(D) = L(D) \cap (\Q X_0), \qquad N(D) = L(D) \cap (\Q X_0)^\perp.
		\]
  As in Section \ref{sec orthogonal and spinor groups} we let $H=\{g\in G:gX_0=X_0\}$ be the stabilizer of $X_0$ in $G$, which we now identify with $\SO(W)$ where $W = (\Q X_0)^\perp$ is a positive definite three-dimensional quadratic space with quadratic form $-Q$.
  
 We start with the twisted traces of the (non-twisted) Siegel theta function $\Theta_L(\tau,P)$, which was defined in Section~\ref{section theta functions}.
 
 \begin{theorem}\label{theorem splitting theta twisted Siegel}
		Fix a primitive vector $X_0 \in L_{m|D|,0}^{+,0}$. Then we have
		\[
		\tr_{m|D|,\chi_D}^0\left(\Theta_L(\tau,\, \cdot \,)\right) = \chi_D(X_0)\tr_{m|D|,0}^0(1) \cdot \left(\Theta_{P}(\tau) \otimes \overline{\vartheta_{N^-,\chi_D}(\tau)}v^{3/2}\right)^L,
		\]
		where $\Theta_{P}(\tau)$ is the weight $1/2$ holomorphic theta function for $P$ and $\vartheta_{N^-,\chi_D}(\tau)$ is the weight $3/2$ cusp form described in Section~\ref{section twisted siegel weil}, and the superscript $L$ denotes the operator defined in Section~\ref{section modular forms}.
	\end{theorem}
	
	\begin{proof}
		The arguments are very similar to the proof of Theorem~\ref{theorem splitting theta}, where the non-twisted trace of $\Theta_L(\tau,P)$ was computed. Using Theorem~\ref{shimura theorem}, we have
		\[
		\tr_{m|D|,\chi_D}^0(\Theta_L(\tau,\,\cdot\,)) = \sum_{X \in \Gamma \backslash L_{m|D|,0}^{+,0}}\frac{\chi_D(Q)}{|\Gamma_X|}\Theta_L(\tau,P_X) = \sum_{h \in H(\Q)\backslash H(\A_f)/K}\frac{\chi_D(g^{-1}X_0)}{|\Gamma_h|}\Theta_L(\tau,P_0,h),
		\]
		where $P_0 \in \H^3$ is the special point corresponding to $X_0$, and we wrote $h = gu$ with $g \in G(\Q)$ and $u \in U$. By Proposition~\ref{proposition character invariance} we have
		\begin{align*}
		\chi_D(g^{-1}X_0) &= \chi_D(uX_0) = \chi_D(\nu(u))\chi_D(X_0) = \chi_D(\nu(h))\chi_D(X_0).
		\end{align*}
		Here we wrote $g^{-1} = uh^{-1}$, used that $h^{-1}X_0 = X_0$ since $H$ is the stabilizer of $X_0$, and $\chi_D(\nu(g^{-1})) = 1$ since $g \in G(\Q)$ and $\chi_D$ is trivial on $\Q^+$. Using Lemma \ref{sum lemma} we obtain as in the proof of Theorem~\ref{theorem splitting theta} that
		\begin{align*}
		\tr_{m|D|,\chi_D}^0(\Theta_L(\tau,\,\cdot\,)) &= \chi_D(X_0)\sum_{h \in H(\Q)\backslash H(\A_f)/K}\frac{\chi_D(\nu(h))}{|\Gamma_h|}\Theta_L(\tau,P_0,h) \\
		&= \frac{\chi_D(X_0)}{\vol(K)}\int_{H(\Q) \backslash H(\A_f)}\chi_D(\nu(h))\Theta_L(\tau,P_0,h)dh \\
		&= \frac{\chi_D(X_0)}{\vol(K)}\left( \Theta_P(\tau) \otimes v^{3/2}\overline{\int_{H(\Q) \backslash H(\A_f)}\chi_D(\nu(h))\Theta_{N^-}(\tau,P_0,h)dh}\right)^L.
		\end{align*}
  Notice that $\chi_D\circ \nu$ is precisely the non-trivial quadratic character on the twisted Siegel--Weil integral \eqref{eq twisted Siegel Weil integral}. Hence, the integral is the function $2\vartheta_{N^-,\chi_D}$. Finally, we plug in $\frac{2}{\vol(K)} = \tr_{m|D|}^0(1)$ to finish the proof.
	\end{proof}
	
	The crucial difference between Theorem~\ref{theorem splitting theta} and its twisted version Theorem~\ref{theorem splitting theta twisted Siegel} is the fact that the Eisenstein series $E_{3/2,N^-}$ is replaced by the cusp form $\vartheta_{N^-,\chi_D}$.
	
	Next, we consider the twisted and non-twisted traces of the twisted modified theta function $\Theta_{L,\chi_D}^*(\tau,P)$. We will use the same notation as in Section~\ref{section twisted theta function}.

	\begin{theorem}\label{theorem splitting theta twisted modified Siegel}
		Fix a primitive vector $X_0 \in L_{m,\mu}^{+,0}$. Then we have
		\[
		\tr_{m,\mu}^0\left(\Theta_{L,\chi_D}^*(\tau,\, \cdot \,)\right) = \chi_D(X_0)\tr_{m,\mu}^0(1)\left( \Theta_{P(D)}^*(\tau) \otimes \overline{\vartheta_{N(D)^-,\chi_D}(\tau)}v^{3/2}\right)^{L(D)}\cdot \psi_D.
		\]
		Similarly, if $X_0 \in L_{m|D|,0}^{+,0}$, then we have
		\[
		\tr_{m|D|,\chi_D}^0\left(\Theta_{L,\chi_D}^*(\tau,\, \cdot \,)\right) = \chi_D(X_0)\tr_{m|D|,0}^0(1)\left( \Theta_{P(D)}^*(\tau) \otimes \overline{E_{3/2,N(D)^-}(\tau)}v^{3/2}\right)^{L(D)}\cdot \psi_D,
		\]
		Here,
		\begin{align}\label{eq weight 3half theta}
		\Theta_{P(D)}^*(\tau) = \frac{1}{\sqrt{2Q(X_0)}}\sum_{X \in P(D)'}\tfrac{1}{|D|}(X,X_0)e\left(\frac{Q(X)}{|D|}\tau\right)\e_{X+P(D)}
		\end{align}
		is the weight $3/2$ holomorphic theta function for $P(D)$, and $\psi_{D}$ is the invariant vector defined in Lemma~\ref{lemma intertwiner}.
	\end{theorem}

	\begin{proof}
		Let $P_0 \in \H^3$ be the point corresponding to $X_0$. Using Proposition~\ref{proposition character invariance} one can check that the twisted modified Siegel theta function satisfies
		\[
		\Theta_{L,\chi_D}^*(\tau,g^{-1}P_0) = \chi_D(\nu(h))\Theta_{L,\chi_D}^*(\tau,P_0,h)
		\]
		for $h = gu \in H(\A_f)$ with $g \in G(\Q)$ and $u \in U$. Moreover, we can write $\Theta_{L,\chi_D}^*(\tau,P,h) = \Theta_{L(D)}^*(\tau,P,h)\cdot \psi_D$. Using these facts, the proof of the theorem is analogous to the proof of Theorem~\ref{theorem splitting theta twisted Siegel}. We leave the details to the reader.
	\end{proof}
	
	\begin{remark}
		If $P(D)$ is spanned by $rX_0$ for some $r \in \Q$, and we put $N_0= \frac{1}{|D|}Q(rX_0)$, then we have $P(D) \cong (\Z,N_0x^2)$ and $P(D)'/P(D) \cong \Z/2N_0\Z$. Hence, we can write more explicitly
		\[
		\Theta_{P(D)}^*(\tau)  = \frac{\sqrt{2N_0}}{\sqrt{|D|}}\sum_{\rho \!\!\!\!\! \pmod {2N_0}}\sum_{\substack{n\in \Z \\ n \equiv \rho \!\!\!\!\! \pmod {2N_0}}}nq^{n^2/4N_0}\e_\rho.
		\]
	\end{remark}

\subsection{Twisted traces of Green's functions}\label{section twisted traces}

We are now ready to give an explicit evaluation of the twisted double traces of the Green's function, similar to Theorem~\ref{theorem evaluation Green's function}. We can either twist both traces, or only one of them. Moreover, we can consider the Green's function $G_s$ at even or odd integer values for $s$. Throughout this section, $X_0$ denotes a fixed primitive positive definite vector in $L_{m',\mu'}^{+,0}$ or $L_{m'|D|,0}^{+,0}$, which defines sublattices $P$ and $N$ as in \eqref{eq splitting lattice}.

We start with the partially-twisted double trace of $G_{s}$ at even integral $s = 2n$ for $n \geq 1$, where one of the traces is twisted, and the other trace is non-twisted.

\begin{theorem}\label{theorem evaluation Green's function single twist even}
	Let $n \in \N$ and let 
	\[
	f = \sum_{\mu \in L'/L}\sum_{m \in \Z-Q(\mu)}a_f(m,\mu)q^m \e_\mu \in M_{1-2n,L^-}^!
	\]
	be a weakly holomorphic modular form of weight $1-2n$ for the Weil representation $\overline{\rho}_L$. Let $m' \in \N$ such that $a_f(-m'|D|r^2,0) = 0$ for all integers~$r\geq 1$. Then we have
	\begin{align*}
	&\frac{1}{2}\sum_{\mu \in L'/L}\sum_{m > 0}m^{n-1/2}a_f(-m,\mu)\tr_{m'|D|,\chi_D}^0\tr_{m,\mu}(G_{2n})  \\
	&\qquad = \frac{4^n \pi}{\binom{2n}{n}} \chi_D(X_0) \tr_{m'|D|,0}^0(1)\CT\left(f_{P \oplus N} \cdot \left[\Theta_{P},\widetilde{\vartheta}_{N^-,\chi_D}^+\right]_n \right),
	\end{align*}
	where $\widetilde{\vartheta}_{N^-,\chi_D}^+ $ denotes the holomorphic part of $\widetilde{\vartheta}_{N^-,\chi_D}$, and $[\cdot,\cdot]_n$ denotes the $n$-th Rankin--Cohen bracket as defined in Section~\ref{section rankin cohen brackets}, with $k = \ell = 1/2$.
\end{theorem}

\begin{proof}
	Note that the character $\chi_D$ is compatible with the quadratic space $W=N^-\otimes \Q$ because of Remark \ref{rem:primes_in_traces}. Using Theorem~\ref{theorem splitting theta twisted Siegel} and Theorem~\ref{theorem twisted siegel weil preimage} the proof is analogous to the proof of Theorem~\ref{theorem evaluation Green's function}.
\end{proof}

\begin{theorem}~\label{theorem evaluation Green's function single twist even II}
	Let  $f\in M_{1-2n,L^-}^! $ and $m'\in \N$ be  as in Theorem~\ref{theorem evaluation Green's function single twist even}, and suppose that the coefficients $a_{f}(-m,\mu)$ for $m > 0$ are rational. Then the linear combination of partially twisted double traces
	\[
	\sum_{\mu \in L'/L}\sum_{m > 0}m^{n-1/2}a_f(-m,\mu)\tr_{m'|D|,\chi_D}^0\tr_{m,\mu}(G_{2n}) 
	\]
	is a rational multiple of $\pi \sqrt{m'}$.
\end{theorem}

\begin{proof}
	Since the space $M_{1-2n,L^-}^!$ has a basis of forms with rational coefficients by a result of McGraw \cite{mcgraw}, a weakly holomorphic modular form with rational principal part of negative weight has only rational coefficients. Therefore, $f$ (hence $f_{P \oplus N}$) has rational Fourier coefficients. Now, by Theorem~\ref{theorem twisted siegel weil preimage}, we can assume that the coefficients of~$\widetilde{\vartheta}_{N^-,\chi_D}^+$  are rational multiples of $\sqrt{2\det(N^-)|D|}$. We have seen in Remark~\ref{remark determinant N} that $\det(N^-) = 2|D|m'$. Since~$\Theta_P$ has rational coefficients, the right-hand side of the formula in Theorem~\ref{theorem evaluation Green's function single twist even} is a rational multiple of $\pi \sqrt{m'}$.
\end{proof}

As in the case of non-twisted double traces, one can rephrase Theorem~\ref{theorem evaluation Green's function single twist even II} in terms of linear combinations of partially-twisted traces with coefficients coming from rational relations for spaces of cusp forms. As a consequence, we obtain the following corollary.

\begin{corollary}\label{corollary that implies theorem 1.4(2) n even}
Let~$\{\lambda(t)\}_{t\in \N}$ be a rational relation for~$S^+_{1+2n}(\Gamma_0(|D|),\chi_D)$, and let $m' \in \N$ be such that $\lambda (m'|D|^2r^2) = 0$ for all integers~$r\geq 1$. Then the linear combination of partially-twisted double traces
	\[
	\sum_{m > 0}m^{n-1/2}\lambda(m|D|)\tr_{m'|D|,\chi_D}^0\tr_{m}(G_{2n}) 
	\]
	is a rational multiple of $\pi \sqrt{m'}$.    
\end{corollary}

Next, we compute the partially-twisted double trace of $G_s$ at odd integral $s = 2n+1$.

\begin{theorem}\label{theorem evaluation Green's function single twist odd}
	Let $n \in \N$ and let 
	\[
	f = \sum_{m \in \Z}a_f(m)q^m \in M_{-2n}^!(\SL_2(\Z))
	\]
	be a weakly holomorphic modular form of weight $-2n$ for $\SL_2(\Z)$. Let $\mu' \in L'/L$ and $-m' \in \Z-Q(\mu')$ with $m' > 0$ such that $a_f(-m'r^2/|D|) = 0$ for all integers~$r\geq 1$ if $\mu' = 0$. Then we have
	\begin{align*}
		&\frac{1}{2}\sum_{m > 0}m^{n}a_f(-m)\tr_{m',\mu'}^0 \tr_{m|D|,\chi_D}(G_{2n+1}) \\
		&\qquad = \frac{4^n \pi}{\binom{2n}{n}}\tr_{m',\mu'}^0(1)\chi_D(X_0) \sqrt{\frac{|D|}{2}}\CT\left(f \cdot \left[\Theta_{P(D)}^*, \widetilde{\vartheta}_{N(D)^-,\chi_D}\right]_n^{L(D)} \cdot \psi_D \right),
	\end{align*}
	where $\Theta_{P(D)}^*$ is the weight $3/2$ unary theta function defined in \eqref{eq weight 3half theta}, $\widetilde{\vartheta}_{N(D)^-,\chi_D}^+ $ denotes the holomorphic part of $\widetilde{\vartheta}_{N(D)^-,\chi_D}$, and $[\cdot,\cdot]_n$ denotes the $n$-th Rankin--Cohen bracket as defined in Section~\ref{section rankin cohen brackets}, with $k = 3/2$ and $\ell = 1/2$. Moreover, the superscript $L(D)$ denotes the map defined in Section~\ref{section modular forms}, and $\psi_D$ is the invariant vector defined in Lemma~\ref{lemma intertwiner}.
\end{theorem}

\begin{proof}
	The proof is again very similar to the proof of Theorem~\ref{theorem evaluation Green's function}. However, this time we first write the $m|D|$-th twisted trace of $G_{2n+1}$ as a twisted theta lift using Theorem~\ref{theorem twisted theta lift}, and then compute the non-twisted trace in the second variable using Theorem~\ref{theorem splitting theta twisted modified Siegel}. For brevity, we omit the details of the computation.
\end{proof}

\begin{theorem}\label{theorem evaluation Green's function single twist odd II}
	Let the notation be as in Theorem \ref{theorem evaluation Green's function single twist odd}, and suppose that the coefficients $a_{f}(-m)$ for $m > 0$ are rational. Then the linear combination of partially-twisted double traces
	\[
	\sum_{m > 0}m^{n}a_f(-m)\tr_{m',\mu'}^0\tr_{m|D|,\chi_D}(G_{2n+1}) 
	\]
	is a rational multiple $\pi$.
\end{theorem}

\begin{proof}
	The coefficients of the unary theta function $\Theta_{P(D)}^*$ are rational multiples of $\sqrt{2m'}$ since $Q(X_0) = m'$. Moreover, by Theorem~\ref{theorem twisted siegel weil preimage} we can choose $\widetilde{\vartheta}_{N(D)^-,\chi_D}^+$ such that its coefficients are rational multiples of $\sqrt{2\det(N(D)^-)|D|}$. Now $\det(N(D)^-) = 2m'|D|^4$ by Remark~\ref{remark determinant N}. This implies the stated result.
\end{proof}

We obtain the following corollary.

\begin{corollary}\label{corollary that implies theorem 1.4(2) n odd}
	Let~$\{\lambda(t)\}_{t\in \N}$ be a rational relation for~$S_{2+2n}(\SL_2(\Z))$, and let $m' \in \N$ be such that $\lambda (m'r^2/|D|) = 0$ for all integers~$r\geq 1$. Then the linear combination of partially-twisted double traces
	\[
	\sum_{m > 0}m^{n}\lambda(m)\tr_{m'}^0\tr_{m|D|,\chi_D}(G_{2n+1}) 
	\]
	is a rational multiple $\pi$.
\end{corollary}

Note that part~$(2)$ of Theorem~\ref{main theorem twisted} is a direct consequence of Corollaries~\ref{corollary that implies theorem 1.4(2) n even} and~\ref{corollary that implies theorem 1.4(2) n odd}.

Finally, we evaluate the doubly-twisted double trace of $G_s$ at odd integral $s = 2n+1$. 

\begin{theorem}\label{theorem evaluation Green's function double twist odd}
	Let $n \in \N$ and let 
	\[
	f = \sum_{m \in \Z}a_f(m)q^m \in M_{-2n}^!(\SL_2(\Z))
	\]
	be a weakly holomorphic modular form of weight $-2n$ for $\SL_2(\Z)$. Let $m' \in \N$ such that $a_f(-m'r^2) = 0$ for all integers~$r\geq 1$. Then we have
	\begin{align*}
		&\frac{1}{2}\sum_{m > 0}m^{n}a_f(-m)\tr_{m'|D|,\chi_D}^0 \tr_{m|D|,\chi_D}(G_{2n+1}) \\
		&\qquad = \frac{4^n \pi}{\binom{2n}{n}}\tr_{m',\mu'}^0(1)\chi_D(X_0)\sqrt{\frac{|D|}{2}}\CT\left(f \cdot \left[\Theta_{P(D)}^*, \widetilde{E}_{1/2,N(D)}\right]_n^{L(D)} \cdot \psi_D \right),
	\end{align*}
	where $\Theta_{P(D)}^*$ is the weight $3/2$ unary theta function defined in \eqref{eq weight 3half theta}, $\widetilde{E}_{1/2,N}^+ $ denotes the holomorphic part of $\widetilde{E}_{1/2,N}$, and $[\cdot,\cdot]_n$ denotes the $n$-th Rankin--Cohen bracket as defined in Section~\ref{section rankin cohen brackets}, with $k = 3/2$ and $\ell = 1/2$. Moreover, the superscript $L(D)$ denotes the map defined in Section~\ref{section modular forms}, and $\psi_D$ is the invariant vector defined in Lemma~\ref{lemma intertwiner}.
\end{theorem}

\begin{proof}
	Write the twisted trace of $G_{2n+1}$ as a twisted theta lift using Theorem~\ref{theorem twisted theta lift} and evaluate the second twisted trace using Theorem~\ref{theorem splitting theta twisted modified Siegel}. The computation is analogous to the proof of Theorem~\ref{theorem evaluation Green's function}.
\end{proof}

\begin{theorem}
	Let the notation be as in Theorem \ref{theorem evaluation Green's function double twist odd}, and suppose that the coefficients $a_{f}(-m)$ for $m > 0$ are rational. Then the linear combination of twisted double traces
	\[
	\sum_{m > 0}m^{n}a_f(-m)\tr_{m'|D|,\chi_D}^0\tr_{m|D|,\chi_D}(G_{2n+1}) 
	\]
	is a rational linear combination of $\log(p)$ for some primes $p$ and $\log(\varepsilon_\Delta)/\sqrt{\Delta}$ for some fundamental discriminants $\Delta > 0$.
\end{theorem}

\begin{proof}
	The coefficients of $\Theta_{P(D)}^*$ are rational multiples of $\sqrt{2|D|m'}$ since $Q(X_0) = m'|D|$. Note that $|N(D)'/N(D)| = 2m'|D|^4$. Hence, the coefficients of $\widetilde{E}_{1/2,N(D)}^+$ are of the form $\frac{\sqrt{m'}}{\pi}$ times a rational number times $\log(p)$ or $\log(\varepsilon_\Delta)/\sqrt{\Delta}$. This implies the claimed statement.
\end{proof}

The following corollary implies part $(1)$ of Theorem~\ref{main theorem twisted}.

\begin{corollary}\label{corollary that implies theorem 1.4(1)}
	Let~$\{\lambda(t)\}_{t\in \N}$ be a rational relation for~$S_{2+2n}(\SL_2(\Z))$, and let $m' \in \N$ be such that $\lambda (m'r^2) = 0$ for all integers~$r\geq 1$. Then the linear combination of partially-twisted double traces
	\[
	\sum_{m > 0}m^{n}\lambda(m)\tr_{m'|D|,\chi_D}^0\tr_{m|D|,\chi_D}(G_{2n+1})
	\]
	is a rational linear combination of $\log(p)$ for some primes $p$ and $\log(\varepsilon_\Delta)/\sqrt{\Delta}$ for some fundamental discriminants $\Delta > 0$. 
\end{corollary}

\subsection{Example: The value of the Green's function at an individual special point}\label{section example}
In some special cases it is possible to combine our formulas for twisted and non-twisted double traces to obtain values of Green's functions at individual points. We demonstrate this in the following example. We consider the field $\Q(i)$, and take $n = 1$, that is, we evaluate the Green's function over $\Q(i)$ at $s = 2$. We have $S_{3,L} =\{ 0\}$, so we can take $f = F_{-1,m,\mu}$ as a Maass Poincar\'e series in Theorem~\ref{theorem evaluation Green's function}. Recall that $L'/L \cong \Z/2\Z \times \Z/2\Z$.

	We take the determinants $m = 1$ and $m' = 4$, and $\mu =\mu'= (0,0)$ in both cases. We first compute the non-twisted double trace, using Theorem~\ref{theorem evaluation Green's function}. The computation is similar to Example~\ref{example 1}. We need to simplify \eqref{example 1 step 1}. Now $\tr_{m',\mu'}^0$ has two summands, corresponding to
	\[
	X_0 = \begin{pmatrix}4 & 0 \\ 0 & 1 \end{pmatrix}, \qquad X_1 = \begin{pmatrix}3 & 1+i \\ 1-i & 2 \end{pmatrix},
	\]
	with stabilizers of size $2$ each, and corresponding points $P_0 = 2j$ and $P_1 = \frac{1+i}{2}+j$. Hence we have $\tr_{m',\mu'}^0(1) = 1$, and \eqref{example 1 step 1} becomes
	\begin{align*}\label{example 2 step 2}
	\frac{1}{8}G_2\left(j,2j \right) + \frac{1}{8}G_2\left(j,\frac{1+i}{2}+j \right) = 2\pi \CT\left(f_{P \oplus N} \cdot \left[\Theta_{P},\widetilde{E}_{1/2,N}^+\right]_1 \right).
	\end{align*}
	Note that the factor $8 = 4\cdot 2$ in the denominator on the left is the product of the orders of the stabilizers of $j$ and $2j$, respectively.
	
	The lattices $P$ and $N$ are given by
	\begin{align*}
	P &= L \cap (\Q X_0) = \left\{n\begin{pmatrix}4 & 0 \\ 0 & 1 \end{pmatrix}: n \in \Z \right\}, \\
	N &= L \cap (\Q X_0)^\perp = \left\{\begin{pmatrix}4a & b \\ \overline{b} & -a \end{pmatrix}: a \in \Z, b \in \Z[i] \right\},
	\end{align*}
	with dual lattices
	\begin{align*}
	P' &= \left\{\frac{n}{8}\begin{pmatrix}4 & 0 \\ 0 & 1 \end{pmatrix}: n \in \Z\right\}, \\
	N' &= \left\{\frac{1}{8}\begin{pmatrix}4a & 4b \\ 4\overline{b} & -a \end{pmatrix}: a \in \Z, b \in \Z[i]\right\}.
\end{align*}
	For $\alpha \in P'$ and $\beta \in N'$ in the form above, we have $\alpha + \beta \in L'$ if and only if $a \equiv n \pmod 8$. We have $P'/P \cong \Z/8\Z$ and $N'/N \cong \Z/8\Z \times (\Z/2\Z)^2$, so we will write the elements of $(P \oplus N)'/(P \oplus N)$ as $(n,(a,b_1,b_2))$ with $n,a \in \Z/8\Z$ and $b_1,b_2 \in \Z/2\Z$.
	
	Since we take $f = F_{-1,1,(0,0)} = 2q^{-1}\e_0 + \dots$, we have
	\[
	\CT\left(f_{P \oplus N} \cdot \left[\Theta_{P},\widetilde{E}_{1/2,N}^+\right]_1 \right) = 2\sum_{a \!\!\!\!\! \pmod 8}c_{\left[\Theta_{P},\widetilde{E}_{1/2,N}^+\right]_1}\big(1,(a,(a,0,0))\big).
	\]
	We have
	\[
	\Theta_P(\tau) = \sum_{a\!\!\!\!\! \pmod 8}\sum_{\substack{n \in \Z \\ n \equiv a \!\!\!\!\! \pmod 8}}q^{n^2/16}\e_a.
	\]
	Note that the coefficient $c_{\left[\Theta_{P},\widetilde{E}_{1/2,N}^+\right]_1}\big(1,(a,(a,0,0))\big)$  is invariant under $a \mapsto -a \pmod 8$, so we get
	\begin{eqnarray*}
	& &\CT\left(f_{P \oplus N} \cdot \left[\Theta_{P},\widetilde{E}_{1/2,N}^+\right]_1 \right)\\
 &= & 2\bigg(c_{\left[\Theta_{P},\widetilde{E}_{1/2,N}^+\right]_1}\big(1,(0,(0,0,0))\big) +2\sum_{a = 1}^3c_{\left[\Theta_{P},\widetilde{E}_{1/2,N}^+\right]_1}\big(1,(a,(a,0,0))\big) \\
 & & + c_{\left[\Theta_{P},\widetilde{E}_{1/2,N}^+\right]_1}\big(1,(4,(4,0,0))\bigg) \\
	&=& 2\bigg(-\frac{1}{2}c_{\widetilde{E}_{1/2,N}^+}\big(1,(0,0,0)\big)  +2  \bigg(  -\frac{14}{32}c_{\widetilde{E}_{1/2,N}^+}\left(\frac{15}{16},(1,0,0)\right)   -\frac{1}{4}c_{\widetilde{E}_{1/2,N}^+}\left(\frac{3}{4},(2,0,0)\right) \\
 & & +\frac{1}{16}c_{\widetilde{E}_{1/2,N}^+}\left(\frac{7}{16},(3,0,0)\right)\bigg)  + c_{\widetilde{E}_{1/2,N}^+}\left(0,(4,0,0)\right) \bigg).
	\end{eqnarray*}
	The relevant coefficients of the Eisenstein series $\widetilde{E}_{1/2,N}^+$ can be computed using Theorem~\ref{eisenstein series fourier expansion}, and are given by
	\begin{align*}
	c_{\widetilde{E}_{1/2,N}^+}\big(1,(0,0,0)\big) &= -\frac{2\log(2)}{\pi}, \\
	c_{\widetilde{E}_{1/2,N}^+}\left(\frac{15}{16},(1,0,0)\right) &= -\frac{2L(\chi_{60},1)}{\pi} = -\frac{8\log(4+\sqrt{15})}{\sqrt{60}\pi},\\
	c_{\widetilde{E}_{1/2,N}^+}\left(\frac{3}{4},(2,0,0)\right) &= -\frac{2L(\chi_{12},1)}{\pi} = -\frac{2\log(7+2\sqrt{12})}{\sqrt{12}\pi},\\
	c_{\widetilde{E}_{1/2,N}^+}\left(\frac{7}{16},(3,0,0)\right) &= -\frac{2L(\chi_{28},1)}{\pi} = -\frac{2\log(127+24\sqrt{28})}{\sqrt{28}\pi}, \\
	c_{\widetilde{E}_{1/2,N}^+}\left(0,(4,0,0)\right) &= -\frac{\log(2)}{\pi}.
	\end{align*}
	Taking everything together, we find that
	\[
	G_2\left(j,2j \right) + G_2\left(j,\frac{1+i}{2}+j \right) =  32 L(\chi_{12},1) - 8 L(\chi_{28},1) + 56L(\chi_{60},1).
	\]
 
	  Note that the real quadratic discriminants that appear, namely $\Delta=60, 12, 28$, are exactly the discriminants of the quadratic fields $\Q(\sqrt{(4mm'-r^2)|D|})$ for $r=1,2,3$, respectively.	

 Next, we compute the partially-twisted double trace using Theorem~\ref{theorem evaluation Green's function single twist even}. From the theorem we get, similarly as in the non-twisted case discussed above, that
	\begin{align}\label{eq twisted evaluation}
	\frac{1}{8}G_2\left(j,2j \right) - \frac{1}{8}G_2\left(j,\frac{1+i}{2}+j \right) = 2\pi \CT \left(f_{P \oplus N} \cdot \left[\Theta_P,\widetilde{\vartheta}_{N^-,\chi_D}^+\right]_1 \right),
	\end{align}
	where $f = F_{-1,1,(0,0)}$ as before, and the lattice $P$ and $N$ are as above. Here $\widetilde{\vartheta}_{N^-,\chi_D}$ is a harmonic Maass form of weight $1/2$ which maps to the cusp form $\frac{1}{2}\vartheta_{N^-,\chi_D}$ (see Theorem~\ref{theorem twisted siegel weil preimage}). In this case, $\vartheta_{N^-,\chi_D}$ is precisely the cusp form given in Example~\ref{twisted siegel weil example}. An explicit $\xi$-preimage can be constructed using \cite[Proposition~6.1]{lischwagenscheidt}. It is given by
	\begin{eqnarray*}
	\widetilde{\vartheta}_{N^-,\chi_D} &=&  \widetilde{\theta}_{4,1}^*\left(\e_{(\frac{1}{8},0,0)}+\e_{(\frac{7}{8},0,0)}-\e_{(\frac{5}{8},\frac{1}{2},\frac{1}{2})}-\e_{(\frac{7}{8},\frac{1}{2},\frac{1}{2})}\right) \\
	& &+ \widetilde{\theta}_{4,2}^*\left(\e_{(0,\frac{1}{2},0)}+\e_{(0,0,\frac{1}{2})}-\e_{(\frac{1}{2},\frac{1}{2},0)}-\e_{(\frac{1}{2},0,\frac{1}{2})}\right) \\
	& &+ \widetilde{\theta}_{4,3}^* \left(\e_{(\frac{1}{8},\frac{1}{2},\frac{1}{2})}+\e_{(\frac{7}{8},\frac{1}{2},\frac{1}{2})}-\e_{(\frac{3}{8},0,0)}-\e_{(\frac{5}{8},0,0)}\right),
 	\end{eqnarray*}
	where the weight $1/2$ scalar-valued harmonic Maass forms $\widetilde{\theta}_{4,\rho}^*$ have holomorphic parts
	\begin{align*}
	\widetilde{\theta}_{4,1}^{*,+} &= \frac{1}{4}q^{-\frac{1}{16}} - \frac{7}{4}q^\frac{15}{16} - \frac{21}{4}q^\frac{31}{16} - \frac{43}{4}q^{\frac{47}{16}} - \frac{47}{2}q^{\frac{63}{16}} - 42q^{\frac{79}{16}} - 77q^{\frac{95}{16}} + \dots , \\
	\widetilde{\theta}_{4,2}^{*,+} &= -2q^{\frac{3}{4}} - 6q^{\frac{7}{4}} - 14q^{\frac{11}{4}} - 28q^{\frac{15}{4}} - 54q^{\frac{19}{4}} - 98q^{\frac{23}{4}} + \dots , \\
	\widetilde{\theta}_{4,3}^{*,+} &= -\frac{3}{4}q^{\frac{7}{16}} - \frac{7}{2}q^{\frac{23}{16}} - 7q^{\frac{39}{16}} - \frac{69}{4}q^{\frac{55}{16}} - \frac{119}{4}q^{\frac{71}{16}} - \frac{239}{4}q^{\frac{87}{16}} + \dots .
	\end{align*}
	We computed these expansions numerically in sage, using \cite[Proposition~6.1(2)]{lischwagenscheidt}, with $N = 4$ and $\varepsilon_4 =6\sqrt{8}+17$.
	
	 The constant term $\CT(\cdot)$ appearing on the right-hand side of \eqref{eq twisted evaluation} can be computed similarly as in the non-twisted case above, replacing $\widetilde{E}_{1/2,N}^+$ by $\widetilde{\vartheta}_{N^-,\chi_D}$, so we do not repeat this computation here. Putting in the coefficients of $\widetilde{\theta}_{4,\rho}^{*,+}$ above, we obtain from \eqref{eq twisted evaluation} that
	\[
	G_2\left(j,2j \right) - G_2\left(j,\frac{1+i}{2}+j \right) = -4\pi.
	\]
	
	Combining the evaluation of the non-twisted and partially-twisted double trace yields
	\[
	G_{2}\left(j,2j \right) = 16 L(\chi_{12},1) -4 L(\chi_{28},1) + 28 L(\chi_{60},1)-2\pi.
	\]
	In particular, the ``individual'' value $G_2\left(j,2j \right)$ is not just a rational linear combination of $L$-values.
	
	For $n = 2$ we have $S_{5,L} = \{0\}$, so we can compute in a similar way that
	\[
	G_4\left(j,2j \right) + G_4\left(j,\frac{1+i}{2}+j \right) = -64 \log(2) + 32L(\chi_{12},1)+62L(\chi_{28},1)- 34L(\chi_{60},1),
	\]
	and
	\[
	G_4\left(j,2j \right) - G_4\left(j,\frac{1+i}{2}+j \right) = -\pi.
	\]
	This yields
	\[
	G_4\left(j,2j \right) = -32 \log(2) + 16L(\chi_{12},1)+31L(\chi_{28},1)- 17L(\chi_{60},1) - \frac{\pi}{2}.
	\]

\end{document}